\documentclass[final,3p,times,square]{elsarticle}

\usepackage{amsmath,amssymb}
\usepackage{amsthm}
\usepackage{mathtools}
\usepackage[dvipsnames]{xcolor}
\usepackage{graphicx}
\usepackage{fancyhdr}
\usepackage{courier}
\usepackage{pdflscape}
\usepackage{bm}
\usepackage{tensor}
\usepackage[margin=1cm]{caption}
\usepackage{multirow}
\usepackage{subfiles}
\usepackage{xfrac}
\usepackage{float}
\usepackage{hhline}
\usepackage{pifont}
\usepackage{tabu}
\usepackage{url}
\usepackage{framed}
\usepackage{mathrsfs}
\usepackage{subcaption}
\usepackage{lscape}
\usepackage{longtable}
\usepackage[ruled]{algorithm}
\usepackage{algpseudocode}

\usepackage{tikz}
\usepackage{environ}

\NewEnviron{mybox}[1][]{%
 \noindent
 \begin{tikzpicture}
   \node[minimum width=\linewidth-0.4pt,draw,text width=\linewidth-12pt] (a){\BODY};
   \node[fill=white,yshift=0.5ex] at (a.north) {#1};
 \end{tikzpicture}}

\definecolor{darkblue}{rgb}{0.0,0.0,0.3}

\usepackage[pdftex,
		bookmarks,
		bookmarksopen=true,
		bookmarksnumbered=true,
		pdfauthor={Joseph Benzaken, Rasmus Tamstorf},
		pdftitle={TBD},
		colorlinks,
		linkcolor=darkblue,
		citecolor=darkblue,
		filecolor=black,
		urlcolor=darkblue,
		anchorcolor=darkblue,
		menucolor=black,
		breaklinks=true,
		pageanchor=true,
		plainpages=false,
		pdfpagelabels=true]{hyperref}

\newcommand{\R}{\mathbb R}

\newtheorem{theorem}{Theorem}

\newtheorem{lemma}{Lemma}

\newtheorem{assumption}{Assumption}
\newtheorem{remark}{Remark}

\theoremstyle{definition}

\theoremstyle{remark}

\theoremstyle{definition}

\newcommand{\vertiii}[1]{{\left\vert\kern-0.1ex\left\vert\kern-0.1ex\left\vert #1\right\vert\kern-0.1ex\right\vert\kern-0.1ex\right\vert}}

\usepackage{enumitem}
\usepackage{booktabs}
\usepackage{footnote}
\usepackage[normalem]{ulem} 

\usepackage{tikz}
\usetikzlibrary{patterns,calc,plotmarks}
\usetikzlibrary{decorations.pathmorphing}
\usetikzlibrary{arrows.meta}
\usepackage{pgfplots}
\usepackage{pgfplotstable}
\usepgfplotslibrary{colorbrewer}

\usepackage{amsmath}
\usepackage{amssymb}
\usepackage{amsfonts}
\usepackage{adjustbox}
\newcommand{\vvvertiii}[1]{{\left\vert\kern-0.15ex\left\vert\kern-0.15ex\left\vert #1 \right\vert\kern-0.15ex\right\vert\kern-0.15ex\right\vert}}

\usepackage{wrapfig}

\usepackage{appendix}

\usepackage[english]{babel}
\addto\captionsenglish{}

\newcommand{\TWODSS}{\raisebox{0pt}{\includegraphics{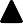}}}
\newcommand{\THREEDSS}{\raisebox{0pt}{\includegraphics{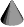}}}
\newcommand{\CLAMPED}{\raisebox{0pt}{\includegraphics{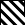}}}
\newcommand{\FREE}{\raisebox{0pt}{\includegraphics{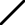}}}
\newcommand{\SYM}{\raisebox{0pt}{\includegraphics{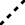}}}

\newcommand{\pVariable}{\xi}
\newcommand{\FFF}{a}
\newcommand{\SFF}{b}
\newcommand{\TFF}{c}

\newcommand{\constitutive}{C}
\newcommand{\midsurfDispTrial}{u}
\newcommand{\midsurfDispTest}{v}
\newcommand{\midsurfRot}{\theta}
\newcommand{\shellDisp}{U}
\newcommand{\bdyNormal}{n}
\newcommand{\bdyTangent}{t}
\newcommand{\memStrain}{\alpha}
\newcommand{\bendStrain}{\beta}
\newcommand{\memStress}{A}
\newcommand{\bendStress}{B}
\newcommand{\GLStrain}{\varepsilon}

\newcommand{\traction}{\tau}
\newcommand{\moment}{B}
\newcommand{\thickness}{\zeta}
\newcommand{\ersatz}{\textup{T}}
\newcommand{\force}{\textup{S}}

\newcommand{\KLS}{S}
\newcommand{\CSconstant}{\gamma}
\newcommand{\Projector}{P}
\newcommand{\Sym}{\text{Sym}}

\newcommand{\parametric}[1]{\hat{#1}}
\newcommand{\undef}[1]{#1}
\newcommand{\fullTens}[1]{ {\bf #1}}
\newcommand{\surfVec}[1]{\uline{#1}}
\newcommand{\surfTens}[1]{\uuline{#1}}
\newcommand{\outOfPlane}[1]{#1_3}

\newcommand{\fullNorm}[1]{\left\|\, #1 \,\right\|}

\newcommand{\applied}[1]{\hat{#1}}
\newcommand{\cTrace}[1]{C_{\textup{tr} #1}}

\newcommand{\cPen}[1]{C_{\textup{pen} #1}}
\newcommand{\cornerSet}[1]{\chi_{#1}}

\newcommand{\aVec}{\undef{\fullTens{\FFF}}}
\newcommand{\aMet}{\undef{\FFF}}

\newcommand{\midsurfMap}{\fullTens{x}}
\newcommand{\shellMap}{\fullTens{X}}
\newcommand{\bCurv}{\undef{\SFF}}
\newcommand{\cCurv}{\undef{\TFF}}
\newcommand{\christoffel}{\undef{\Gamma}}
\newcommand{\shellBody}{\mathcal{B}}
\newcommand{\gVec}{{\bf G}}
\newcommand{\covTrans}{\undef{\Lambda}}

\newcommand{\cartBasis}{{\bf e}}
\newcommand{\parDomain}{\parametric{\Omega}}

\newcommand{\physDomain}{\Omega}
\newcommand{\physBoundary}{\Gamma}

\newcommand{\forcing}{\textup{f}}
\newcommand{\linFunctional}{f}
\newcommand{\shear}{\outOfPlane{\applied{\traction}}}

\numberwithin{equation}{section}

\journal{}

\begin{document}

\begin{frontmatter}

\title{Weak Boundary Condition Enforcement for Linear Kirchhoff-Love Shells:\\ Formulation, Error Analysis, and Verification}

\author[wdas]{Joseph Benzaken\corref{mycorrespondingauthor}}
\cortext[mycorrespondingauthor]{Corresponding author}
\ead{joseph.benzaken@disneyanimation.com}

\author[cuboulder]{John A. Evans}
\author[cubouldermath]{Stephen McCormick}
\author[wdas]{Rasmus Tamstorf}

\address[wdas]{Walt Disney Animation Studios, Burbank, United States}
\address[cuboulder]{Ann and H.J. Smead Department of Aerospace Engineering Sciences, University of Colorado at Boulder, United States}
\address[cubouldermath]{Department of Applied Mathematics, University of Colorado at Boulder, United States}

\begin{abstract}

Stable and accurate modeling of thin shells requires proper enforcement of all types of boundary conditions.  Unfortunately, for Kirchhoff-Love shells, strong enforcement of Dirichlet boundary conditions is difficult because both functional and derivative boundary conditions must be applied.  A popular alternative is to employ Nitsche's method to weakly enforce all boundary conditions.  However, while many Nitsche-based formulations have been proposed in the literature, they lack comprehensive error analyses and verifications.  In fact, existing formulations are variationally inconsistent and yield sub-optimal convergence rates when used with common boundary condition specifications.  In this paper, we present a novel Nitsche-based formulation for the linear Kirchhoff-Love shell that is provably stable and optimally convergent for general sets of admissible boundary conditions.  To arrive at our formulation, we first present a framework for constructing Nitsche's method for any abstract variational constrained minimization problem.  We then apply this framework to the linear Kirchhoff-Love shell and, for the particular case of NURBS-based isogeometric analysis, we prove that the resulting formulation yields optimal convergence rates in both the shell energy norm and the standard $L^2$-norm.  To arrive at this formulation, we derive the Euler-Lagrange equations for general sets of admissible boundary conditions and show that the Euler-Lagrange boundary conditions typically presented in the literature is incorrect.  We verify our formulation by manufacturing solutions for a new shell obstacle course that encompasses flat, parabolic, hyperbolic, and elliptic geometric configurations with a variety of common boundary condition specifications.  These manufactured solutions allow us to robustly measure the error across the entire shell in contrast with current best practices where displacement and stress errors are only measured at specific locations.  We use NURBS discretizations to represent the shell geometry and show optimal convergence rates in both the shell energy norm and the standard $L^2$-norm with varying polynomial degrees for all of the problems in the obstacle course.

\end{abstract}


\end{frontmatter}



\section{Introduction}
\label{sec:introduction}

Thin shells are a key component of engineered structures such as aircraft fuselages, architectural domes, and textile fabrics.  It is well known that thin shells can be modeled using the Kirchhoff-Love kinematical assumption, but this assumption yields a vector-valued fourth-order partial differential equation (PDE) to be solved.  Unfortunately, $C^0$-continuous finite element methods cannot be directly applied to the Galerkin approximation of fourth-order PDEs, so thick shell formulations based on Reissner-Mindlin kinematics are typically preferred in finite element shell analysis \cite{bischoff2018models}.  However, in recent years, isogeometric analysis \cite{Hughes2005} has rekindled interest in Kirchhoff-Love shell formulations.  Splines exhibit the requisite $C^1$-continuity to directly solve fourth-order PDEs, and isogeometric discretizations based on Non-Uniform Rational B-splines (NURBS) \cite{kiendl2009isogeometric,echter2013hierarchic}, multi-patch NURBS \cite{kiendl2010bending}, hierarchical NURBS \cite{coradellohierarchically}, T-splines \cite{bazilevs2012isogeometric,casquero2017arbitrary}, PHT- and RHT-splines \cite{nguyen2011rotation,nguyen2017isogeometric}, and subdivision surfaces \cite{Cirak2000,cirak2001fully} have been successfully applied to the Galerkin approximation of the Kirchhoff-Love shell equations.  Isogeometric discretizations have also been combined with immersed and embedded methods to treat trimmed NURBS \cite{guo2015weak,guo2017parameter,guo2018variationally} and implicitly defined geometries \cite{schollhammer2019kirchhoff}.

Stable and accurate numerical modeling of Kirchhoff-Love shells requires proper enforcement of all types of boundary conditions.  In a classical Galerkin method, Dirichlet boundary conditions are enforced strongly, but this is a difficult task for Kirchhoff-Love shells because both displacement (functional) and normal rotation (derivative) boundary conditions must be applied\footnote{If a given discretization method interpolates both function values and derivatives at specified points in the domain, then both homogeneous and non-homogeneous displacement and normal rotation boundary conditions may be easily enforced in a strong manner.  Unfortunately, state-of-the-art isogeometric discretization methods do not interpolate either function values or derivatives, so strong enforcement of non-homogeneous displacement and normal boundary conditions is much more difficult using these methods.}.  This has inspired the development of weak boundary condition enforcement strategies, the most common approach of which is the classical penalty method wherein penalty terms are added to the underlying variational formulation \cite{lei2015c0,breitenberger2015analysis,duong2017new,herrema2019penalty}.  However, the penalty method is quite inaccurate unless parameters associated with the penalty terms are chosen sufficiently large, but large penalty parameters in turn yield an overly stiff, ill-conditioned linear system after discretization.  A second common approach to weak boundary condition enforcement is to introduce Lagrange multiplier fields \cite{apostolatos2015domain,schuss2019multi}.  The primary disadvantages of this approach are that it leads to a discrete saddle-point problem and stability can only be ensured if the approximation spaces for the primal and Lagrange multiplier fields satisfy the Babu\v{s}ka-Brezzi inf-sup condition \cite{brezzi2012mixed}.

Nitsche's method is an alternative approach for the weak enforcement of Dirichlet boundary conditions.  Nitsche's method was first proposed in 1971 \cite{Nitsche1971}, but it did not grow in popularity until the recent emergence of meshless \cite{FernandezMendez2004}, extended \cite{annavarapu2012robust,hansbo2002unfitted}, immersed \cite{kamensky2015immersogeometric,ruess2013weakly,schillinger2012isogeometric}, and isogeometric \cite{embar2010imposing,Apostolatos2014,nguyen2014nitsche,ruess2014weak,harari2015unified,guo2015nitsche} finite element methods.  For these emerging finite element methods, strong enforcement of boundary and interface conditions is quite difficult due to the non-interpolatory nature of the primal field approximation space along domain boundaries and interfaces.  Nitsche's method involves the addition of consistency, symmetry, and penalty terms to the underlying variational formulation.  The design of the consistency and symmetry terms is guided by the Euler-Lagrange equations for the problem of interest, while the design of the penalty terms is guided by trace inequalities.  Nitsche's method is variationally consistent and stable by construction, and it provides optimal convergence rates.  Moreover, for self-adjoint elliptic PDEs, Nitsche's method yields a relatively well-conditioned symmetric positive-definite linear system after discretization.

Nitsche's method is particularly appealing for Kirchhoff-Love shells since it can be used to enforce both displacement and normal rotation boundary conditions.  It comes as no surprise, then, that a number of Nitsche-based formulations have been proposed in the literature for Kirchhoff-Love shells, most commonly for isogeometric finite element shell analysis \cite{nguyen2017isogeometric,guo2015weak,guo2018variationally,guo2015nitsche}.  However, a comprehensive error analysis or verification has not yet been conducted for any of these formulations.  In fact, as we demonstrate later in this paper, the formulations proposed in \cite{guo2015weak,guo2015nitsche} for the linear Kirchhoff-Love shell are variationally inconsistent and provide sub-optimal convergence rates when used with common boundary condition specifications. This variational inconsistency is due to the fact that existing Nitsche-based formulations are based upon Euler-Lagrange equations typically presented in the literature, and these equations are incorrect for general sets of admissible boundary conditions.  In particular, the so-called ersatz force that appears in one of the Euler-Lagrange boundary conditions is incorrect.  We believe this fact has been missed previously in the literature as state-of-the-art verification tests, such as the so-called ``shell obstacle course'' \cite{Belytschko1985}, are unable to assess order of accuracy.  Instead, these verification tests only gauge convergence of displacement or stress fields to reference values at particular spatial locations.

In this paper, we present a new Nitsche-based formulation for the linear Kirchhoff-Love shell that is provably stable and optimally convergent for general sets of admissible boundary conditions.  To arrive at our formulation, we first present a framework for constructing Nitsche's method for an abstract variational constrained minimization problem admitting a generalized Green's identity.  Our construction follows that of \cite{stenberg1995some} in that we first construct a stabilized Lagrange multiplier method for the abstract variational problem and then statically condense the Lagrange multiplier field.  With the guidance of generalized trace and Cauchy-Schwarz inequalities, we are able to establish conditions under which the resulting method is both stable and convergent.  We then apply this abstract framework to the construction of a stable and convergent Nitsche-based formulation for the linear Kirchhoff-Love shell.  The resulting formulation has not appeared previously in the literature.  Most notably, it involves consistency and symmetry terms associated with corner forces and penalty terms associated with corner displacement boundary conditions, similar to the Nitsche-based formulation proposed in \cite{harari2012embedded} for the Kirchhoff-Love plate.  To arrive at our formulation, we derive the Euler-Lagrange equations for general sets of admissible boundary conditions and discover, as previously noted, that the equations typically presented in the literature are incorrect.  For a NURBS-based isogeometric discretization of the linear Kirchhoff-Love shell, we establish \textit{a priori} error estimates for the $H^2$-, $H^1$-, and $L^2$-norms of the error in the displacement field, and we confirm these estimates using a new suite of manufactured solutions that covers a wide variety of geometric configurations and boundary conditions.  To the best of our knowledge, this suite is the first comprehensive verification test bed capable of assessing convergence rates for Kirchhoff-Love shell discretizations, and we are aware of only one manufactured solution test case for the linear Kirchhoff-Love shell in the literature \cite{gfrerer2018code}.

While the focus of this paper is weak enforcement of boundary conditions for the linear Kirchhoff-Love shell, the abstract framework presented here can be employed to construct Nitsche-based formulations for other linear problems arising from energy minimization.  Moreover, given the close connection between the method of stabilized Lagrange multipliers, Nitsche's method, and the symmetric interior penalty Galerkin method \cite{arnold1982interior}, the framework can also be used to construct discontinuous Galerkin \cite{hansbo2002discontinuous,Noels2008} and continuous/discontinuous Galerkin methods \cite{engel2002continuous} for membranes, plates, shells, and other problems of interest.  For example, the Nitsche-based Kirchhoff-Love formulation presented here can be easily modified to weakly enforce continuity of displacement and normal rotation along patch interfaces for non-conforming multi-patch NURBS geometries and along trimming curves for trimmed NURBS geometries.  Finally, while the framework presented in this paper is strictly for linear problems arising from energy minimization, it is easily extended to nonlinear and nonsymmetric problems, including those involving contact, damage, and fracture.  In fact, the only reason we consider linear problems arising from energy minimization in this paper is the simplicity in establishing stability and convergence results for these problems in an abstract setting, and we plan to extend our formulation to nonlinear Kirchhoff-Love shells in future work \cite{kiendl2015isogeometric,tepole2015isogeometric}.

The remainder of this paper proceeds as follows. In Section~\ref{sec:Nitsche}, Nitsche's method is constructed for an abstract variational constrained minimization problem, and this framework is applied to the construction of a Nitsche-based formulation for the linear Kirchhoff-Love shell in Section~\ref{sec:KL_Shell}.  In Section~\ref{sec:apriori}, \textit{a priori} error estimates for the $H^2$-, $H^1$-, and $L^2$-norms of the error in the displacement field are established for NURBS-based isogeometric discretizations of the linear Kirchhoff-Love shell problem, and these estimates are confirmed using a suite of manufactured solutions in Section~\ref{sec:num_results}.  Finally, concluding remarks and future research directions are presented in Section~\ref{sec:conclusion}.


\section{Nitsche's Method for an Abstract Variational Constrained Minimization Problem}
\label{sec:Nitsche}

This section develops an abstract framework and theory that is applied later to the linear Kirchhoff-Love shell.  To this end, let $\mathcal{V}$ and $\mathcal{Q}$ be two Hilbert spaces with respective inner products $(\cdot,\cdot)_\mathcal{V}$ and $(\cdot,\cdot)_\mathcal{Q}$ and induced norms $\| \cdot \|_\mathcal{V} = (\cdot,\cdot)^{1/2}_\mathcal{V}$ and $\| \cdot \|_\mathcal{Q} = (\cdot,\cdot)^{1/2}_\mathcal{Q}$. We also use the notation $| \cdot |$ to refer to the absolute value for scalar quantities and the Euclidean norm for vector quantities. Let $\mathcal{V}^*$ and $\mathcal{Q}^*$ be the respective dual spaces of $\mathcal{V}$ and $\mathcal{Q}$, and let ${}_{\mathcal{V}^*}\langle \cdot, \cdot \rangle_{\mathcal{V}}$ be the duality pairing between $\mathcal{V}$ and its dual and ${}_{\mathcal{Q}^*}\langle \cdot, \cdot \rangle_{\mathcal{Q}}$ the duality pairing between $\mathcal{Q}$ and its dual.  Let $\mathcal{T}: \mathcal{V} \rightarrow \mathcal{Q}$ be a bounded, surjective linear map, and given $g \in \mathcal{Q}$, define
\begin{equation*}
\mathcal{V}_g := \left\{ v \in \mathcal{V}: \mathcal{T}v = g \right\}.
\end{equation*}
Finally, let $a : \mathcal{V} \times \mathcal{V} \rightarrow \mathbb{R}$ be a bounded, symmetric, positive semi-definite bilinear form satisfying the following coercivity condition on the kernel of $\mathcal{T}$:
\begin{equation*}
a(v,v) \geq C \| v \|^2_\mathcal{V} \hspace{10pt} \forall v \in \mathcal{V}_0
\label{eqn:coerA}
\end{equation*}
for some constant $C \in \mathbb{R}_+$.

\begin{remark}
In the context of structural mechanics, $\mathcal{V}$ is the space of admissible displacements free of boundary conditions and $\mathcal{Q}$ is the space of admissible essential boundary conditions (e.g., displacement and rotation boundary conditions in the context of a Kirchhoff-Love shell).  The map $\mathcal{T}$ then gives the trace of the displacement field (e.g., the displacement and normal rotation in the context of a Kirchhoff-Love shell) along portions of the boundary where essential boundary conditions are being enforced.  Consequently, $\mathcal{V}_g$ denotes the space of admissible displacements satisfying prescribed essential boundary conditions and $\mathcal{V}_0$ denotes the corresponding space of virtual displacements.
\end{remark}

We are interested in the following constrained minimization problem:
$$
(M) \left\{ \hspace{5pt}
\parbox{6.00in}{
\noindent Given $f \in \mathcal{V}^*$ and $g \in \mathcal{Q}$, find $u \in \mathcal{V}_g$ that minimizes the total energy
\begin{eqnarray*}
E_{\textup{total}}(u) = E_{\textup{int}}(u) + E_{\textup{ext}}(u)
\end{eqnarray*}
where the \textbf{\textit{internal energy}} is defined by
\begin{eqnarray*}
E_{\textup{int}}(u) = \frac{1}{2} a(u,u)
\end{eqnarray*}
and the \textbf{\textit{external energy}} is defined by
\begin{eqnarray*}
E_{\textup{ext}}(u) = -{}_{\mathcal{V}^*}\langle f, u \rangle_{\mathcal{V}}.
\end{eqnarray*}
}
\right.
$$
Note that the G\^{a}teaux derivative of the total energy associated with a solution is zero for any variation $\delta u \in \mathcal{V}_0$.  Consequently, Problem $(M)$ is equivalent to the following variational problem:
$$
(V) \left\{ \hspace{5pt}
\parbox{6.0in}{
\noindent Given $f \in \mathcal{V}^*$ and $g \in \mathcal{Q}$, find $u \in \mathcal{V}_g$ such that
\begin{eqnarray}
a(u, \delta u) = {}_{\mathcal{V}^*}\langle f, \delta u \rangle_{\mathcal{V}}
\label{eq:virtual_work}
\end{eqnarray}
for every $\delta u \in \mathcal{V}_0$.
}
\right.
$$
The Lax-Milgram theorem guarantees that Problem $(V)$ has a unique solution $u \in \mathcal{V}$ that depends continuously on the input data $f \in \mathcal{V}^*$ and $g \in \mathcal{Q}$ \cite{EvansPDEs}.

\begin{remark}
The quantity $E_{\textup{int}}(u)$ denotes the internal energy of a system displaced by $u$ due to internal stresses and strains, while the quantity $E_{\textup{ext}}(u)$ denotes the external energy of the same system due to external forces, tractions, and moments.  The quantity $a(u,\delta u)$ represents the virtual work due to internal stresses as the system undergoes a virtual displacement $\delta u$, while the quantity ${}_{\mathcal{V}^*}\langle f, \delta u \rangle_{\mathcal{V}}$ represents the virtual work done to the system by external forces, tractions, and moments. Therefore, \eqref{eq:virtual_work} is often referred to as the principle of virtual work since it states that in equilibrium the external and internal virtual work must be in balance.  The kernel of the bilinear form $a(\cdot,\cdot)$ consists of rigid body modes.
\end{remark}

Let $\mathcal{V}_h \subset \mathcal{V}$ be a finite-dimensional approximation space and $\mathcal{V}_{g,h} = \mathcal{V}_h \cap \mathcal{V}_g$ for every $g \in \mathcal{Q}$.  The Bubnov-Galerkin approximation of Problem $(V)$ then reads as follows:
$$
(V_h) \left\{ \hspace{5pt}
\parbox{6.00in}{
\noindent Given $f \in \mathcal{V}^*$ and $g \in \mathcal{Q}$, find $u_h \in \mathcal{V}_{g,h}$ such that
\begin{eqnarray*}
a(u_h, \delta u_h) = {}_{\mathcal{V}^*}\langle f, \delta u_h \rangle_{\mathcal{V}}
\end{eqnarray*}
for every $\delta u_h \in \mathcal{V}_{0,h}$.
}
\right.
$$
The Lax-Milgram theorem guarantees that Problem $(V_h)$ has a unique solution $u_h \in \mathcal{V}_h$ that depends continuously on the input data $f \in \mathcal{V}^*$ and $g \in \mathcal{Q}$, and it is also easily shown that the solution to Problem $(V_h)$ best approximates the solution to Problem $(V)$ with respect to the norm induced by the bilinear form $a(\cdot,\cdot)$.  The difficulty associated with Problem $(V_h)$ is the need for strong enforcement of the condition $\mathcal{T}u_h = g$. This is straightforward for simple approximation spaces (e.g., piecewise linear finite elements), simple applications (e.g., linear elasticity), and simple constraints (e.g., displacement boundary conditions). However, for complex approximation spaces (e.g., B-splines and subdivision surfaces), complex applications (e.g., Kirchhoff-Love shells), and complex constraints (e.g., rotation boundary conditions), this condition is much more difficult to enforce.  Alternatively, we may turn to the method of Lagrange multipliers for weak enforcement of $\mathcal{T}u_h = g$.

It is well known that the solution of Problem $(M)$ may be found by solving the following saddle point problem:
$$
(\mathit{SP}) \left\{ \hspace{5pt}
\parbox{5.90in}{
\noindent Given $f \in \mathcal{V}^*$ and $g \in \mathcal{Q}$, find the saddle point $(u,\lambda) \in \mathcal{V} \times \mathcal{Q}^*$ of the Lagrangian
\begin{eqnarray*}
\mathcal{L}(u,\lambda) = E_{\textup{total}}(u) + {}_{\mathcal{Q}^*}\langle \lambda, \mathcal{T}u - g \rangle_{\mathcal{Q}}.
\end{eqnarray*}
}
\right.
$$
Note that the G\^{a}teaux derivative of the Lagrangian at the solution is zero for any direction $(\delta u, \delta \lambda) \in \mathcal{V} \times \mathcal{Q}^*$.  Consequently, Problem $(\mathit{SP)}$ is equivalent to the following variational problem:
$$
(L) \left\{ \hspace{5pt}
\parbox{6.00in}{
\noindent Given $f \in \mathcal{V}^*$ and $g \in \mathcal{Q}$, find $(u,\lambda) \in \mathcal{V} \times \mathcal{Q}^*$ such that
\begin{eqnarray}
a(u, \delta u) + {}_{\mathcal{Q}^*}\langle \lambda, \mathcal{T}\delta u \rangle_{\mathcal{Q}} + {}_{\mathcal{Q}^*}\langle \delta \lambda, \mathcal{T} u \rangle_{\mathcal{Q}} = {}_{\mathcal{V}^*}\langle f, \delta u \rangle_{\mathcal{V}} + {}_{\mathcal{Q}^*}\langle \delta \lambda, g \rangle_{\mathcal{Q}} \label{eq:Lagrange}
\end{eqnarray}
for every $(\delta u,\delta \lambda) \in \mathcal{V} \times \mathcal{Q}^*$.
}
\right.
$$
By the surjectivity and boundedness of $\mathcal{T}$, Problem $(L)$ has a unique solution $(u,\lambda) \in \mathcal{V} \times \mathcal{Q}^*$ that depends continuously on the input data $f \in \mathcal{V}^*$ and $g \in \mathcal{Q}$.  The variable $\lambda$ is commonly referred to as the Lagrange multiplier associated with the constraint $\mathcal{T}u = g$.

\begin{remark}
  In the context of structural mechanics, $\lambda$ comprises the reaction or constraint forces, tractions, and moments that result from application of essential boundary conditions.
\end{remark}

The discretization of Problem $(L)$ requires approximations of both $\mathcal{V}$ and $\mathcal{Q}^*$.  To this end, let $\mathcal{V}_h \subset \mathcal{V}$ and $\mathcal{Q}^*_h \subset \mathcal{Q}^*$ be two finite-dimensional approximation spaces.  The Galerkin approximation of Problem $(L)$ then reads as follows:
$$
(L_h) \left\{ \hspace{5pt}
\parbox{6.00in}{
\noindent Given $f \in \mathcal{V}^*$ and $g \in \mathcal{Q}$, find $(u_h,\lambda_h) \in \mathcal{V}_h \times \mathcal{Q}^*_h$ such that:
\begin{eqnarray*}
a(u_h, \delta u_h) + {}_{\mathcal{Q}^*}\langle \lambda_h, \mathcal{T}\delta u_h \rangle_{\mathcal{Q}} + {}_{\mathcal{Q}^*}\langle \delta \lambda_h, \mathcal{T} u_h \rangle_{\mathcal{Q}} = {}_{\mathcal{V}^*}\langle f, \delta u_h \rangle_{\mathcal{V}} + {}_{\mathcal{Q}^*}\langle \delta \lambda_h, g \rangle_{\mathcal{Q}}
\end{eqnarray*}
for every $(\delta u_h,\delta \lambda_h) \in \mathcal{V}_h \times \mathcal{Q}_h^*$.
}
\right.
$$
The advantage of the formulation given by Problem ($L_h$), which is commonly referred to as the \textit{\textbf{method of Lagrange multipliers}}, over the formulation given by Problem ($V_h$), which is commonly referred to as \textit{\textbf{Galerkin's method}}, is that the condition $\mathcal{T} u_h = g$ need not be directly embedded into the solution space $\mathcal{V}_h$.  However, the disadvantage of the method of Lagrange multipliers is that two approximation spaces $\mathcal{V}_h \subset \mathcal{V}$ and $\mathcal{Q}^*_h \subset \mathcal{Q}^*$ are needed and, moreover, they must be chosen intelligently in order to arrive at a stable and convergent method.  Namely, the two approximation spaces $\mathcal{V}_h \subset \mathcal{V}$ and $\mathcal{Q}^*_h \subset \mathcal{Q}^*$ must satisfy the so-called \textit{\textbf{Babu\v{s}ka-Brezzi inf-sup condition}} \cite{babuvska1973finite}.  For simple approximation spaces and applications, selecting inf-sup stable approximation spaces is reasonably straightforward, but for complex problems, it is more difficult. An alternative is to use stabilization to bypass the inf-sup condition entirely.  This is a rather elegant solution first proposed by Franca and Hughes for enforcing incompressibility for Stokes flow \cite{hughes1986new} and later proposed by Barbosa and Hughes for enforcing essential boundary conditions for contact problems \cite{barbosa1991finite}.

To this end, we make the following assumption:

\begin{assumption}
  \label{assumption1}
  There exists a dense subspace $\tilde{\mathcal{V}} \subset \mathcal{V}$ and linear maps $\mathcal{L} : \tilde{\mathcal{V}} \rightarrow \mathcal{V}^*$ and $\mathcal{B}: \tilde{\mathcal{V}} \rightarrow \mathcal{Q}^*$ such that the following \textit{\textbf{generalized Green's identity}} holds:
  \begin{equation}
      a(w,v) = {}_{\mathcal{V}^*}\langle \mathcal{L}w, v \rangle_{\mathcal{V}} + {}_{\mathcal{Q}^*}\langle \mathcal{B}w, \mathcal{T}v \rangle_{\mathcal{Q}}
      \label{eq:int_by_parts}
    \end{equation}
    for all $w \in \tilde{\mathcal{V}}$ and $v \in \mathcal{V}$, and the solution $u$ of Problem $(M)$ satisfies $\mathcal{L} u = f$ whenever $f \in \mathcal{V}^*$ and $g \in \mathcal{Q}$ are such that $u \in \tilde{\mathcal{V}}$.
\end{assumption}

\begin{remark}
  In the context of structural mechanics, \eqref{eq:int_by_parts} results from the application of integration by parts to the original variational formulation in order to arrive at the Euler-Lagrange equations. Thus, in this context, the map $\mathcal{L}$ encodes the differential-algebraic operators associated with the governing system of PDEs in their strong form as well as those associated with the natural boundary conditions. The map $\mathcal{B}$ encodes the energetically conjugate essential boundary conditions that result from the application of integration by parts.  In the context of linear elasticity, the quantity $\mathcal{B}u$, where $u$ is the solution to Problem $(M)$, is the traction field along portions of the boundary where essential boundary conditions are being enforced.  In the context of Kirchhoff-Love shells, the quantity $\mathcal{B}u$ consists of shears, moments, and corner forces.  Note that in order for the Euler-Lagrange equations to hold, the solution to Problem $(M)$ must be sufficiently smooth.  This is why we introduced an additional subspace $\tilde{\mathcal{V}} \subset \mathcal{V}$, one for which \eqref{eq:int_by_parts} holds. Generally, $\tilde{\mathcal{V}}$ is significantly more regular than $\mathcal{V}$, raising concerns about the numerical practicality of a method requiring such level of regularity. However, as discussed below, we can extend the necessary operators to an enlarged, additive space between $\tilde{\mathcal{V}}$ and a subspace of $\mathcal{V}$. This permits discretizations of far less regularity and should alleviate these initial concerns.
  \label{remark:V_tilde}
\end{remark}

\begin{remark}
  To distinguish the Green's identity given in \eqref{eq:int_by_parts} from Green's first, second, and third identities, we have used the clarifier ``generalized''.  For the Poisson problem subject to homogeneous Dirichlet boundary conditions, the Green's identity given in \eqref{eq:int_by_parts} coincides with Green's first identity.
  \end{remark}

With Assumption~\ref{assumption1} in hand, we establish an important result giving an expression for the Lagrange multiplier $\lambda \in \mathcal{Q}^*$, provided the solution $u$ of Problem $(M)$ satisfies $u \in \tilde{\mathcal{V}}$.

\begin{theorem}
  \label{theorem:L_is_Bu}
  Suppose that Assumption~\ref{assumption1} holds and the solution $u$ of Problem $(M)$ satisfies $u \in \tilde{\mathcal{V}}$.  Then the solution $(u,\lambda)$ of Problem $(L)$ satisfies $\lambda = - \mathcal{B}u$.

  \begin{proof}
    By \eqref{eq:Lagrange}, we know that
    \begin{align*}
        0 & = a(u, \delta u) + {}_{\mathcal{Q}^*}\langle \lambda, \mathcal{T}\delta u \rangle_{\mathcal{Q}} - {}_{\mathcal{V}^*}\langle f, \delta u \rangle_{\mathcal{V}} \nonumber \\
        & = {}_{\mathcal{V}^*}\langle \mathcal{L}u - f, \delta u \rangle_{\mathcal{V}} + {}_{\mathcal{Q}^*}\langle \mathcal{B}u + \lambda, \mathcal{T}\delta u \rangle_{\mathcal{Q}} \\
        & = {}_{\mathcal{Q}^*}\langle \mathcal{B}u + \lambda, \mathcal{T}\delta u \rangle_{\mathcal{Q}} \nonumber
    \end{align*}
    for all $\delta u \in \mathcal{V}$.  Since $\mathcal{T}$ is surjective, it follows that $\lambda = - \mathcal{B}u$.
  \end{proof}
\end{theorem}

\begin{remark}
  In the context of structural mechanics, Theorem 1 comes as no surprise.  It says that the reaction forces, tractions, and moments that result from the application of essential boundary conditions are balanced by the traction field in the context of linear elasticity and shears, moments, and corner forces in the context of Kirchhoff-Love shells along the portions of the boundary where essential boundary conditions are being enforced.  Thus, Theorem 1 is simply a re-statement of Newton's third law: For every action, there is an equal and opposite reaction.
\end{remark}

Given Theorem~\ref{theorem:L_is_Bu}, we can now construct a \textit{\textbf{stabilized Lagrange multiplier method}}.  First, let $\epsilon: \textup{dom}(\epsilon) \subseteq \mathcal{Q}^* \rightarrow \mathcal{Q}$ be a densely defined, positive, surjective, self-adjoint linear map.  Note that since $\epsilon$ is linear, positive, and surjective, it is also invertible.  We assume that $\mathcal{Q}^*_h \subset \textup{dom}(\epsilon)$.  Provided that the domain of definition of the operator $\mathcal{B}: \tilde{\mathcal{V}} \rightarrow \mathcal{Q}^*$ can be extended to the enlarged space $\tilde{\mathcal{V}} + \mathcal{V}_h$, we also assume that $\left\{ \mathcal{B}v: v \in \tilde{\mathcal{V}} + \mathcal{V}_h \right\} \subset \textup{dom}(\epsilon)$.  We associate with $\epsilon$ a symmetric, positive-definite \textit{\textbf{stabilization bilinear form}} $S: \textup{dom}(\epsilon) \times \textup{dom}(\epsilon) \rightarrow \mathbb{R}$ satisfying
\begin{equation*}
S(\mu,\xi) = {}_{\mathcal{Q}^*}\langle \xi, \epsilon \mu \rangle_{\mathcal{Q}}
\end{equation*}
for all $\mu, \xi \in \textup{dom}(\epsilon)$.  If the solution $u$ of Problem $(M)$ satisfies $u \in \tilde{\mathcal{V}}$, then by Theorem~\ref{theorem:L_is_Bu}, since $\lambda + \mathcal{B}u = 0$, the solution $(u,\lambda) \in \mathcal{V} \times \mathcal{Q}^*$ of Problem $(L)$ satisfies
\begin{equation*}
S(\lambda + \mathcal{B}u,\delta \lambda) = 0
\end{equation*}
for all $\delta \lambda \in \textup{dom}(\epsilon)$.  This then inspires the following stabilized Lagrange multiplier method:

$$
(L^s_h) \left\{ \hspace{5pt}
\parbox{6.00in}{
\noindent Given $f \in \mathcal{V}^*$ and $g \in \mathcal{Q}$, find $(u_h,\lambda_h) \in \mathcal{V}_h \times \mathcal{Q}^*_h$ such that
\begin{equation}
B_h\left( (u_h, \lambda_h), (\delta u_h, \delta \lambda_h) \right) = {}_{\mathcal{V}^*}\langle f, \delta u_h \rangle_{\mathcal{V}} + {}_{\mathcal{Q}^*}\langle \delta \lambda_h, g \rangle_{\mathcal{Q}}
\label{eq:stabilized}
\end{equation}
for every $\left( \delta u_h, \delta \lambda_h \right) \in \mathcal{V}_h  \times \mathcal{Q}^*_h$, where $B_h: \left( \mathcal{V}_h \times \mathcal{Q}^*_h \right) \times \left( \mathcal{V}_h \times \mathcal{Q}^*_h \right) \rightarrow \mathbb{R}$ is the bilinear form defined by
\begin{equation*}
B_h\left( (w_h, \theta_h), (v_h, \mu_h) \right) = a(w_h, v_h) + {}_{\mathcal{Q}^*}\langle \theta_h, \mathcal{T}v_h \rangle_{\mathcal{Q}} + {}_{\mathcal{Q}^*}\langle \mu_h, \mathcal{T} w_h \rangle_{\mathcal{Q}} - S(\theta_h + \mathcal{B} w_h,\mu_h + \mathcal{B} v_h)
\end{equation*}
for every $(w_h,\theta_h), (v_h,\mu_h) \in \mathcal{V}_h \times \mathcal{Q}_h^*$.
}
\right.
$$
The stabilization bilinear form acts to improve the stability of the method of Lagrange multipliers.  In fact, provided that $\epsilon$ is chosen appropriately, the stabilized Lagrange multiplier method can restore stability for an otherwise unstable choice of $\mathcal{V}_h$ and $\mathcal{Q}^*_h$ \cite{hughes1986new,barbosa1991finite}.

Nitsche's method corresponds to the formal selection of $\mathcal{Q}_h^*$ as the entire space $\textup{dom}(\epsilon)$ rather than a finite dimensional subspace in the stabilized Lagrange multiplier method.  This selection generally yields an infinite-dimensional linear system since $\textup{dom}(\epsilon)$ is dense in $\mathcal{Q}^*$.  However, the Lagrange multiplier variable $\lambda_h$ may be statically condensed from the system, resulting in a finite-dimensional linear system for the primal variable $u_h$.  To see this, take $\delta u_h = 0$ in \eqref{eq:stabilized} to obtain
\begin{equation*}
{}_{\mathcal{Q}^*}\langle \delta \lambda_h, \mathcal{T} u_h - g - \epsilon \left( \lambda_h + \mathcal{B} u_h \right) \rangle_{\mathcal{Q}} = 0
\end{equation*}
for all $\delta \lambda_h \in \mathcal{Q}^*_h$.  Since $\mathcal{Q}_h^* = \textup{dom}(\epsilon)$, $\textup{dom}(\epsilon)$ is dense in $\mathcal{Q}^*$, and $\epsilon$ is invertible, it follows that
\begin{equation*}
\lambda_h = -\mathcal{B} u_h + \epsilon^{-1} \left( \mathcal{T} u_h - g \right).
\label{eqn:NitscheLM}
\end{equation*}
Inserting the above expression for $\lambda_h$ into \eqref{eq:stabilized} and taking $\delta \lambda_h = 0$, we obtain the following reduced formulation:

\begin{mybox}[\emph{Nitsche's Method for an Abstract Variational Constrained Minimization Problem}]
  \vspace{-7pt}
$$
(N_h) \left\{ \hspace{5pt}
\parbox{6.00in}{
Given $f \in \mathcal{V}^*$ and $g \in \mathcal{Q}$, find $u_h \in \mathcal{V}_h$ such that
\begin{equation*}
  a_h(u_h,\delta u_h) = {}_{\mathcal{V}^*}\langle f, \delta u_h \rangle_{\mathcal{V}} \ {\color{ForestGreen} \underbrace{ - {}_{\mathcal{Q}^*}\langle \mathcal{B} \delta u_h, g \rangle_{\mathcal{Q}} }_\text{Symmetry Term} } \ {\color{Orchid} \underbrace{ + {}_{\mathcal{Q}^*}\langle \epsilon^{-1} \mathcal{T} \delta u_h, g \rangle_{\mathcal{Q}} }_\text{Penalty Term} }
  \label{eq:nitsche}
\end{equation*}
for every $\delta u_h \in \mathcal{V}_h$, where $a_h: \left( \tilde{\mathcal{V}} + \mathcal{V}_h \right) \times \left( \tilde{\mathcal{V}} + \mathcal{V}_h \right) \rightarrow \mathbb{R}$ is the bilinear form defined by
\begin{equation*}
  a_h(w, v) = a(w, v) {\color{Cerulean} \ \underbrace{ - {}_{\mathcal{Q}^*}\langle \mathcal{B} w, \mathcal{T}v \rangle_{\mathcal{Q}} }_\text{Consistency Term} } \ {\color{ForestGreen} \underbrace{ - {}_{\mathcal{Q}^*}\langle \mathcal{B} v, \mathcal{T} w \rangle_{\mathcal{Q}} }_\text{Symmetry Term} } \ {\color{Orchid} \underbrace{ + {}_{\mathcal{Q}^*}\langle \epsilon^{-1} \mathcal{T} v, \mathcal{T} w \rangle_{\mathcal{Q}} }_\text{Penalty Term} }
\end{equation*}
for all $w, v \in \tilde{\mathcal{V}} + \mathcal{V}_h$.
}
\right.
$$
\end{mybox}

\noindent We refer to the above formulation as \textit{\textbf{Nitsche's method}} since it is a generalization of Nitsche's method for second-order elliptic boundary value problems to arbitrary variational constrained minimization problems.

\begin{remark} Note that we can interpret Nitsche's method as a Lagrange multiplier method in which the Lagrange multiplier field $\lambda$ is approximated as
\begin{equation*}
\lambda_h = -\mathcal{B} u_h + \epsilon^{-1} \left( \mathcal{T} u_h - g \right).
\end{equation*}
Since the Lagrange multiplier field often represents one or more physical quantities of interest (e.g., $\lambda$ comprises the reaction or constraint forces, tractions, and moments in the context of structural mechanics), this formula provides a means of recovering such quantities in a variationally consistent manner \cite{hughes2000continuous,van2012flux}.
\end{remark}

Nitsche's method exhibits several important properties that give rise to its stability and convergence.  Namely, it is \textit{\textbf{consistent}}, its bilinear form $a_h(\cdot,\cdot)$ is \textit{\textbf{symmetric}}, and, provided the map $\epsilon: \textup{dom}(\epsilon) \subseteq \mathcal{Q}^* \rightarrow \mathcal{Q}$ is chosen appropriately (see Assumption~\ref{assumption2} below), its bilinear form $a_h(\cdot,\cdot)$ is also \textit{\textbf{coercive}} on the discrete space $\mathcal{V}_h$.

\begin{lemma}[Consistency]
  \label{lemma:abstract_consistency}
  Suppose that Assumption~\ref{assumption1} holds and the solution $u$ of Problem $(M)$ satisfies $u \in \tilde{\mathcal{V}}$.  Then
  \begin{equation*}
    a_h(u,\delta u_h) = {}_{\mathcal{V}^*}\langle f, \delta u_h \rangle_{\mathcal{V}} - {}_{\mathcal{Q}^*}\langle \mathcal{B} \delta u_h, g \rangle_{\mathcal{Q}} + {}_{\mathcal{Q}^*}\langle \epsilon^{-1} \mathcal{T} \delta u_h, g \rangle_{\mathcal{Q}}
  \end{equation*}
  for all $\delta u_h \in \mathcal{V}_h$.

  \begin{proof} Since Assumption~\ref{assumption1} holds and $u$ is the solution to Problem $(M)$, it follows that
    \begin{align*}
      a_h(u,\delta u_h) &= a(u, \delta u_h) - {}_{\mathcal{Q}^*}\langle \mathcal{B} u, \mathcal{T}\delta u_h \rangle_{\mathcal{Q}} - {}_{\mathcal{Q}^*}\langle \mathcal{B} \delta u_h, \mathcal{T} u \rangle_{\mathcal{Q}} + {}_{\mathcal{Q}^*}\langle \epsilon^{-1} \mathcal{T} \delta u_h, \mathcal{T} u \rangle_{\mathcal{Q}} \nonumber \\
      &= {}_{\mathcal{V}^*}\langle \mathcal{L}u, \delta u_h \rangle_{\mathcal{V}} + {}_{\mathcal{Q}^*}\langle \mathcal{B}u, \mathcal{T}\delta u_h \rangle_{\mathcal{Q}} - {}_{\mathcal{Q}^*}\langle \mathcal{B} u, \mathcal{T}\delta u_h \rangle_{\mathcal{Q}} - {}_{\mathcal{Q}^*}\langle \mathcal{B} \delta u_h, \mathcal{T} u \rangle_{\mathcal{Q}} \nonumber \\ & \hspace{8pt} + {}_{\mathcal{Q}^*}\langle \epsilon^{-1} \mathcal{T} \delta u_h, \mathcal{T} u \rangle_{\mathcal{Q}} \nonumber \\
      &= {}_{\mathcal{V}^*}\langle f, \delta u_h \rangle_{\mathcal{V}} - {}_{\mathcal{Q}^*}\langle \mathcal{B} \delta u_h, g \rangle_{\mathcal{Q}} + {}_{\mathcal{Q}^*}\langle \epsilon^{-1} \mathcal{T} \delta u_h, g \rangle_{\mathcal{Q}}
    \end{align*}
    for all $\delta u_h \in \mathcal{V}_h$.
  \end{proof}
\end{lemma}

\begin{lemma}[Symmetry]
  \label{lemma:abstract_symmetry}
  It holds that
  \begin{equation*}
    a_h(w, v) = a_h(v, w)
  \end{equation*}
  for all $w, v \in \tilde{\mathcal{V}} + \mathcal{V}_h.$

  \begin{proof}
    The result follows by direct computation.
  \end{proof}
\end{lemma}

To establish a coercivity result for Nitsche's method, we must make another assumption.

\begin{assumption}
  \label{assumption2}
There exists a densely defined, positive, self-adjoint linear map $\eta: \textup{dom}(\eta) \subseteq \mathcal{Q}^* \rightarrow \mathcal{Q}$ with the following properties:
\begin{enumerate}
\item The space $\left\{ \mathcal{B}v: v \in \tilde{\mathcal{V}} + \mathcal{V}_h \right\}$ is a subset of $\textup{dom}(\eta)$.
\item The following \textit{\textbf{generalized trace inequality}} holds:
\begin{equation*}
{}_{\mathcal{Q}^*}\langle \mathcal{B}v_h, \eta \mathcal{B}v_h \rangle_{\mathcal{Q}} \leq a(v_h,v_h)
\label{eqn:gen_trace}
\end{equation*}
for all $v_h \in \mathcal{V}_h$.
\item The following \textit{\textbf{generalized Cauchy-Schwarz inequality}} holds:
\begin{equation*}
\left| {}_{\mathcal{Q}^*}\langle \mathcal{B} v, \mathcal{T} w \rangle_{\mathcal{Q}} \right| \leq \frac{1}{\CSconstant} {}_{\mathcal{Q}^*}\langle \mathcal{B}v, \eta \mathcal{B}v \rangle^{1/2}_{\mathcal{Q}} {}_{\mathcal{Q}^*}\langle \epsilon^{-1} \mathcal{T} w, \mathcal{T} w \rangle^{1/2}_{\mathcal{Q}}
\label{eqn:gen_CS}
\end{equation*}
for all $v, w \in \tilde{\mathcal{V}} + \mathcal{V}_h$, where $\CSconstant \in (1,\infty)$.
\end{enumerate}
\end{assumption}

Now, defining an energy norm $\vvvertiii{\cdot}: \tilde{\mathcal{V}} + \mathcal{V}_h \rightarrow \mathbb{R}$ via
\begin{equation*}
\vvvertiii{v}^2 := a(v,v) + {}_{\mathcal{Q}^*}\langle \mathcal{B}v, \eta \mathcal{B}v \rangle_{\mathcal{Q}} + 2 {}_{\mathcal{Q}^*}\langle \epsilon^{-1} \mathcal{T} v, \mathcal{T} v \rangle_{\mathcal{Q}},
\end{equation*}
we can derive a coercivity result for Nitsche's method.

\begin{lemma}[Coercivity]
  \label{lemma:abstract_coercivity}
Suppose that Assumption~\ref{assumption2} holds.  Then
\begin{equation*}
a_h(v_h,v_h) \geq \frac{1}{2} \left(1 - \frac{1}{\CSconstant} \right) \vvvertiii{v_h}^2
\end{equation*}
for all $v_h \in \mathcal{V}_h$.

\begin{proof}
  Since Assumption~\ref{assumption2} holds, it follows that
\begin{align*}
a_h(v_h,v_h) &= a(v_h, v_h) - 2{}_{\mathcal{Q}^*}\langle \mathcal{B} v_h, \mathcal{T}v_h \rangle_{\mathcal{Q}} + {}_{\mathcal{Q}^*}\langle \epsilon^{-1} \mathcal{T} v_h, \mathcal{T} v_h \rangle_{\mathcal{Q}} \nonumber \\
& \geq a(v_h, v_h) - 2 \left| {}_{\mathcal{Q}^*}\langle \mathcal{B} v_h, \mathcal{T}v_h \rangle_{\mathcal{Q}} \right| + {}_{\mathcal{Q}^*}\langle \epsilon^{-1} \mathcal{T} v_h, \mathcal{T} v_h \rangle_{\mathcal{Q}} \nonumber \\
& \geq a(v_h, v_h) - \frac{2}{\CSconstant}  {}_{\mathcal{Q}^*}\langle \mathcal{B}v_h, \eta \mathcal{B}v_h \rangle^{1/2}_{\mathcal{Q}} {}_{\mathcal{Q}^*}\langle \epsilon^{-1} \mathcal{T} v_h, \mathcal{T} v_h \rangle^{1/2}_{\mathcal{Q}} + {}_{\mathcal{Q}^*}\langle \epsilon^{-1} \mathcal{T} v_h, \mathcal{T} v_h \rangle_{\mathcal{Q}} \nonumber \\
& \geq a(v_h, v_h) - \frac{1}{\CSconstant}  \left( {}_{\mathcal{Q}^*}\langle \mathcal{B}v_h, \eta \mathcal{B}v_h \rangle_{\mathcal{Q}} + {}_{\mathcal{Q}^*}\langle \epsilon^{-1} \mathcal{T} v_h, \mathcal{T} v_h \rangle_{\mathcal{Q}} \right) + {}_{\mathcal{Q}^*}\langle \epsilon^{-1} \mathcal{T} v_h, \mathcal{T} v_h \rangle_{\mathcal{Q}} \nonumber \\
& \geq a(v_h, v_h) - \frac{1}{\CSconstant}  \left( a(v_h, v_h) + {}_{\mathcal{Q}^*}\langle \epsilon^{-1} \mathcal{T} v_h, \mathcal{T} v_h \rangle_{\mathcal{Q}} \right) + {}_{\mathcal{Q}^*}\langle \epsilon^{-1} \mathcal{T} v_h, \mathcal{T} v_h \rangle_{\mathcal{Q}} \nonumber \\
& \geq \left( 1 - \frac{1}{\CSconstant} \right) a(v_h, v_h) + \left( 1 - \frac{1}{\CSconstant} \right)  {}_{\mathcal{Q}^*}\langle \epsilon^{-1} \mathcal{T} v_h, \mathcal{T} v_h \rangle_{\mathcal{Q}} \nonumber \\
& \geq \frac{1}{2} \left( 1 - \frac{1}{\CSconstant} \right) \left( a(v_h, v_h) + {}_{\mathcal{Q}^*}\langle \mathcal{B}v_h, \eta \mathcal{B}v_h \rangle_{\mathcal{Q}} \right) + \left( 1 - \frac{1}{\CSconstant} \right) {}_{\mathcal{Q}^*}\langle \epsilon^{-1} \mathcal{T} v_h, \mathcal{T} v_h \rangle_{\mathcal{Q}} \nonumber \\
& = \frac{1}{2} \left( 1 - \frac{1}{\CSconstant} \right) \vvvertiii{v_h}^2\nonumber
\end{align*}
for all $v_h \in \mathcal{V}_h$.  Young's inequality ($|ab| \leq \frac{1}{2}(a^2 + b^2)$ for $a,b \in \mathbb{R}$) is used here from line three to four.
\end{proof}

\end{lemma}

We need one more result before we can establish an error estimate for Nitsche's method.

\begin{lemma}[Continuity]
  \label{lemma:abstract_continuity}
Suppose that Assumption~\ref{assumption2} holds.  Then
\begin{equation*}
|a_h(w,v)| \leq \vvvertiii{w} \cdot \vvvertiii{v}
\end{equation*}
for all $w, v \in \tilde{\mathcal{V}} + \mathcal{V}_h$.

\begin{proof}
Since Assumption~\ref{assumption2} holds, it follows that
\begin{align*}
a_h(w,v) &= a(w, v) - {}_{\mathcal{Q}^*}\langle \mathcal{B} w, \mathcal{T}v \rangle_{\mathcal{Q}} - {}_{\mathcal{Q}^*}\langle \mathcal{B} v, \mathcal{T}w \rangle_{\mathcal{Q}} + {}_{\mathcal{Q}^*}\langle \epsilon^{-1} \mathcal{T} v, \mathcal{T} w \rangle_{\mathcal{Q}} \nonumber \\
& \leq a(w,v) + \left| {}_{\mathcal{Q}^*}\langle \mathcal{B} w, \mathcal{T}v \rangle_{\mathcal{Q}} \right| + \left| {}_{\mathcal{Q}^*}\langle \mathcal{B} v, \mathcal{T}w \rangle_{\mathcal{Q}} \right| + {}_{\mathcal{Q}^*}\langle \epsilon^{-1} \mathcal{T} v, \mathcal{T} w \rangle_{\mathcal{Q}} \nonumber \\
& \leq a(w, w)^{1/2} a(v,v)^{1/2} + {}_{\mathcal{Q}^*}\langle \mathcal{B}w, \eta \mathcal{B}w \rangle^{1/2}_{\mathcal{Q}} {}_{\mathcal{Q}^*}\langle \epsilon^{-1} \mathcal{T} v, \mathcal{T} v \rangle^{1/2}_{\mathcal{Q}} + {}_{\mathcal{Q}^*}\langle \mathcal{B}v, \eta \mathcal{B}v \rangle^{1/2}_{\mathcal{Q}} {}_{\mathcal{Q}^*}\langle \epsilon^{-1} \mathcal{T} w, \mathcal{T} w \rangle^{1/2}_{\mathcal{Q}} \nonumber \\
&\phantom{\leq} + {}_{\mathcal{Q}^*}\langle \epsilon^{-1} \mathcal{T} v, \mathcal{T} v \rangle^{1/2}_{\mathcal{Q}} {}_{\mathcal{Q}^*}\langle \epsilon^{-1} \mathcal{T} w, \mathcal{T} w \rangle^{1/2}_{\mathcal{Q}} \nonumber \\
& \leq \vvvertiii{w} \cdot \vvvertiii{v}
\end{align*}
for all $w,v \in \tilde{\mathcal{V}} + \mathcal{V}_h$. The standard Cauchy-Schwarz inequality $(|(x,y)| \leq \| x \|_2 \|y \|_2$ for $x,y \in \mathbb{R}^n)$ is used from line two to three above.
\end{proof}
\end{lemma}

We are now ready to prove well-posedness and an error estimate for Nitsche's method.

\begin{theorem}[Well-Posedness and Error Estimate]
\label{theorem:error_estimate}
Suppose that Assumptions~\ref{assumption1} and~\ref{assumption2} hold.  Then there exists a unique discrete solution $u_h \in \mathcal{V}_h$ to Problem $(N_h)$.  Moreover, if the continuous solution $u \in \mathcal{V}$ to Problem $(M)$ satisfies $u \in \tilde{\mathcal{V}}$, then the discrete solution $u_h$ satisfies the error estimate
\begin{equation*}
\vvvertiii{u - u_h} \leq \left( 1+ \frac{2}{1-\frac{1}{\CSconstant}} \right) \min_{v_h \in \mathcal{V}_h} \vvvertiii{u - v_h}.
\end{equation*}

\begin{proof}
Well-posedness is a direct result of the Lax-Milgram Theorem and coercivity and continuity as established by Lemmas~\ref{lemma:abstract_coercivity} and \ref{lemma:abstract_continuity}.  To prove the error estimate, let $v_h \in \mathcal{V}_h$ be an arbitrary function.  Since Assumption~\ref{assumption2} holds, by Lemma~\ref{lemma:abstract_coercivity}, we have that
\begin{align*}
\vvvertiii{u_h - v_h}^2 \leq \frac{2}{1-\frac{1}{\CSconstant}} a\left(u_h - v_h, u_h - v_h\right).
\end{align*}
By Assumption~\ref{assumption1} and Lemma~\ref{lemma:abstract_consistency}, we have that
\begin{align*}
\vvvertiii{u_h - v_h}^2 \leq \frac{2}{1-\frac{1}{\CSconstant}} a\left(u - v_h, u_h - v_h\right).
\end{align*}
By Assumption~\ref{assumption2} and Lemma~\ref{lemma:abstract_continuity}, we have that
\begin{align*}
\vvvertiii{u_h - v_h}^2 \leq \frac{2}{1-\frac{1}{\CSconstant}} \vvvertiii{u - v_h} \vvvertiii{u_h - v_h}
\end{align*}
and, hence,
\begin{align*}
\vvvertiii{u_h - v_h} \leq \frac{2}{1-\frac{1}{\CSconstant}} \vvvertiii{u - v_h}.
\end{align*}
By the triangle inequality, we have that
\begin{align*}
\vvvertiii{u - u_h} &\leq \vvvertiii{u - v_h} + \vvvertiii{u_h - v_h} \nonumber
\leq  \left( 1+ \frac{2}{1-\frac{1}{\CSconstant}} \right) \vvvertiii{u - v_h}
\end{align*}
and, since $v_h \in \mathcal{V}_h$ is arbitrary, the final result holds.
\end{proof}
\end{theorem}

Note that the above theorem applies to any formulation and problem setup for which Assumptions~\ref{assumption1} and~\ref{assumption2} hold.  Consequently, constructing Nitsche-based formulations for a new problem class should proceed according to the following steps:\\

\noindent \textbf{Step 1:} Construct an appropriate variational formulation (including specification of the Hilbert spaces $\mathcal{V}$ and $\mathcal{Q}$, the map $\mathcal{T}: \mathcal{V} \rightarrow \mathcal{Q}$, and the bilinear form $a : \mathcal{V} \times \mathcal{V} \rightarrow \mathbb{R}$) such that Assumption~\ref{assumption1} is satisfied, and determine the space $\tilde{\mathcal{V}}$ and the linear maps $\mathcal{L} : \tilde{\mathcal{V}} \rightarrow \mathcal{V}^*$ and $\mathcal{B} : \tilde{\mathcal{V}} \rightarrow \mathcal{Q}^*$ associated with Assumption~\ref{assumption1}. Note that the relevant operators will ultimately be defined over the extended domain $\tilde{\mathcal{V}} + \mathcal{V}_h$ for discretization.\\

\noindent \textbf{Step 2:} Construct suitable linear maps $\epsilon: \textup{dom}(\epsilon) \subseteq Q^* \rightarrow Q$ and $\eta: \textup{dom}(\eta) \subseteq Q^* \rightarrow Q$ such that Assumption~\ref{assumption2} is satisfied.\\

\noindent \textbf{Step 3:} Pose Nitsche's method according to Problem $(N_h)$.\\

In the following, we complete the above three steps to construct a new Nitsche-based formulation for the linear Kirchhoff-Love shell.  Note that we do not need to conduct a full stability and convergence analysis, since we can readily employ the abstract framework presented here.  It should further be mentioned that symmetry can be employed to arrive at error estimates in norms other than the energy norm using the well-known Aubin-Nitsche trick \cite{ciarlet1991basic}.  To complete our analysis, and in particular to construct the maps $\epsilon$ and $\eta$ such that Assumption~\ref{assumption2} is satisfied, we need to use function analytic results such as trace inequalities. We discuss later how to compute trace constants in a practical manner.  Finally, to ensure that our results are invariant with respect to scaling, we take special care in constructing the maps $\epsilon$ and $\eta$ so that the resulting energy norm is dimensionally consistent.  This requires a little finesse and rigor, but we believe that arriving at scale-invariant error estimates is worth the added effort.


\section{Nitsche's Method for the Linear Kirchhoff-Love Shell Problem}
\label{sec:KL_Shell}

Our abstract framework provides a convenient means for constructing and analyzing Nitsche-based formulations for problems of interest, regardless of their complexity. In this section, we apply our framework to the vector-valued, fourth-order PDE that governs the linear Kirchhoff-Love shell to arrive at a provably convergent Nitsche-based formulation. We also derive and discuss what are known as the \textbf{\emph{ersatz forces}}, or modified boundary shear forces used to maintain variational consistency of the Nitsche formulation. These are either incorrect or incomplete in the existing literature. Because we only consider the linear case in what follows, we drop ``linear'' from ``linear Kirchhoff-Love shell'' in this and subsequent sections.

In the following, underline and double underline ($\surfVec{\bullet}$ and $\surfTens{\bullet}$) are used to denote manifold quantities, that is, quantities that can be expressed through a linear combination of tensorial quantities lying in the tensor bundle of the manifold, with the number of underlines indicating the order of the tensor. By contrast, \fullTens{bold-faced text} denotes quantities residing in three-dimensional space. The concepts presented in this section, and those that follow, rely heavily on differential geometry and continuum mechanics posed over differentiable manifolds. For a brief discussion of the necessary differential-geometric subjects, see to \ref{sec:Appendix_Diff_Geo} and for a review of continuum mechanics, see to \ref{sec:Appendix_Cont_Mech}.

\begin{figure}[ht!]
\includegraphics[width=\textwidth]{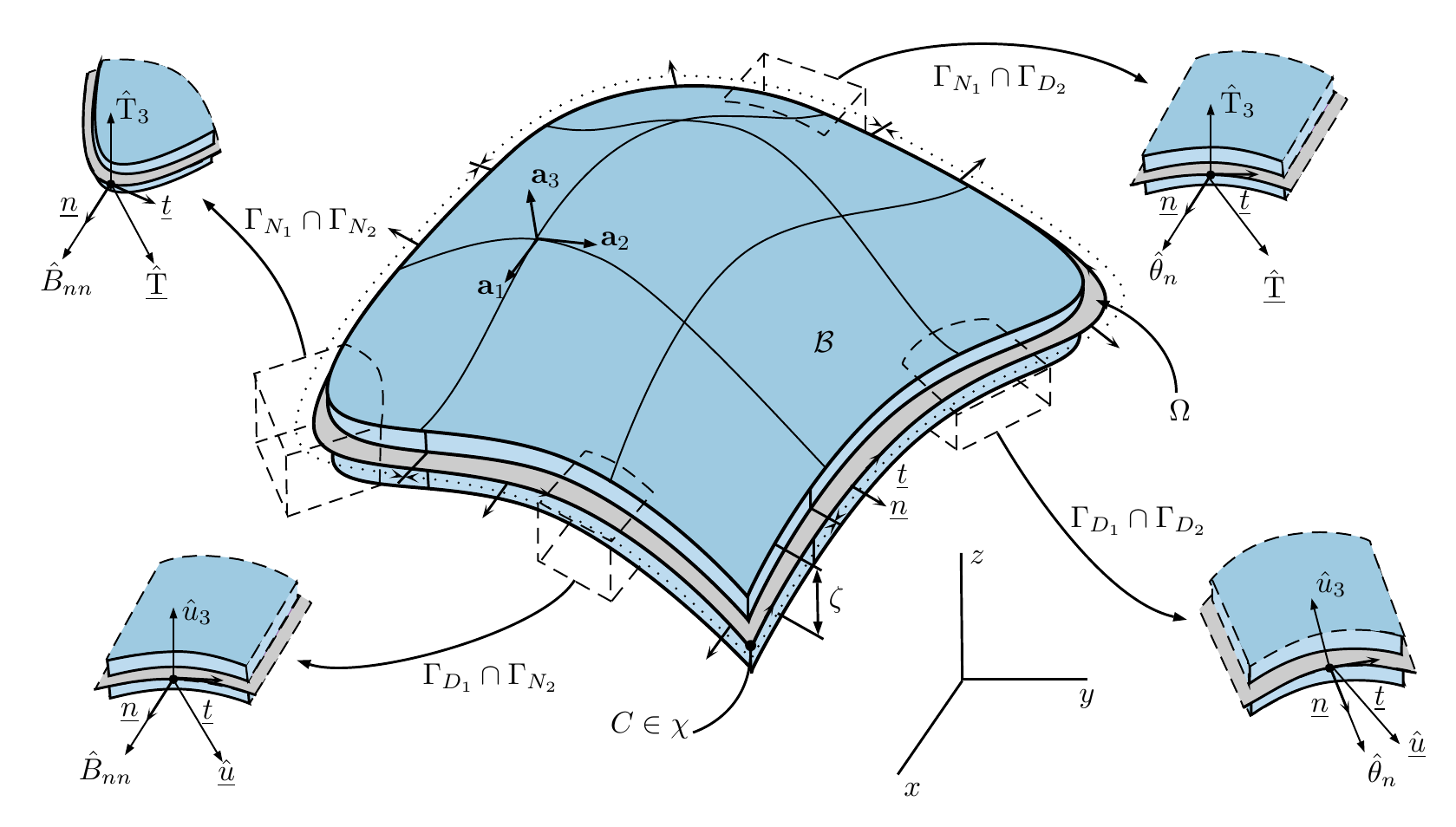}
\caption{An arbitrary shell domain. All positive conventions for degrees of freedom and applied loadings are depicted.}
\label{fig:KL_shell_domain}
\end{figure}

Shell models simulate the structural response of curved, load-bearing members subject to both in-plane and out-of-plane loadings. They are idealized through a midsurface model with linearized through-thickness displacement profiles, where we use $\thickness$ to denote the thickness variable. The midsurface is chosen to be the surface midway through the thickness of the shell body. In special cases, namely, small strains and displacements and shells comprised of an isotropic material, the midsurface coincides with the neutral plane, that is, the plane that undergoes no compressive or tensile forces due to bending. A general shell model employs a displacement variable, denoted ${\bf u}$, as well as a rotational degree of freedom, denoted $\surfVec{\theta}$. The Kirchhoff-Love shell displacement field is assumed to be free of transverse shear strain. Consequently, this introduces a constraint between the rotational and displacement degrees of freedom, namely, $\surfVec{\theta}({\bf u}) = - \aVec_{3} \cdot \surfVec{\nabla} {\bf u}$, which appears later in our derivations. We integrate through-thickness before discretization, which introduces a $\thickness$-dependence in the expression for membrane action and a $\thickness^3$-dependence in the expression for bending action due to a zero-transverse shear strain constraint imposed on the displacement variable. It is this discrepancy between thickness-dependence in the presence of intrinsic curvature coupling that gives rise to \textbf{\emph{membrane locking}}, a parasitic numerical phenomenon that causes little-to-no displacement for thin shells in specific configurations until sufficient mesh resolution is attained.


\subsection{The Variational Formulation}

Let $\physDomain \subset \mathbb{R}^3$ be an immersed two-dimensional manifold with Lipschitz-continuous boundary $\physBoundary = \partial \physDomain$. Assume that $\physDomain$ is smooth enough that the derivatives of the curvature are finite. Note that less-smooth manifolds are admissible to the methodology we present in this section; however, special care must be taken in regions without appropriate smoothness and, hence, we invoke this assumption for simplicity of exposition. Since the Kirchhoff-Love shell accommodates both prescribed displacements and rotations as well as their energetically conjugate shears and moments on the boundary, we partition the boundary accordingly. In particular, let $\physBoundary_{D_1}$ and $\physBoundary_{N_1}$ be the Dirichlet-1 and Neumann-1 boundaries associated with prescribed transverse displacements and applied transverse shears, respectively. Let $\physBoundary_{D_2}$ and $\physBoundary_{N_2}$ be the Dirichlet-2 and Neumann-2 boundaries associated with prescribed normal rotations and applied bending moments, respectively. For a well-posed PDE, we require that $\physBoundary = \overline{\physBoundary_{D_\alpha} \cup \physBoundary_{N_\alpha}}$, $\physBoundary_{D_\alpha} \cap \physBoundary_{N_\alpha} = \emptyset$, and $\physBoundary_{D_\alpha} \neq \emptyset$ for $\alpha = 1,2$. Note that there are no constraints between the 1- and 2-boundaries because there is no energetic exchange between the two sets. In general, $\physBoundary_{D_\alpha} \neq \emptyset$ for $\alpha = 1,2$ is not necessary, e.g., $\physBoundary_{D_2}$ can be empty provided sufficient conditions are imposed on $\physBoundary_{D_1}$; however, for simplicity, we make the broader assumption. We further introduce the set $\cornerSet{} \subset \physBoundary$ as the set of ``corners'', that is, the non-differentiable loci with zero-Lebesgue measure, along the boundary.  We further decompose this set into $\cornerSet{D} := \cornerSet{} \cap \overline{\physBoundary_{D_1}}$ and $\cornerSet{N} := \cornerSet{} \cap \physBoundary_{N_1}$ and note that, by construction, $\cornerSet{} = \cornerSet{D} \cup \cornerSet{N}$ and $\cornerSet{D} \cap \cornerSet{N} = \emptyset$.  We denote corners as $C \in \cornerSet{}$.

Given a geometric mapping ${\bf x}$ from a parametric domain to the midsurface $\physDomain$, we are able to construct a \textbf{\emph{covariant}} coordinate basis through the derivatives of the convected coordinates. In particular, the $\alpha^{th}$ covariant basis vector is given by $\aVec_\alpha = {\bf x}_{,\alpha}$, where the comma notation refers to differentiation of the geometric mapping with respect to the $\alpha^{th}$ coordinate. The midsurface normal director $\aVec_3$ can be constructed through a cross product of these in-plane vectors and, provided the geometric mapping ${\bf x}$ is non-degenerate, the resulting covariant set can be shown to form a basis of $\mathbb{R}^3$. According to differential-geometric theory, we can uniquely construct an algebraically dual set of \textbf{\emph{contravariant}} basis vectors to this set denoted $\aVec^\alpha$, which, by definition, satisfy the Kronecker relationship $\aVec_\alpha \cdot \aVec^\beta = \delta^\beta_\alpha$. Note by construction that $\aVec_3 = \aVec^3$. For a deeper discussion of the required differential-geometric tools, see \ref{sec:Appendix_Diff_Geo}. These contravariant basis vectors allow us to effectively combine the in-plane and out-of-plane behaviors via ${\bf w} = \surfVec{w} + w_3 \aVec^3$, where $\surfVec{w} = w_\alpha \aVec^\alpha$.  Later in this section, we will invoke the \textit{\textbf{in-plane projector}} $\surfTens{\Projector} := \textbf{I} - \aVec^3 \otimes \aVec_3$ that, when acting on a vector, returns the in-plane part of that vector, where $\textbf{I}$ is the identity tensor. Note that $\surfTens{\Projector}$ is symmetric and thus also satisfies the definition $\surfTens{\Projector} = {\bf I} - \aVec_3 \otimes \aVec^3$. Finally, we present various quantities defined over manifolds with their required regularities in terms of Sobolev embeddings. Since these spaces are defined over manifolds, we present what this entails more rigorously in Section~\ref{sec:apriori}.

Let $\applied{\textbf{\forcing}} \in \left( L^2(\physDomain) \right)^3$ be the applied body loading, $\applied{\bf u} = \applied{\surfVec{u}} + \applied{u}_3 \aVec^3$ such that $(\applied{u}_1,\applied{u}_2) \in \left( H^{1/2}(\physBoundary_{D_1}) \right)^2$ and $\applied{u}_3 \in H^{3/2}(\physBoundary_{D_1})$ are the prescribed displacement, and let $\applied{\theta}_n \in H^{1/2}(\physBoundary_{D_2})$ be the prescribed normal rotation. In general, ``hat'' notation $(\applied{\bullet})$ is used to denote a quantity that is prescribed or applied. Note that, by Sobolev embedding, $\applied{u}_3 \in C^0(\physBoundary_{D_1})$. Given an applied traction $\applied{\bm{\traction}} \colon \physBoundary_{N_1} \rightarrow \mathbb{R}^3$ and an applied twisting moment $\applied{\moment}_{nt} \colon \physBoundary_{N_1} \rightarrow \mathbb{R}$, define the \textbf{\emph{ersatz traction}} via
\begin{equation}
  \applied{\bf T} = \underbrace{ \applied{\surfVec{\traction}} - \applied{\moment}_{nt} \surfTens{\SFF} \cdot \surfVec{\bdyTangent} }_{\displaystyle \surfVec{\applied{\ersatz}}} + \underbrace{ \left[ \shear + \frac{\partial \applied{\moment}_{nt}}{\partial \bdyTangent} \right] }_{\displaystyle \applied{\ersatz}_3} \aVec^{3},
  \label{eqn:Ersatz_traction_KLS}
\end{equation}
where the term $\surfTens{b}$ is the second fundamental form, or \textbf{\emph{curvature tensor}} \eqref{eqn:secondFF}, associated with the manifold. The corresponding \textbf{\emph{corner forces}} are defined via
\begin{equation*}
  \applied{\force} = \llbracket \applied{\moment}_{nt} \rrbracket,
\end{equation*}
where
\begin{equation}
\llbracket \applied{\moment}_{nt} \rrbracket = \lim_{\epsilon \rightarrow 0} \left( \applied{\moment}_{nt}(\textbf{x} + \epsilon \surfVec{\bdyTangent}) - \applied{\moment}_{nt}(\textbf{x} - \epsilon \surfVec{\bdyTangent}) \right)
\label{eqn:jumpDef}
\end{equation}
and $\surfVec{\bdyTangent}$ is the positively oriented, counter-clockwise unit tangent vector to $\physBoundary$. The corner forces and the ersatz traction arise from the integration-by-parts formula
\begin{equation*}
  \begin{aligned}
    \int_{\physBoundary_{N_1}} \applied{\moment}_{nt} \midsurfRot_t({\bf v}) \ d \physBoundary &= \int_{\physBoundary_{N_1}} \outOfPlane{v} \frac{\partial \applied{\moment}_{nt}}{\partial \bdyTangent} \ d \physBoundary + \sum_{C \in \cornerSet{N}} \left. \left( \llbracket \applied{\moment}_{nt} \rrbracket \outOfPlane{v} \right) \right|_{C} - \int_{\physBoundary_{N_1}} \left( \applied{\moment}_{nt} \ \surfTens{b} \cdot \surfVec{\bdyTangent} \right) \cdot \surfVec{v} \ d \physBoundary
  \end{aligned}
\label{eqn:Ersatz_IBP_KLS}
\end{equation*}
for any ${\bf v} \colon \physBoundary_{N_1} \rightarrow \mathbb{R}^3$ with $\left. v_3 \right|_{\partial \physBoundary_{N_1}} = 0$, where $\midsurfRot_t({\bf v}) = - \left( \aVec_3 \cdot \surfVec{\nabla} {\bf v} \right) \cdot \surfVec{t}$ is the \textbf{\emph{twisting rotation}} and $\surfVec{t}$ is again the positively oriented unit tangent vector to $\physBoundary$. In contrast to the boundary traction and twisting moment, the ersatz traction and corner forces are energetically conjugate to the boundary displacement, so they are the the natural entities to use in our derivation of Nitsche's method for the Kirchhoff-Love shell through our abstract framework (see Remark~\ref{alternative_f_KLS} below). Assume that $\applied{\bf T} \in \left( L^2(\physBoundary_{N_1}) \right)^3$ and $\left\{ \applied{\force}|_{C} \right\}_{C \in \cornerSet{N}} \in \mathbb{R}^{\#\cornerSet{N}}$. Finally, let $\applied{\moment}_{nn} \in L^2(\physBoundary_{N_2})$ be the applied bending moment that is energetically conjugate to the boundary rotation.

Throughout the remainder of the paper, we use superscript $\KLS$ to denote quantities associated with the Kirchhoff-Love shell problem to differentiate them from those in the abstract framework. In order to apply the abstract results from Section~\ref{sec:Nitsche} to the Kirchhoff-Love shell, let
\begin{equation*}
  \mathcal{V}^{\KLS} := \left\{ {\bf v} = \surfVec{v} + v_3 \aVec^3 \ \colon \ (v_1,v_2) \in \left( H^1(\physDomain) \right)^2 \hspace{5pt} \text{and} \hspace{5pt} v_3 \in H^2(\physDomain) \right\}
\end{equation*}
and
\begin{equation*}
  \mathcal{Q}^{\KLS} := \left\{ ({\bf v},\mu_n) \ \colon \ (v_1,v_2) \in \left( H^{1/2}(\physBoundary_{D_1}) \right)^2, \ v_3 \in H^{3/2}(\physBoundary_{D_1}), \hspace{5pt} \text{and} \hspace{5pt} \mu_n \in H^{1/2}(\physBoundary_{D_2}) \right\}.
\end{equation*}
These spaces are selected in this way to accommodate the required smoothness of a weak solution to the underlying PDE. In particular, $\mathcal{V}^{\KLS}$ is constructed such that members of $\mathcal{V}^{\KLS}$ have one integrable derivative in-plane and two integrable derivatives out-of-plane, respectively. Accordingly, $\mathcal{Q}^{\KLS}$ is the corresponding trace space that outlines the necessary smoothness for the applied displacement field and normal rotation field along the Dirichlet boundary. As such, define the trace operator $\mathcal{T}^{\KLS} \colon \mathcal{V}^{\KLS} \rightarrow \mathcal{Q}^{\KLS}$ via its action on the displacement field ${\bf v} \in \mathcal{V}^{\KLS}$, i.e., $\mathcal{T}^{\KLS} {\bf v} = \left. \left( {\bf v}, \midsurfRot_n({\bf v}) \right) \right|_{\physBoundary_D}$, where $\midsurfRot_n({\bf v}) = - \left( \aVec_3 \cdot \surfVec{\nabla} {\bf v} \right) \cdot \surfVec{n}$ is the \textbf{\emph{normal rotation}} and $\surfVec{n}$ is the outward-facing unit normal to $\physBoundary$. Given $\left( \applied{\bf u}, \applied{\midsurfRot}_n \right) \in \mathcal{Q}^{\KLS}$, define
\begin{equation*}
  \mathcal{V}^{\KLS}_{\applied{\bf u},\applied{\midsurfRot}_n} := \left\{ {\bf v} \in \mathcal{V}^{\KLS} \colon \mathcal{T}^{\KLS} {\bf v} = \left( \applied{\bf u},\applied{\midsurfRot}_n \right) \right\}
\end{equation*}
as the trial space of displacement fields satisfying the prescribed Dirichlet boundary conditions. To this end, $\mathcal{V}^{\KLS}_{{\bf 0},0}$ denotes the test space of displacement fields satisfying homogeneous Dirichlet boundary conditions, particularly $\applied{\bf u} = {\bf 0}$ and $\applied{\midsurfRot}_n = 0$.

Given external loadings and boundary conditions, we introduce the corresponding strain and stress measures that serve as our proxy for the resulting displacement field. For ${\bf w} \in \mathcal{V}^{\KLS}$, the midsurface rotation is given by the negative gradient of this displacement variable projected onto the midsurface normal director through the Kirchhoff-Love kinematical assumption \eqref{eqn:KL_constraint}, namely, $\surfVec{\theta}({\bf w}) = - \aVec_3 \cdot \surfVec{\nabla} {\bf w}$. This is readily seen by setting the transverse shear strain to zero and solving algebraically for $\surfVec{\theta}$ in terms of ${\bf w}$ in Table~\ref{table:VariousStrains}. The \textbf{\emph{membrane strain}} \eqref{eqn:mem_strain} is defined as $\surfTens{\memStrain}({\bf w}) := \surfTens{\Projector} \cdot \Sym\left( \surfVec{\nabla} {\bf w} \right) \cdot \surfTens{\Projector}$, where the operator $\Sym(\cdot)$ returns the symmetric part of the displacement gradient, in particular, $\Sym \left( \surfVec{\nabla} \ {\bf w} \right) := \frac{1}{2}\left[ \left( \surfVec{\nabla} \ {\bf w} \right) + \left( \surfVec{\nabla} \ {\bf w} \right)^T \right]$. The \textbf{\emph{membrane stress} \eqref{eqn:mem_stress}} is defined via $\surfTens{\memStress}({\bf w}) := \thickness \mathbb{\constitutive} \colon \surfTens{\memStrain}({\bf w})$, that is, the composition of the membrane strain with the elasticity tensor. Analogously, the \textbf{\emph{bending strain}} \eqref{eqn:bend_strain} is defined as $\surfTens{\bendStrain}({\bf w}) := - \surfTens{\Projector} \cdot \Sym\left( \aVec_3 \cdot \surfVec{\nabla} \ \surfVec{\nabla} {\bf w} \right) \cdot \surfTens{\Projector}$ and the \textbf{\emph{bending stress} \eqref{eqn:bend_stress} }is defined as $\surfTens{\bendStress}({\bf w}) := \frac{\thickness^3}{12}\mathbb{C} \colon \surfTens{\bendStrain}({\bf w})$. The \textbf{\emph{surface gradient}}, which we denote $\surfVec{\nabla}$, is defined in \eqref{eqn:surfGradDef}.

\begin{remark}
  It can be shown that the magnitude of $\mathbb{\constitutive}$ is given by
  \begin{equation*}
    | \mathbb{\constitutive} |^2 = \mathbb{\constitutive}^{\alpha\beta\lambda\mu} \mathbb{\constitutive}_{\alpha\beta\lambda\mu} =  \frac{3 \nu^2 - 2 \nu + 3}{\left( 1-\nu^2 \right)^2} E^2,
  \end{equation*}
  where $E$ is Young's modulus and $\nu$ is Poisson's ratio. Since $0 \le \nu \le \frac{1}{2}$, it follows that $| \mathbb{\constitutive} |^2 \le \frac{44}{9} E^2$.
\end{remark}

We are interested in the following variational constrained minimization problem for the Kirchhoff-Love shell:
$$
(M^{\KLS}) \left\{ \hspace{5pt}
\parbox{6.00in}{
\noindent Find ${\bf u} \in \mathcal{V}^{\KLS}_{\applied{\bf u},\applied{\theta}_n}$ that minimizes the total energy
\begin{eqnarray*}
  E^{\KLS}_\text{total}({\bf u}) = E^{\KLS}_\text{int}({\bf u}) + E^{\KLS}_\text{ext}({\bf u}),
\end{eqnarray*}
where
\begin{equation}
  E^{\KLS}_\text{int}({\bf v}) = \underbrace{ \frac{1}{2} \int_{\physDomain} \surfTens{\memStress}({\bf v}) \colon \surfTens{\memStrain}({\bf v}) \ d \physDomain}_{\text{Membrane Energy}} + \underbrace{ \frac{1}{2} \int_{\physDomain} \surfTens{\bendStress}({\bf v}) \colon \surfTens{\bendStrain}({\bf v}) \ d \physDomain}_{\text{Bending Energy}}
  \label{eqn:E_int_S}
\end{equation}
is the internal strain energy due to both membrane and bending effects and
\begin{equation}
  E^{\KLS}_\text{ext}({\bf v}) = -\int_{\physDomain} \applied{\textbf{\forcing}} \cdot {\bf v} \ d \physDomain - \int_{\physBoundary_{N_1}} \applied{\bf T} \cdot {\bf v} \ d \physBoundary - \sum_{C \in \cornerSet{N}} \left. \left( \applied{\force} \outOfPlane{v} \right) \right|_{C} - \int_{\physBoundary_{N_2}} \applied{\moment}_{nn} \midsurfRot_n({\bf v}) \ d \physBoundary
  \label{eqn:E_ext_S}
\end{equation}
is the external energy due to applied loadings.
}
\right.
$$

\noindent We define an associated bilinear form $a^{\KLS}(\cdot,\cdot) \colon \mathcal{V}^{\KLS} \times \mathcal{V}^{\KLS} \rightarrow \mathbb{R}$ as twice the shell strain energy
\begin{equation*}
  a^{\KLS}({\bf w},{\bf v}) := \int_{\physDomain} \surfTens{\memStress}({\bf w}) \colon \surfTens{\memStrain}({\bf v}) \ d \physDomain + \int_{\physDomain} \surfTens{\bendStress}({\bf w}) \colon \surfTens{\bendStrain}({\bf v}) \ d \physDomain
\end{equation*}
for all ${\bf w},{\bf v} \in \mathcal{V}^{\KLS}$. The linear functional $\linFunctional^{\KLS} \in \left( \mathcal{V}^{\KLS} \right)^*$ is defined via
\begin{equation*}
\begin{aligned}
 \left\langle \linFunctional^{\KLS}, {\bf v} \right\rangle = \int_{\physDomain} \applied{\textbf{\forcing}} \cdot {\bf v} \ d \physDomain + \int_{\physBoundary_{N_1}} \applied{\bf T} \cdot {\bf v} \ d \physBoundary + \sum_{C \in \cornerSet{N}} \left. \left( \applied{\force} \outOfPlane{v} \right) \right|_{C} + \int_{\physBoundary_{N_2}} \applied{\moment}_{nn} \midsurfRot_n({\bf v}) \ d \physBoundary
\end{aligned}
\end{equation*}
for all ${\bf v} \in \mathcal{V}^{\KLS}$. Therefore, the solution to Problem $(M^{\KLS})$ is also the solution to the following variational problem:
$$
(V^{\KLS}) \left\{ \hspace{5pt}
\parbox{6.00in}{
\noindent Find ${\bf u} \in \mathcal{V}^{\KLS}_{\applied{\bf u},\applied{\theta}_n}$ such that
\begin{eqnarray*}
a^{\KLS}({\bf u},\delta {\bf u}) = \left\langle \linFunctional^{\KLS}, \delta {\bf u} \right\rangle
\label{eqn:Shell_weak}
\end{eqnarray*}
\noindent for every $\delta {\bf u} \in \mathcal{V}^{\KLS}_{{\bf 0},0}$.
}
\right.
$$

Note that the bilinear form $a^{\KLS}(\cdot,\cdot)$ is symmetric and positive semi-definite, and the kernel consists of constant and linear functions that are the rigid-body modes of the shell. Furthermore, note that $a^{\KLS}(\cdot,\cdot)$ is coercive on $\mathcal{V}^{\KLS}_{{\bf 0},0}$ (i.e., the kernel of $\mathcal{T}^{\KLS}$) with respect to the induced norm (see \cite[Theorem 4.3-4]{Ciarlet2005}). The Lax-Milgram Theorem guarantees that Problem $(V^{\KLS})$ has a unique solution ${\bf u} \in \mathcal{V}^{\KLS}$ that depends continuously on the external loading $\linFunctional^{\KLS} \in \left( \mathcal{V}^{\KLS} \right)^*$ and the boundary data $\left( \applied{\bf u}, \applied{\midsurfRot}_n \right) \in \mathcal{Q}^{\KLS}$, $\applied{\bf T} \in \left[ L^2(\physBoundary_{N_1}) \right]^3$, $\left\{ \applied{\force}|_{C} \right\}_{C \in \cornerSet{N}} \in \mathbb{R}^{\#\cornerSet{N}}$, and $\applied{\moment}_{nn} \in L^2(\physBoundary_{N_2})$.

\begin{remark}
  Often when dealing with homogeneous boundary conditions, it is convenient to split the domain boundary in to four disjoint sets, i.e., $\Gamma = \Gamma_C \cup \Gamma_{SS} \cup \Gamma_S \cup \Gamma_F$, where $\Gamma_C$ is the clamped portion, $\Gamma_{SS}$ is the simply supported portion, $\Gamma_S$ is the symmetric portion, and $\Gamma_F$ is the free portion. Physically, these boundary segments are summarized in the following:
  \begin{equation*}
    \begin{array}{llllll}
      \text{(clamped)} & \hat{\bf u} = {\bf 0} ,& \hat{\theta}_n = 0 & & &\text{on} \ \Gamma_{C} := \Gamma_{D_1} \cap \Gamma_{D_2}\\
      \text{(simply supported)}&\hat{\bf u} = {\bf 0} ,& & & \hat{B}_{nn} = 0 &\text{on} \ \Gamma_{SS} := \Gamma_{D_1} \cap \Gamma_{N_2}\\
      \text{(symmetric)}& & \hat{\theta}_n = 0, &\hat{\bf T} = {\bf 0} &  &\text{on} \ \Gamma_{S} := \Gamma_{N_1} \cap \Gamma_{D_2}\\
      \text{(free)} & & & \hat{\bf T} = {\bf 0} ,& \hat{B}_{nn} = 0  &\text{on} \ \Gamma_{F} := \Gamma_{N_1} \cap \Gamma_{N_2}\\
    \end{array}
  \end{equation*}
  \label{remark:shellBCs}
\end{remark}

\begin{remark}
  \label{alternative_f_KLS}
  The linear functional $\linFunctional^{\KLS} \in \left( \mathcal{V}^{\KLS} \right)^*$ we employ in this section may be replaced by its more common definition
  \begin{equation*}
    \begin{aligned}
      \left\langle \linFunctional^{\KLS}, {\bf v} \right\rangle = \int_{\physDomain} \applied{\textbf{\forcing}} \cdot {\bf v} \ d \physDomain + \int_{\physBoundary_{N_1}} \applied{\bm{\tau}} \cdot {\bf v} \ d \physBoundary +
      \int_{\physBoundary_{N_1}} \applied{\moment}_{nt} \midsurfRot_t({\bf v}) \ d \physBoundary + \int_{\physBoundary_{N_2}} \applied{\moment}_{nn} \midsurfRot_n({\bf v}) \ d \physBoundary
    \end{aligned}
  \end{equation*}
  without changing the solution of the Kirchhoff-Love shell problem. This is because both linear functionals return the same result when acting on ${\bf v} \in \mathcal{V}^{\KLS}_{{\bf 0},0}$.  However, they do not return the same result when acting on arbitrary ${\bf v} \in \mathcal{V}^{\KLS}$.  In fact, it turns out that Assumption \ref{assumption1} from our abstract framework does not hold for the above definition of $\linFunctional^{\KLS} \in \left( \mathcal{V}^{\KLS} \right)^*$ since the transverse shearing and twisting moment are not energetically conjugate to the boundary displacement and, hence, are not the Lagrange multiplier fields associated with enforcing the displacement boundary condition.  Instead, the ersatz traction and corner forces are the Lagrange multiplier fields associated with enforcing the displacement boundary condition.
\end{remark}


\subsection{A Generalized Green's Identity}
Now that we have stated a suitable variational formulation for the Kirchhoff-Love shell problem, we present a generalized Green's identity to be used later in constructing Nitsche's method. We first have the following lemma regarding performing integration by parts along a manifold.

\begin{lemma}[Green's Theorems for In-Plane Vector and Tensor Fields on Manifolds]
  \label{lemma:gen_greens_man}
  Let $\phi$ be a differentiable scalar field, $\fullTens{v} = \surfVec{v} + v_3 \aVec^3$ be a differentiable vector field, and $\surfTens{M}$ be a differentiable in-plane tensor field. Then
  \begin{equation*}
  \begin{aligned}
  \int_{\physDomain} \surfVec{\nabla} \phi \cdot \surfVec{v} \ d \physDomain = \int_{\physBoundary} \phi \left( \surfVec{v} \cdot \surfVec{\bdyNormal} \right) \ d \physBoundary - \int_{\physDomain} \phi \left( \surfVec{\nabla} \cdot \surfVec{v} \right) \ d \physDomain
  \end{aligned}
\end{equation*}
and
  \begin{equation*}
  \begin{aligned}
  \int_{\physDomain} \left( \surfVec{\nabla} \fullTens{\midsurfDispTest} \right) \colon \surfTens{M} \ d \physDomain = \underbrace{\int_{\physBoundary} \surfVec{\midsurfDispTest} \cdot \surfTens{M} \cdot \surfVec{\bdyNormal} \ d \physBoundary - \int_{\physDomain} \surfVec{\midsurfDispTest} \cdot \left( \surfVec{\nabla} \cdot \surfTens{M} \right) \ d \physDomain}_{\text{in-plane}} - \underbrace{\int_{\physDomain} \outOfPlane{\midsurfDispTest} \left( \surfTens{M} \colon \surfTens{\bCurv} \right) \ d \physDomain}_{\text{out-of-plane}}.
  \end{aligned}
\end{equation*}

\begin{proof}
  By the product rule, we can write
  \begin{equation*}
    \int_{\physDomain} \surfVec{\nabla} \cdot \left( \phi \ \surfVec{v} \right) \ d \physDomain = \int_{\physDomain} \surfVec{\nabla} \phi \cdot \surfVec{v} \ d \physDomain + \int_{\physDomain} \phi \left( \surfVec{\nabla} \cdot \surfVec{v} \right) \ d \physDomain
  \end{equation*}
  and by the divergence theorem, it follows that
  \begin{equation*}
    \int_{\physDomain} \surfVec{\nabla} \cdot \left( \phi \ \surfVec{v} \right) \ d \physDomain = \int_{\physBoundary} \phi \left( \surfVec{v} \cdot \surfVec{\bdyNormal} \right) \ d \physBoundary.
  \end{equation*}
  Combining these two expressions yields the first result.

  To establish the second result, we begin by expressing the vector field as $\fullTens{v} = \surfVec{v} + v_3 \aVec^3$ and observe that
  \begin{equation*}
    \int_{\physDomain} \left( \surfVec{\nabla} \fullTens{\midsurfDispTest} \right) \colon \surfTens{M} \ d \physDomain = \int_{\physDomain} \left( \surfVec{\nabla} \  \surfVec{\midsurfDispTest} \right) \colon \surfTens{M} \ d \physDomain + \int_{\physDomain} \left( \surfVec{\nabla} \outOfPlane{\midsurfDispTest} \otimes \aVec^3 \right) \colon \surfTens{M} \ d \physDomain + \int_{\physDomain} v_3 \left( \surfVec{\nabla} \aVec^3 \right) \colon \surfTens{M} \ d \physDomain.
  \end{equation*}
  The second-to-last integral in the above expression vanishes by the orthogonality between $\aVec^{3}$ and $(\aVec^{1},\aVec^{2})$ and the last integral can be rewritten as
  \begin{equation*}
    \int_{\physDomain} v_3 \left( \surfVec{\nabla} \aVec^3 \right) \colon \surfTens{M} \ d \physDomain = - \int_{\physDomain} \outOfPlane{\midsurfDispTest} \left( \surfTens{M} \colon \surfTens{\bCurv} \right) \ d \physDomain
  \end{equation*}
  by the relationship $\surfVec{\nabla} \aVec^{3} = - \surfTens{\bCurv}$, ultimately arriving at the out-of-plane expression in the result of the Green's identity. By the product rule, it follows that
  \begin{equation*}
    \int_{\physDomain} \surfVec{\nabla} \cdot \left( \surfVec{\midsurfDispTest} \cdot \surfTens{M} \right) \ d \physDomain = \int_{\physDomain} \left( \surfVec{\nabla} \ \surfVec{\midsurfDispTest} \right) \colon \surfTens{M} \ d \physDomain + \int_{\physDomain}  \surfVec{\midsurfDispTest} \cdot \left( \surfVec{\nabla} \cdot \surfTens{M} \right) \ d \physDomain,
  \end{equation*}
  and by the divergence theorem, it follows that
  \begin{equation*}
    \int_{\physDomain} \surfVec{\nabla} \cdot \left( \surfVec{\midsurfDispTest} \cdot \surfTens{M} \right) \ d \physDomain = \int_{\physBoundary} \surfVec{\midsurfDispTest} \cdot \surfTens{M} \cdot \surfVec{\bdyNormal} \ d \physBoundary.
  \end{equation*}
  Combining these two expressions yields the in-plane result of the Green's identity.
  \end{proof}
\end{lemma}

With the ability to perform vector integration by parts along manifolds, we are ready to state and prove our generalized Green's identity for the Kirchhoff-Love shell. Let
\begin{equation}
  \tilde{\mathcal{V}}^{\KLS} := \left\{ {\bf v} = \surfVec{v} + v_3 \aVec^3 \ \colon \ (v_1,v_2) \in \left[ H^2(\Omega) \right]^2 \hspace{5pt} \text{and} \hspace{5pt} v_3 \in H^4(\physDomain) \right\}
  \label{eqn:V_tilde_KLS}
\end{equation}
and note that $\tilde{\mathcal{V}}^{\KLS} \subset \mathcal{V}^{\KLS}$ is indeed a subspace by Sobolev embedding \cite{EvansPDEs}. Recall from Remark~\ref{remark:V_tilde} that $\tilde{\mathcal{V}}^{\KLS}$ is more regular than what is ultimately required for discretization, a point to be addressed in the next subsection. Then the following generalized Green's identity holds for the Kirchhoff-Love shell:

\begin{lemma}[Generalized Green's Identity for the Kirchhoff-Love Shell]
  For ${\bf w} \in \tilde{\mathcal{V}}^{\KLS}$ and ${\bf v} \in \mathcal{V}^{\KLS}$, the following Green's identity holds:
  \begin{equation}
    \begin{aligned}
      a^{\KLS}&({\bf w},{\bf v}) = \\
      &\clipbox{-2 0 495 0}{$\underbrace{ \int_{\physDomain} \surfVec{\midsurfDispTest} \cdot \left[ \surfVec{\nabla} \cdot \left( \surfTens{\SFF} \cdot \surfTens{\bendStress}({\bf w}) \right) + \left( \surfVec{\nabla} \cdot \surfTens{\bendStress}({\bf w}) \right) \cdot \surfTens{\SFF} - \surfVec{\nabla} \cdot \surfTens{\memStress}({\bf w}) \right]  \ d \physDomain + \int_{\physDomain} \outOfPlane{\midsurfDispTest} \left[ \surfTens{\bendStress}({\bf w}) \colon \surfTens{\TFF} - \surfVec{\nabla} \cdot \left( \surfTens{\Projector} \cdot \left( \surfVec{\nabla} \cdot \surfTens{\bendStress}({\bf w}) \right) \right) - \surfTens{\memStress}({\bf w}) \colon \surfTens{\SFF} \right] \ d \physDomain \hspace{50em}}$} \\
      &\phantom{=} \clipbox{10 0 -2 0}{$\underbrace{ \hspace{1em} + \int_{\physBoundary_{N_2}} \moment_{nn}({\bf w}) \midsurfRot_n({\bf v}) \ d \physBoundary + \sum_{C \in \cornerSet{N}} \left. \left( \llbracket \moment_{nt}({\bf w}) \rrbracket \outOfPlane{v} \right) \right|_{C} + \int_{\physBoundary_{N_1}} \fullTens{\midsurfDispTest} \cdot {\bf T}({\bf w}) \ d \physBoundary }_{\displaystyle \langle \mathcal{L}^{\KLS} {\bf w}, {\bf v} \rangle}$}\\
      &\phantom{=} + \underbrace{ \int_{\physBoundary_{D_2}} \moment_{nn}({\bf w}) \midsurfRot_n(\fullTens{\midsurfDispTest}) \ d \physBoundary + \sum_{C \in \cornerSet{D}} \left. \left( \llbracket \moment_{nt}({\bf w}) \rrbracket \outOfPlane{v} \right) \right|_{C} + \int_{\physBoundary_{D_1}} \fullTens{\midsurfDispTest} \cdot {\bf T}({\bf w}) \ d \physBoundary }_{ \displaystyle \langle \mathcal{B}^{\KLS} {\bf w}, \mathcal{T}^{\KLS} {\bf v} \rangle},
    \end{aligned}
    \label{eqn:Greens_ID_KLS}
  \end{equation}
  where $\surfTens{c} = \surfTens{b} \cdot \surfTens{b}$ is the third fundamental form of $\physDomain$, $\moment_{nn}({\bf w}) = \surfVec{n} \cdot \surfTens{\bendStress}({\bf w})\cdot \surfVec{n}$ is the \textbf{bending moment}, $\moment_{nt}({\bf w}) = \surfVec{n} \cdot \surfTens{\bendStress}({\bf w})\cdot \surfVec{t}$ is the \textbf{twisting moment}, and
  \begin{equation}
    {\bf T}({\bf w}) = \underbrace{ \overbrace{ \surfTens{\memStress}({\bf w}) \cdot \surfVec{\bdyNormal} }^{\surfVec{\ersatz}^{(\memStress)}({\bf w})} \overbrace{ - \surfTens{\SFF} \cdot \left( \surfTens{\bendStress}({\bf w}) \cdot \surfVec{\bdyNormal} + \surfVec{\bdyTangent} \moment_{nt}({\bf w}) \right) }^{\surfVec{\ersatz}^{(\bendStress)}({\bf w})} }_{\surfVec{\ersatz}({\bf w})} + \underbrace{ \left[ \left( \surfVec{\nabla} \cdot \surfTens{\bendStress}({\bf w}) \right) \cdot \surfVec{\bdyNormal} + \frac{\partial \moment_{nt}({\bf w})}{\partial t} \right] }_{\outOfPlane{\ersatz}({\bf w})} \aVec^{3}
    \label{eqn:KL_Shell_ersatz}
  \end{equation}
  is the \textbf{ersatz force}.  Moreover, the solution ${\bf u}$ of Problem $(V^{\KLS})$ satisfies $\mathcal{L}^{\KLS} {\bf u} = \linFunctional^{\KLS}$ provided the problem parameters are smooth enough that ${\bf u} \in \tilde{\mathcal{V}}^{\KLS}$.

  \begin{proof}
    The Green's identity follows immediately by one application of reverse integration by parts on the membrane contribution and two applications of reverse integration by parts on the bending contribution through the results of Lemma~\ref{lemma:gen_greens_man}. Beginning with the membrane portion of the variational form, we have
    \begin{equation*}
      \begin{aligned}
        \int_{\physDomain} \surfTens{\memStress}({\bf w}) \colon \surfTens{\memStrain}({\bf v}) \ d \physDomain &= \int_{\physDomain} \surfTens{\memStress}({\bf w}) \colon \Sym \left( \surfVec{\nabla} {\bf v} \right) \ d \physDomain\\
        &= \int_{\physBoundary} \surfVec{v} \cdot \surfTens{\memStress}({\bf w}) \cdot \surfVec{\bdyNormal}  \ d \physBoundary - \int_{\physDomain} \surfVec{v} \cdot \left( \surfVec{\nabla} \cdot \surfTens{\memStress}({\bf w}) \right) \ d \physDomain - \int_{\physDomain} v_3 \left( \surfTens{\memStress}({\bf w}) \colon \surfTens{\bCurv} \right)  \ d \physDomain.
      \end{aligned}
    \end{equation*}
    For the bending portion of the variational form, we begin by employing the decomposition for the bending strain found in Table~\ref{table:VariousStrains}, in particular,
    \begin{equation*}
        \int_{\physDomain} \surfTens{\bendStress}({\bf w}) \colon \surfTens{\bendStrain}({\bf v}) \ d \physDomain = - \int_{\physDomain} \left( \surfTens{b} \cdot \surfTens{\bendStress}({\bf w}) \right) \colon \Sym \left( \surfVec{\nabla} {\bf v} \right) \ d \physDomain + \int_{\physDomain} \surfTens{\bendStress}({\bf w}) \colon \Sym \left( \surfVec{\nabla} \ \surfVec{\theta}({\bf v}) \right) \ d \physDomain.
    \end{equation*}
    We handle each of these integrals individually. Beginning with the first, we apply the results of Lemma~\ref{lemma:gen_greens_man} once to obtain
    \begin{equation*}
        \int_{\physDomain} \left( \surfTens{b} \cdot \surfTens{\bendStress}({\bf w}) \right) \colon \Sym \left( \surfVec{\nabla} {\bf v} \right) \ d \physDomain = \int_{\physBoundary} \surfVec{v} \cdot \left( \surfTens{b} \cdot \surfTens{\bendStress}({\bf w}) \right) \cdot \surfVec{n} \ d \physBoundary - \int_{\physDomain} \surfVec{v} \cdot \left[ \surfVec{\nabla} \cdot \left( \surfTens{b} \cdot \surfTens{\bendStress}({\bf w}) \right) \right] \ d \physDomain - \int_{\physDomain} v_3 \left( \surfTens{\bendStress}({\bf w}) \colon \surfTens{c} \right) \ d \physDomain.
    \end{equation*}
    The second integral proceeds as follows:
    \begin{equation*}
      \begin{aligned}
        \int_{\physDomain} \surfTens{\bendStress}({\bf w}) \colon \Sym \left( \surfVec{\nabla} \ \surfVec{\theta}({\bf v}) \right) \ d \physDomain &= \int_{\physBoundary} \bendStress_{nn}({\bf w}) \theta_n({\bf v}) \ d \physBoundary + \int_{\physBoundary} \bendStress_{nt}({\bf w}) \theta_t({\bf v}) \ d \physBoundary + \int_{\physDomain} \left( \surfVec{\nabla} \cdot \surfTens{\bendStress}({\bf w}) \right) \cdot \left( \aVec_3 \cdot \surfVec{\nabla} {\bf v} \right) \ d \physDomain\\
        &= \int_{\physBoundary} \bendStress_{nn}({\bf w}) \theta_n({\bf v}) \ d \physBoundary + \int_{\physBoundary} \bendStress_{nt}({\bf w}) \theta_t({\bf v}) \ d \physBoundary + \int_{\physDomain}  \left( \surfVec{\nabla} \cdot \surfTens{\bendStress}({\bf w}) \right) \cdot \surfVec{\nabla} v_3 \ d \physDomain \\
        &\phantom{=} - \int_{\physDomain} \left( \surfVec{\nabla} \cdot \surfTens{\bendStress}({\bf w}) \right) \cdot \left( {\bf v} \cdot \surfVec{\nabla} \aVec_3 \right) \ d \physDomain\\
        &= \int_{\physBoundary} \bendStress_{nn}({\bf w}) \theta_n({\bf v}) \ d \physBoundary + \int_{\physBoundary} v_3 \frac{\partial \bendStress_{nt}({\bf w})}{\partial t} \ d \physBoundary + \sum_{C \in \cornerSet{}} \left. \left( \llbracket \moment_{nt}({\bf w}) \rrbracket \outOfPlane{v} \right) \right|_{C} - \int_{\physBoundary} \bendStress_{nt}({\bf w}) \left( \surfVec{t} \cdot \surfTens{b} \cdot \surfVec{v} \right) \ d \physBoundary\\
        &\phantom{=} + \int_{\physDomain} \left( \surfTens{\Projector} \cdot \left( \surfVec{\nabla} \cdot \surfTens{\bendStress}({\bf w}) \right) \right) \cdot \surfVec{\nabla} v_3 \ d \physDomain + \int_{\physDomain} \surfVec{v} \cdot \left( \left( \surfVec{\nabla} \cdot \surfTens{\bendStress}({\bf w}) \right) \cdot \surfTens{b} \right) \ d \physDomain \\
        &= \int_{\physBoundary} \bendStress_{nn}({\bf w}) \theta_n({\bf v}) \ d \physBoundary + \int_{\physBoundary} v_3 \frac{\partial \bendStress_{nt}({\bf w})}{\partial t} \ d \physBoundary + \sum_{C \in \cornerSet{}} \left. \left( \llbracket \moment_{nt}({\bf w}) \rrbracket \outOfPlane{v} \right) \right|_{C} - \int_{\physBoundary} \bendStress_{nt}({\bf w}) \left( \surfVec{t} \cdot \surfTens{b} \cdot \surfVec{v} \right) \ d \physBoundary\\
        &\phantom{=} + \int_{\physBoundary} v_3 \left[ \surfVec{n} \cdot \left( \surfVec{\nabla} \cdot \surfTens{\bendStress}({\bf w}) \right) \right] \ d \physBoundary - \int_{\physDomain} v_3 \left[ \surfVec{\nabla} \cdot \left( \surfTens{\Projector} \cdot \left( \surfVec{\nabla} \cdot \surfTens{\bendStress}({\bf w}) \right) \right) \right] \ d \physDomain\\
        &\phantom{=} + \int_{\physDomain} \surfVec{v} \cdot \left( \left( \surfVec{\nabla} \cdot \surfTens{\bendStress}({\bf w}) \right) \cdot \surfTens{b} \right) \ d \physDomain
      \end{aligned}
    \end{equation*}
    for ${\bf w} \in \tilde{\mathcal{V}}^{\KLS}$ and ${\bf v} \in \mathcal{V}^{\KLS}$. Note that $\surfTens{\Projector}$ arises in the above relations to make explicit that the tensor contraction between $\surfVec{\nabla} \cdot \surfTens{\bendStress}({\bf w})$ and ${\bf v} \cdot \surfVec{\nabla} \aVec_3$ only contains in-plane quantities. In general, $\surfVec{\nabla} \cdot \surfTens{\bendStress}({\bf w})$ contains both in-plane and out-of-plane components and, in order to apply Lemma~\ref{lemma:gen_greens_man} correctly, $\surfTens{\Projector}$ is needed.

    Combining these relationships for membrane and bending contributions, utilizing the definition of the ersatz forces \eqref{eqn:KL_Shell_ersatz}, and splitting the boundary along the 1- and 2-portions yields the presented Green's identity for the linearized Kirchhoff-Love shell. Note that all integrals present in this relationship are well defined since ${\bf T}({\bf w}) \in \left( L^2(\physBoundary_{D_1}) \right)^3$ and $\moment_{nn}({\bf w}) \in L^2(\physBoundary_{D_2})$ by the Trace theorem for Sobolev spaces.

    Now suppose that the problem parameters are sufficiently smooth such that ${\bf u} \in \tilde{\mathcal{V}}^{\KLS}$. We can then write
    \begin{equation}
      \begin{aligned}
        0 &= \left\langle \linFunctional^{\KLS}, \delta {\bf u} \right\rangle - a^{\KLS}({\bf u}, \delta {\bf u}) \\
          &= \int_{\physDomain} \delta \surfVec{u} \cdot \left( \applied{\surfVec{\forcing}} - \surfVec{\nabla} \cdot \left( \surfTens{b} \cdot \surfTens{\bendStress}({\bf u}) \right) - \left( \surfVec{\nabla} \cdot \surfTens{\bendStress}({\bf u}) \right) \cdot \surfTens{b} + \surfVec{\nabla} \cdot \surfTens{\memStress}({\bf u}) \right) \ d \physDomain + \int_{\physBoundary_{N_1}} \delta \surfVec{u} \cdot \left( \applied{\surfVec{\ersatz}} - \surfVec{\ersatz}({\bf u}) \right) \ d \physBoundary \\
          &\phantom{=} + \int_{\physDomain} \delta u_3 \left( \applied{\forcing}_3 - \surfTens{\bendStress}({\bf u}) \colon \surfTens{c} + \surfVec{\nabla} \cdot \left( \surfTens{\Projector} \cdot \left( \surfVec{\nabla} \cdot \surfTens{\bendStress}({\bf u}) \right) \right) + \surfTens{\memStress}({\bf u}) \colon \surfTens{b} \right) \ d \physDomain + \int_{\physBoundary_{N_1}} \delta u_3 \left( \applied{\ersatz}_3 - \ersatz_3({\bf u}) \right) \ d \physBoundary\\
          &\phantom{=} + \int_{\physBoundary_{N_2}} \left( \applied{\moment}_{nn} - \moment_{nn}({\bf u}) \right) \theta_n(\delta {\bf u}) \ d \physBoundary + \sum_{C \in \cornerSet{N}} \left( \left( \applied{\force} - \llbracket \moment_{nt}({\bf u}) \rrbracket \right) \delta u_3 \right) \Big|_{C} \label{eq:g_inv_1}
      \end{aligned}
    \end{equation}
    for all $\delta {\bf u} \in \mathcal{V}^{\KLS}_{{\bf 0},0}$ as a consequence of the generalized Green's identity.  Since $\left( C^\infty_0 (\physDomain) \right)^3 \subset \mathcal{V}^{\KLS}_{{\bf 0},0}$,
    \begin{equation*}
      \begin{aligned}
        0 &= \int_{\physDomain} \delta \surfVec{u} \cdot \left( \applied{\surfVec{\forcing}} - \surfVec{\nabla} \cdot \left( \surfTens{b} \cdot \surfTens{\bendStress}({\bf u}) \right) - \left( \surfVec{\nabla} \cdot \surfTens{\bendStress}({\bf u}) \right) \cdot \surfTens{b} + \surfVec{\nabla} \cdot \surfTens{\memStress}({\bf u}) \right) \ d \physDomain\\
        &\phantom{=} + \int_{\physDomain} \delta u_3 \left( \applied{\forcing}_3 - \surfTens{\bendStress}({\bf u}) \colon \surfTens{c} + \surfVec{\nabla} \cdot \left( \surfTens{\Projector} \cdot \left( \surfVec{\nabla} \cdot \surfTens{\bendStress}({\bf u}) \right) \right) + \surfTens{\memStress}({\bf u}) \colon \surfTens{b} \right) \ d \physDomain
      \end{aligned}
    \end{equation*}
    for all infinitely smooth compact test functions $\delta {\bf u} \in \left( C^\infty_0 (\physDomain) \right)^3$. Since $\left( C^\infty_0 (\physDomain) \right)^3$ is dense in $\left( L^2(\physDomain) \right)^3$, $\applied{\bf f} \in \left( L^2(\physDomain) \right)^3$, $\left[ \surfVec{\nabla} \cdot \left( \surfTens{b} \cdot \surfTens{\bendStress}({\bf u}) \right) + \left( \surfVec{\nabla} \cdot \surfTens{\bendStress}({\bf u}) \right) \cdot \surfTens{b} - \surfVec{\nabla} \cdot \surfTens{\memStress}({\bf u}) \right] \in  \left( L^2(\physDomain) \right)^3$, and $\surfTens{\bendStress}({\bf u}) \colon \surfTens{c} - \surfVec{\nabla} \cdot \left( \surfTens{\Projector} \cdot \left( \surfVec{\nabla} \cdot \surfTens{\bendStress}({\bf u}) \right) \right) - \surfTens{\memStress}({\bf u}) \colon \surfTens{b} \in L^2(\physDomain)$, it follows that
    \begin{equation}
      \applied{\surfVec{\forcing}} = \surfTens{\Projector} \cdot \left[ \surfVec{\nabla} \cdot \left( \surfTens{b} \cdot \surfTens{\bendStress}({\bf u}) \right) + \left( \surfVec{\nabla} \cdot \surfTens{\bendStress}({\bf u}) \right) \cdot \surfTens{b} - \surfVec{\nabla} \cdot \surfTens{\memStress}({\bf u}) \right] \label{eq:g_inv_2}
    \end{equation}
    and
    \begin{equation}
      \applied{\forcing}_3 = \surfTens{\bendStress}({\bf u}) \colon \surfTens{c} - \surfVec{\nabla} \cdot \left( \surfTens{\Projector} \cdot \left( \surfVec{\nabla} \cdot \surfTens{\bendStress}({\bf u}) \right) \right) - \surfTens{\memStress}({\bf u}) \colon \surfTens{b} \label{eq:g_inv_3}
    \end{equation}
    almost everywhere in $\physDomain$.  Inserting \eqref{eq:g_inv_2} and \eqref{eq:g_inv_3} into \eqref{eq:g_inv_1} in turn yields
    \begin{equation}
      \begin{aligned}
        0 &= \int_{\physBoundary_{N_1}} \delta {\bf u} \cdot \left( \applied{\bf T} - {\bf T}({\bf u}) \right) \ d \physBoundary + \int_{\physBoundary_{N_2}} \left( \applied{\moment}_{nn} - \moment_{nn}({\bf u}) \right) \theta_n(\delta {\bf u}) \ d \physBoundary + \sum_{C \in \cornerSet{N}} \left( \left( \applied{\force} - \llbracket \moment_{nt}({\bf u}) \rrbracket \right) \delta u_3 \right) \Big|_{C} \label{eq:g_inv_4}
      \end{aligned}
    \end{equation}
    for all $\delta {\bf u} \in \mathcal{V}^{\KLS}_{{\bf 0},0}$.

    To proceed, let
    \begin{equation*}
    \mathcal{Q}^{\KLS}_1 := \left\{ {\bf q} \in \left( L^{2}\left(\physBoundary_{N_1}\right) \right)^3: \left( \mathscr{E}_1^{\KLS} q_1, \mathscr{E}_1^{\KLS} q_2 \right) \in \left(H^{1/2}(\physBoundary)\right)^2,  \mathscr{E}_1^{\KLS} q_3 \in H^{3/2}(\physBoundary), \textup{ and } q_3 |_{\chi_N} = 0 \right\}
    \end{equation*}
    where $\mathscr{E}_1^{\KLS} : L^2(\physBoundary_{N_1}) \rightarrow L^2(\physBoundary)$ is an extension-by-zero operator.  By the surjectivity of the trace operator, we can define a linear and bounded lifting operator $\mathscr{L}_1^{\KLS} : \mathcal{Q}^{\KLS}_1 \rightarrow \mathcal{V}^{\KLS}_{{\bf 0},0}$ such that $\mathscr{L}_1^{\KLS} {\bf q} |_{\physBoundary} = \left(\mathscr{E}_1^{\KLS} q_i\right) \aVec^i$ and $\theta_n(\mathscr{L}_1^{\KLS} {\bf q})|_{\physBoundary} = 0$ for all ${\bf q} \in \mathcal{Q}^{\KLS}_1$.  Then, for ${\bf q} \in \mathcal{Q}^{\KLS}_1$, we can choose $\delta {\bf u} = \mathscr{L}_1^{\KLS} {\bf q}$ in \eqref{eq:g_inv_4}, yielding
    \begin{equation*}
      \begin{aligned}
        0 &= \int_{\physBoundary_{N_1}} {\bf q} \cdot \left( \applied{\bf T} - {\bf T}({\bf u}) \right) \ d \physBoundary.
      \end{aligned}
    \end{equation*}
    Since $\mathcal{Q}^{\KLS}_1$ is dense in $\left( L^{2}\left(\physBoundary_{N_1}\right) \right)^3$, $\applied{\bf T} \in \left( L^{2}\left(\physBoundary_{N_1}\right) \right)^3$, and ${\bf T}({\bf u}) \in \left( L^{2}\left(\physBoundary_{N_1}\right) \right)^3$, it follows that
    \begin{equation}
    \applied{\bf T} = {\bf T}({\bf u}) \label{eq:g_inv_5}
    \end{equation}
    almost everywhere on $\Gamma_{N_1}$.  Next, let
    \begin{equation*}
    \mathcal{Q}^{\KLS}_2 := \left\{ q \in L^2\left(\physBoundary_{N_2}\right): \mathscr{E}_2^{\KLS} q \in H^{1/2}(\physBoundary) \right\}
    \end{equation*}
    where $\mathscr{E}_2^{\KLS} : L^2\left(\physBoundary_{N_2}\right) \rightarrow L^2(\physBoundary)$ is an extension-by-zero operator.  By the surjectivity of the trace operator, we can define a linear and bounded lifting operator $\mathscr{L}_2^{\KLS} : \mathcal{Q}^{\KLS}_2 \rightarrow \mathcal{V}^{\KLS}_{{\bf 0},0}$ such that $\mathscr{L}_2^{\KLS} q |_{\physBoundary} = {\bf 0}$ and $\theta_n(\mathscr{L}_1^{\KLS} {\bf q})|_{\physBoundary} = \mathscr{E}_2^{\KLS} q$ for all $q \in \mathcal{Q}^{\KLS}_2$.  Then, for $q \in \mathcal{Q}^{\KLS}_2$, we can choose $\delta {\bf u} = \mathscr{L}_2^{\KLS} q$ in \eqref{eq:g_inv_4}, yielding
    \begin{equation*}
      \begin{aligned}
        0 &= \int_{\physBoundary_{N_2}} \left( \applied{\moment}_{nn} - \moment_{nn}({\bf u}) \right) q \ d \physBoundary.
      \end{aligned}
    \end{equation*}
    Since $\mathcal{Q}^{\KLS}_2$ is dense in $L^{2}\left(\physBoundary_{N_2}\right)$, $\applied{\moment}_{nn} \in L^2\left(\physBoundary_{N_2}\right)$, and $\moment_{nn}({\bf u}) \in L^2\left(\physBoundary_{N_2}\right)$, it follows that
    \begin{equation}
    \applied{\moment}_{nn} = \moment_{nn}({\bf u}) \label{eq:g_inv_6}
    \end{equation}
    almost everywhere on $\Gamma_{N_2}$.  Finally, inserting \eqref{eq:g_inv_5} and \eqref{eq:g_inv_6} into \eqref{eq:g_inv_4} yields
    \begin{equation*}
    \begin{aligned}
        0 &= \sum_{C \in \cornerSet{N}} \left( \left( \applied{\force} - \llbracket \moment_{nt}({\bf u}) \rrbracket \right) \delta u_3 \right) \Big|_{C}
      \end{aligned}
    \end{equation*}
    for all $\delta {\bf u} \in \mathcal{V}^{\KLS}_{{\bf 0},0}$.  For each $C \in \cornerSet{N}$, there exists a $\delta {\bf u} \in \mathcal{V}^{\KLS}_{{\bf 0},0}$ such that $\delta u_3|_C = 1$ and $\delta u_3|_{C'} = 0$ for $C' \in \cornerSet{N}$ such that $C' \neq C$.  It follows that
    \begin{equation}
    \applied{\force} = \llbracket \moment_{nt}({\bf u}) \rrbracket \label{eq:g_inv_7}
    \end{equation}
    on $\cornerSet{N}$.
    Combining \eqref{eq:g_inv_2}, \eqref{eq:g_inv_3}, \eqref{eq:g_inv_5}, \eqref{eq:g_inv_6}, and \eqref{eq:g_inv_7} yields the desired result that $\mathcal{L}^{\KLS} {\bf u} = \linFunctional^{\KLS}$.

  \end{proof}
    \label{lemma:Greens_ID_KLS}
\end{lemma}

\begin{remark}
  The Euler-Lagrange equations of Problem $(V^{\KLS})$ give rise to the following strong formulation:
  $$
  (S^{\KLS}) \left\{ \hspace{5pt}
  \parbox{6.00in}{
  \noindent \textup{Find ${\bf u} \colon \overline{\physDomain} \rightarrow \R^3$ such that:}
  \begin{equation*}
    \begin{aligned}
      \begin{array}{rll}
        \surfTens{\Projector} \cdot \left[ \surfVec{\nabla} \cdot \left( \surfTens{b} \cdot \surfTens{\bendStress}({\bf u}) \right) + \left( \surfVec{\nabla} \cdot \surfTens{\bendStress}({\bf u}) \right) \cdot \surfTens{b} - \surfVec{\nabla} \cdot \surfTens{\memStress}({\bf u}) \right] &= \surfVec{\applied{\forcing}} \hspace{10pt} &\text{in} \ \physDomain\\
        \surfTens{\bendStress}({\bf u}) \colon \surfTens{c} - \surfVec{\nabla} \cdot \left( \surfTens{\Projector} \cdot \left( \surfVec{\nabla} \cdot \surfTens{\bendStress}({\bf u}) \right) \right) - \surfTens{\memStress}({\bf u}) \colon \surfTens{b} & = \applied{\forcing}_3 \hspace{10pt} &\text{in} \ \physDomain\\
        {\bf u} &= \applied{\bf u} \hspace{10pt} &\text{on} \ \physBoundary_{D_1}\\
        \theta_n({\bf u}) &= \applied{\theta}_n \hspace{10pt} &\text{on} \ \physBoundary_{D_2}\\
        {\bf T}({\bf u}) &= \applied{\bf T} \hspace{10pt} &\text{on} \ \physBoundary_{N_1}\\
        \moment_{nn}({\bf u}) &= \applied{\moment}_{nn} \hspace{10pt} &\text{on} \ \physBoundary_{N_2}\\
        \llbracket \moment_{nt}({\bf u}) \rrbracket &= \applied{\force} &\text{on} \ \chi_N.
      \end{array}
    \end{aligned}
    \label{eqn:KLS_Strong}
  \end{equation*}
  }
  \right.
  $$
  \noindent This result follows immediately from the relationship $\mathcal{L}^{\KLS} {\bf u} = \linFunctional^{\KLS}$ that was proved in Lemma~\ref{lemma:Greens_ID_KLS}.
\end{remark}

\begin{remark}
  \label{rem:IncorrectErsatz}
The Euler-Lagrange equations presented above differ from those commonly presented in the literature, for example, from those presented in \cite[p.155]{Ciarlet2005} and in \cite[p.156]{Koiter1973foundations}. In those references, the in-plane bending contribution to the ersatz force is reported to be (in our notation) ``$-2 \surfTens{\SFF} \cdot \surfTens{\bendStress}({\bf w}) \cdot \surfVec{\bdyNormal}$'', which does not agree when compared with our derived forces in \eqref{eqn:KL_Shell_ersatz}, i.e., $\surfVec{\ersatz}^{(\bendStress)}({\bf w})$. However, we believe that the ones presented here are correct for several reasons.  First of all, the Euler-Lagrange equations presented here derive directly from the Green's identity presented in Lemma \ref{lemma:Greens_ID_KLS}.  Furthermore, we later use the Euler-Lagrange equations presented here to derive the required applied forces, tractions, and bending moments for a set of manufactured solutions.  These manufactured solutions are then employed to numerically confirm convergence rates for our proposed Nitsche formulation in conjunction with an isogeometric Kirchhoff-Love shell discretization.  By contrast, when using the equations in \cite{Ciarlet2005} to derive applied forces, tractions, and bending moments, we do not see convergence in the corresponding numerical results to the manufactured solutions. Although the origin of the erroneous term is unclear, we have traced several references back to Koiter's early work \cite[(3.10)]{Koiter1970foundation} which does not include the full derivation. Later work by Koiter and his student, van der Heijden, \cite[p.20]{van1976modified} states that these incorrect boundary terms arise ``after fairly lengthy algebra'' and cites a paper listed in the references section as ``to be published''. As such, we have been unable to trace exactly where the algebra leading to the incorrect result went awry.
\end{remark}

\begin{remark}
The decomposition of the in-plane ersatz force into membrane and bending contributions, i.e.,  $\surfVec{\ersatz}^{(\memStress)}({\bf u})$ and $\surfVec{\ersatz}^{(\bendStress)}({\bf u})$, respectively, is done for later convenience in order to establish trace inequality and penalty constants that are independent of thickness.
\end{remark}


\subsection{Generalized Trace and Cauchy-Schwarz Inequalities}

With a Green's identity in place, we are ready to provide generalized trace and Cauchy-Schwarz inequalities satisfying Assumption~\ref{assumption2}, the final pieces required before presenting Nitsche's method for the Kirchhoff-Love shell. We establish a mesh $\mathcal{K}$ of non-overlapping (mapped) polygons, which we refer to henceforth as elements, associated with $\physDomain$ that is comprised of elements such that $\physDomain = \text{int}(\overline{\cup_{K \in \mathcal{K}} K})$. Next, we assume that the approximation space $\mathcal{V}^{\KLS}_h$ consists of (at least) $C^1$-continuous piecewise polynomial or rational approximations over the mesh $\mathcal{K}$. For each element $K \in \mathcal{K}$, we associate an element size $h_K = \text{diam}(K)$, and we associate with the entire mesh $\mathcal{K}$ a mesh size $h = \max_{K \in \mathcal{K}} h_K$. We collect the boundary edges into an edge mesh $\mathcal{E}$. In the case of the Kirchhoff-Love shell, we must construct two additional edge meshes, $\mathcal{E}_{D_1}$ and $\mathcal{E}_{D_2}$. We associate the members of $\mathcal{E}_{D_1}$ with elements whose edges belong to $\physBoundary_{D_1}$ and likewise for members of $\mathcal{E}_{D_2}$, i.e., for $\alpha = 1,2$,
\begin{equation*}
 \mathcal{E}_{D_\alpha} = \left\{ E \in \mathcal{E} \colon E \subset \physBoundary_{D_\alpha} \right\}.
\end{equation*}
To ensure that each edge in $\mathcal{E}$ belongs to either the Neumann or Dirichlet boundaries, assume that $\physBoundary_{D_\alpha} = \text{int}(\overline{\cup_{E \in \mathcal{E}_{D_\alpha}} E})$ for $\alpha = 1,2$.  We associate an edge size $h_E = h_K$ for each edge $E \in \mathcal{E}$, where $K \in \mathcal{K}$ is the element for which $E$ is the edge. This is not the only size we can associate with the edge, but it is the simplest.  For anisotropic meshes, other prescriptions may be more appropriate (see, e.g., \cite{bazilevs2007weak}).  Note that when it is necessary to differentiate between edges in $\mathcal{E}_{D_1}$ and $\mathcal{E}_{D_2}$, we will introduce subscripts on the edge variable, e.g., $E_1 \in \mathcal{E}_1$. Lastly, since each $C \in \cornerSet{}$ is associated with an element $K \in \mathcal{K}$, we define $h_C = h_K$. With these definitions in place, we have the following lemma:

\begin{lemma}[Trace Inequalities]
  There exists five positive, dimensionless constants $\cTrace{,1}^{\KLS}, \cTrace{,2}^{\KLS}, \cTrace{,3}^{\KLS}, \cTrace{,4}^{\KLS}, \cTrace{,5}^{\KLS} > 0$ such that
  \begin{equation}
    \sum_{E_1 \in \mathcal{E}_{D_1}} \int_{E_1} \frac{h_{E_1}^3}{\cTrace{,1}^{\KLS} \thickness^3 | \mathbb{\constitutive} |} \left( \ersatz_3({\bf v}_h) \right)^2 \ d \physBoundary \le \frac{1}{5} a^{\KLS}({\bf v}_h,{\bf v}_h)
   \label{eqn:TI_KLS_1}
  \end{equation}
  \begin{equation}
    \sum_{C \in \cornerSet{D}} \frac{h_C^2}{\cTrace{,2}^{\KLS} \thickness^3 | \mathbb{\constitutive} |}  \llbracket \moment_{nt}({\bf v}_h) \rrbracket^2 \Big|_C \le \frac{1}{5} a^{\KLS}({\bf v}_h,{\bf v}_h)
     \label{eqn:TI_KLS_2}
  \end{equation}
  \begin{equation}
   \sum_{E_2 \in \mathcal{E}_{D_2}} \int_{E_2} \frac{h_{E_2}}{\cTrace{,3}^{\KLS} \thickness^3 | \mathbb{\constitutive} |} \left( \moment_{nn}({\bf v}_h) \right)^2 \ d \physBoundary \le \frac{1}{5} a^{\KLS}({\bf v}_h,{\bf v}_h)
   \label{eqn:TI_KLS_3}
  \end{equation}
  \begin{equation}
    \sum_{E_1 \in \mathcal{E}_{D_1}} \int_{E_1} \frac{h_{E_1}}{\cTrace{,4}^{\KLS} \thickness^3 | \mathbb{\constitutive} |} \left| \surfVec{\ersatz}^{(\bendStress)}({\bf v}_h) \right|^2 \ d \physBoundary \le \frac{1}{5} a^{\KLS}({\bf v}_h,{\bf v}_h)
    \label{eqn:TI_KLS_4}
  \end{equation}
  \begin{equation}
    \sum_{E_1 \in \mathcal{E}_{D_1}} \int_{E_1} \frac{h_{E_1}}{\cTrace{,5}^{\KLS} \thickness | \mathbb{\constitutive} |} \left| \surfVec{\ersatz}^{(\memStress)}({\bf v}_h) \right|^2 \ d \physBoundary \le \frac{1}{5} a^{\KLS}({\bf v}_h,{\bf v}_h)
     \label{eqn:TI_KLS_5}
  \end{equation}
  for all ${\bf v}_h \in \mathcal{V}^{\KLS}_h$.
  \begin{proof}
      We prove \eqref{eqn:TI_KLS_1} and remark that the proofs for \eqref{eqn:TI_KLS_2}, \eqref{eqn:TI_KLS_3}, \eqref{eqn:TI_KLS_4}, and \eqref{eqn:TI_KLS_5} follow in an identical manner.

      We begin by denoting the space of rigid body modes associated with the Kirchhoff-Love shell by
      \begin{equation*}
        \text{Rig}^{\KLS}(\Omega) \colon = \left\{ {\bf v}_h \in \mathcal{V}_h^{\KLS} \colon \surfTens{\memStrain}({\bf v}_h) = \surfTens{\bendStrain}({\bf v}_h) = \surfTens{0} \right\}\footnote{The nomenclature for this space refers to the fact that it contains the \textbf{\emph{rigid body modes}} associated with the various strain tensors.}.
      \end{equation*}
      We then denote the orthogonal complement of this space by
      \begin{equation}
      \mathring{\mathcal{V}}^{\KLS}_h \colon = \left\{ {\bf v} \in \mathcal{V}_h^{\KLS} \colon ({\bf v}, {\bf r})_{L^2} = 0 \ \forall \ {\bf r} \in \text{Rig}^{\KLS}(\Omega) \right\}.
      \end{equation}
      Since $\text{Rig}^{\KLS}(\Omega)$ is the kernel of $\memStrain$ and $\bendStrain$, it follows that $\memStrain\left( \mathcal{V}_h^{\KLS} \right) = \memStrain\left( \mathring{\mathcal{V}}^{\KLS}_h \right)$ and $\bendStrain\left( \mathcal{V}_h^{\KLS} \right) = \bendStrain\left( \mathring{\mathcal{V}}^{\KLS}_h \right)$; hence, for any ${\bf v}_h \in \mathcal{V}^{\KLS}_h$, there exists $\mathring{{\bf v}}_h \in \mathring{\mathcal{V}}^{\KLS}_h$ such that $\surfTens{\memStrain}({\bf v}_h) = \surfTens{\memStrain}(\mathring{\bf v}_h)$ and $\surfTens{\bendStrain}({\bf v}_h) = \surfTens{\bendStrain}(\mathring{\bf v}_h)$. Consequently, if there exists a positive dimensionless constant $\cTrace{,1}^{\KLS} > 0$ such that \eqref{eqn:TI_KLS_1} holds for all ${\bf v}_h \in \mathring{\mathcal{V}}^{\KLS}_h$, then \eqref{eqn:TI_KLS_1} holds with the same constant $\cTrace{,1}^{\KLS}$ for all ${\bf v}_h \in \mathcal{V}^{\KLS}_h$.

      Now consider the generalized eigenproblem: Find $({\bf u}_h, \lambda_h) \in \mathring{\mathcal{V}}^{\KLS}_h \times \mathbb{R}$ such that
      \begin{equation}
        \sum_{E_1 \in \mathcal{E}_{D_1}} \int_{E_1} \frac{h_{E_1}^3}{ \thickness^3 | \mathbb{\constitutive} |} \ersatz_3({\bf v}_h) \ersatz_3(\delta {\bf v}_h) \ d \physBoundary = \lambda_h a^{\KLS}({\bf v}_h, \delta {\bf v}_h)
      \label{eq:Gen_EP_KLS}
      \end{equation}
      for all $\delta {\bf u}_h \in \mathring{\mathcal{V}}^{\KLS}_h$.  Since the bilinear form $a^{\KLS}({\bf v}_h, \delta {\bf v}_h)$ is coercive on $\mathring{\mathcal{V}}^{\KLS}_h$, all eigenvalues of the above generalized eigenproblem are non-negative and finite, and they are finite in number.  Moreover, the min-max theorem states that the max eigenvalue satisfies
      \begin{equation}
      \lambda^{\KLS}_{\textup{max}} = \sup_{\substack{ {\bf v}_h \in \mathring{\mathcal{V}}^{\KLS}_h \\ {\bf v}_h \neq {\bf 0} }} \frac{ \sum_{E_1 \in \mathcal{E}_{D_1}} \int_{E_1} \frac{h_{E_1}^3}{ \thickness^3 | \mathbb{\constitutive} |} \left( \ersatz_3({\bf v}_h) \right)^2 \ d \physBoundary}{a^{\KLS}({\bf v}_h, {\bf v}_h)}. \nonumber
      \end{equation}
        It is easily seen then that the lemma is satisfied for $\cTrace{,1}^{\KLS} = \lambda^{\KLS}_{\textup{max}}/5$.
  \end{proof}
  \label{lemma:TI_KLS}
\end{lemma}

\begin{remark}
  From its proof, we see that Lemma~\ref{lemma:TI_KLS} is satisfied for $C^{\KLS}_{\textup{tr},1} = \lambda^{\KLS}_{\textup{max},1}/5$, where $ \lambda^{\KLS}_{\textup{max},1}$ is the largest eigenvalue of the generalized eigenproblem \eqref{eq:Gen_EP_KLS}. Unfortunately, it is very difficult to construct a basis for the space $\mathring{\mathcal{V}}^{\KLS}_h$.  Fortunately, $\lambda^{\KLS}_{\textup{max}}$ is also the largest finite eigenvalue of this simpler generalized eigenproblem: Find $({\bf u}_h, \lambda_h) \in \mathcal{V}^{\KLS}_h \times \mathbb{R}$ such that
 \begin{equation*}
 \sum_{E_1 \in \mathcal{E}_{D_1}} \int_{E_1} \frac{h_{E_1}^3}{ \thickness^3 | \mathbb{\constitutive} |} \ersatz_3({\bf u}_h) \ersatz_3(\delta {\bf u}_h) \ d \physBoundary = \lambda_h a^{\KLS}({\bf u}_h, \delta {\bf u}_h)
 \end{equation*}
 for all $\delta {\bf u}_h \in \mathcal{V}^{\KLS}_h$.  Given a basis $\{ N_i {\bf e}_j \}_{i=1}^n$ and $j = 1,2,3$ for the space $\mathcal{V}^{\KLS}_h$, it then follows that $\lambda^{\KLS}_{\textup{max}}$ may be computed as the largest finite eigenvalue of the generalized matrix eigenproblem $\left({\bf A} - \lambda {\bf B} \right) {\bf x} = {\bf 0}$, where
 \begin{equation*}
 \left[ {\bf A} \right]_{ij} = \sum_{E_1 \in \mathcal{E}_{D_1}} \int_{E_1} \frac{h_{E_1}^3}{\thickness^3 | \mathbb{\constitutive} |} \ersatz_3({\bf u}_h) \ersatz_3(\delta {\bf u}_h) \ d \physBoundary
 \end{equation*}
 and
 \begin{equation*}
 \left[ {\bf B} \right]_{ij} = a^{\KLS}({\bf u}_h, \delta {\bf u}_h).
 \end{equation*}
 Thus, it is tractable to compute an explicit value for the trace constant $\cTrace{,1}^{\KLS}$. In a similar vein, $C^{\KLS}_{\textup{tr},2} = \lambda^{\KLS}_{\textup{max},2}/5$, $C^{\KLS}_{\textup{tr},3} = \lambda^{\KLS}_{\textup{max},3}/5$, $C^{\KLS}_{\textup{tr},4} = \lambda^{\KLS}_{\textup{max},4}/5$, and $C^{\KLS}_{\textup{tr},5} = \lambda^{\KLS}_{\textup{max},5}/5$, where $\lambda^{\KLS}_{\textup{max},i}$, for $i = 1,2,\ldots,5$, correspond to the largest finite eigenvalues of generalized eigenproblems derived from \eqref{eqn:TI_KLS_1}, \eqref{eqn:TI_KLS_2}, \eqref{eqn:TI_KLS_3}, \eqref{eqn:TI_KLS_4}, and \eqref{eqn:TI_KLS_5}, respectively. The associated eigenproblems for these constants can likewise be constructed and solved for explicitly.
\end{remark}

To construct Nitsche's method for the Kirchhoff-Love shell, we must specify suitable linear maps $\epsilon^{\KLS}$ and $\eta^{\KLS}$ such that the generalized trace and Cauchy-Schwarz inequalities appearing in Assumption~\ref{assumption2} are satisfied. We begin by extending the domain of definition of the boundary operator $\mathcal{B}^{\KLS} \colon \tilde{\mathcal{V}}^{\KLS} \rightarrow \left( \mathcal{Q}^{\KLS} \right)^*$, defined in \eqref{eqn:Greens_ID_KLS}, to the enlarged space $\tilde{\mathcal{V}}^{\KLS} + \mathcal{V}^{\KLS}_h$. We accomplish this by expressing this boundary operator as a summation of integrals and point evaluations over the edge meshes and corners, rather than as a single integration and function evaluation in the continuous setting. In particular,
\begin{equation}
  \begin{aligned}
  \left\langle \mathcal{B}^{\KLS} {\bf w}, \mathcal{T}^{\KLS} {\bf v} \right\rangle &= \int_{\physBoundary_1} {\bf T}({\bf w}) \cdot \fullTens{\midsurfDispTest} \ d \physBoundary + \sum_{C \in \cornerSet{D}} \left( \llbracket \moment_{nt}({\bf w}) \rrbracket \outOfPlane{v} \right) \Big|_{C} + \int_{\physBoundary_2} \moment_{nn}({\bf w}) \midsurfRot_n(\fullTens{\midsurfDispTest}) \ d \physBoundary\\
  &= \sum_{E_1 \in \mathcal{E}_{D_1}} \int_{E_1} {\bf T}({\bf w}) \cdot \fullTens{\midsurfDispTest} \ d \physBoundary + \sum_{C \in \cornerSet{D}} \left( \llbracket \moment_{nt}({\bf w}) \rrbracket \outOfPlane{v} \right) \Big|_{C} + \sum_{E_2 \in \mathcal{E}_{D_2}} \int_{E_2} \moment_{nn}({\bf w}) \midsurfRot_n(\fullTens{\midsurfDispTest}) \ d \physBoundary
\end{aligned}
  \label{eqn:bdy_integrals_KLS}
\end{equation}
for all ${\bf w} \in \tilde{\mathcal{V}}^{\KLS}$ and ${\bf v} \in \mathcal{V}^{\KLS}$.

Expressing the duality pairing in this manner permits a trivial extension of the domain of definition of $\mathcal{B}^{\KLS}$ to the enlarged space $\tilde{\mathcal{V}}^{\KLS} + \mathcal{V}^{\KLS}_h$. This extension is necessary because the first set of integrals in the above equation may not be well defined for arbitrary ${\bf w} \in \mathcal{V}_h^{\KLS}$ because the Kirchhoff-Love shell requires third derivatives along the boundary rendering low continuity discretizations inadmissible. However, the second set of integrals is well defined for any piecewise $C^1$-continuous polynomial or rational approximation over the mesh $\mathcal{K}$ since these types of discretizations are $C^\infty$ over each edge.

Next, we define the linear map $\epsilon^{\KLS}: \textup{dom}(\epsilon^{\KLS}) \subseteq \left( \mathcal{Q}^{\KLS} \right)^* \rightarrow \mathcal{Q}^{\KLS}$ through its action:
\begin{equation*}
  \begin{aligned}
    \left\langle \left( \epsilon^{\KLS} \right)^{-1} {\bf w}, {\bf v} \right\rangle &= \thickness^3 | \mathbb{\constitutive} | \left( \sum_{E_1 \in \mathcal{E}_{D_1}} \int_{E_1} \frac{\cPen{,1}^{\KLS}}{h^3_{E_1}} w_3 v_3 \ d \physBoundary + \sum_{C \in \cornerSet{D}} \frac{\cPen{,2}^{\KLS}}{h^2_{C}} ( w_3 v_3 ) \Big|_{C} + \sum_{E_2 \in \mathcal{E}_{D_2}} \int_{E_2} \frac{\cPen{,3}^{\KLS}}{h_{E_2}} \theta_n({\bf w}) \theta_n({\bf v}) \ d \physBoundary \right)\\
    &\phantom{=} + \sum_{E_1 \in \mathcal{E}_{D_1}} \int_{E_1} \frac{\cPen{,4}^{\KLS} \thickness | \mathbb{\constitutive} |}{h_{E_1}} \surfVec{w} \cdot \surfVec{v} \ d \physBoundary
  \end{aligned}
\end{equation*}
for all ${\bf w}, {\bf v} \in \mathcal{Q}^{\KLS}$, where $\cPen{,1}^{\KLS} > \cTrace{,1}^{\KLS}$, $\cPen{,2}^{\KLS} > \cTrace{,2}^{\KLS}$, $\cPen{,3}^{\KLS} > \cTrace{,3}^{\KLS}$, and $\cPen{,4}^{\KLS} > \cTrace{,4}^{\KLS} + \cTrace{,5}^{\KLS}$ are positive dimensionless constants.

\begin{remark}
The choice of penalty constants presented here in this paper is not the only stable choice.  For user-specified dimensionless constants $\alpha_1 > 0$, $\alpha_2 > 0$, $\alpha_3 > 0$, $\alpha_4 > 0$, and $\alpha_5 > 0$, we can alternatively select $\cPen{,1}^{\KLS} > \alpha_1 \cTrace{}^{\KLS}$, $\cPen{,2}^{\KLS} > \alpha_2 \cTrace{}^{\KLS}$, $\cPen{,3}^{\KLS} > \alpha_3 \cTrace{}^{\KLS}$, and $\cPen{,4}^{\KLS} > \left( \alpha_4 + \alpha_5 \right) \cTrace{}^{\KLS}$, where $\cTrace{}^{\KLS} > 0$ is a dimensionless constant such that
  \begin{equation*}
  \begin{aligned}
    \frac{1}{\thickness^3 | \mathbb{\constitutive} |} &\left( \sum_{E_1 \in \mathcal{E}_{D_1}} \int_{E_1} \frac{h_{E_1}^3}{\alpha_1} \ersatz_3({\bf w}) \ersatz_3({\bf v}) \ d \physBoundary + \sum_{C \in \cornerSet{D}} \frac{h_C^2}{\alpha_2} ( \llbracket \moment_{nt}({\bf w}) \rrbracket \llbracket \moment_{nt}({\bf v}) \rrbracket ) \Big|_C + \sum_{E_2 \in \mathcal{E}_{D_2}} \int_{E_2} \frac{h_{E_2}}{\alpha_3} \moment_{nn}({\bf w}) \moment_{nn}({\bf v}) \ d \physBoundary \right. \\
    &\phantom{=} \left. + \sum_{E_1 \in \mathcal{E}_{D_1}} \int_{E_1} \frac{h_{E_1}}{\alpha_4} \surfVec{\ersatz}^{(\bendStress)}({\bf w}) \cdot \surfVec{\ersatz}^{(\bendStress)}({\bf v}) \ d \physBoundary \right) + \sum_{E_1 \in \mathcal{E}_{D_1}} \int_{E_1} \frac{h_{E_1}}{\thickness | \mathbb{\constitutive} | \alpha_5} \surfVec{\ersatz}^{(\memStress)}({\bf w}) \cdot \surfVec{\ersatz}^{(\memStress)}({\bf v}) \ d \physBoundary \le \cTrace{}^{\KLS} a^{\KLS}({\bf v}_h, {\bf v}_h)
  \end{aligned}
  \end{equation*}
  for all ${\bf v}_{h} \in \mathcal{V}^{\KLS}_h$.  The advantage of this approach is that only one trace constant, namely, $\cTrace{}^{\KLS}$, must be estimated.  The disadvantage of this approach is that $\alpha_1$, $\alpha_2$, $\alpha_3$, $\alpha_4$, and $\alpha_5$, which control the relative weightings of the out-of-plane displacement boundary condition along $\Gamma_{D_1}$, the displacement boundary condition at corners in $\cornerSet{D}$, the rotation boundary conditions along $\Gamma_{D_2}$, and the in-plane displacement boundary conditions along $\Gamma_{D_1}$, respectively, must be specified.
\label{remark:eig_KLS}
\end{remark}

Let $\eta^{\KLS} \colon \text{dom}(\eta^{\KLS}) \subseteq \left( \mathcal{Q}^{\KLS} \right)^* \rightarrow \mathcal{Q}^{\KLS}$ be a densely defined, positive, self-adjoint linear map that is defined on the enlarged space
\begin{equation*}
  \left\{ \mathcal{B}^{\KLS} {\bf v} \colon {\bf v} \in \tilde{\mathcal{V}}^{\KLS} + \mathcal{V}^{\KLS}_h \right\}
\end{equation*}
and satisfies
\begin{equation*}
  \begin{aligned}
    \left\langle \mathcal{B}^{\KLS} {\bf w}, \eta^{\KLS} \mathcal{B}^{\KLS} {\bf v} \right\rangle &=  \frac{1}{\thickness^3 | \mathbb{\constitutive} |} \left( \sum_{E_1 \in \mathcal{E}_{D_1}} \int_{E_1} \frac{h_{E_1}^3}{\cTrace{,1}^{\KLS}} \ersatz_3({\bf w}) \ersatz_3({\bf v}) \ d \physBoundary + \sum_{C \in \cornerSet{D}} \frac{h_C^2}{\cTrace{,2}^{\KLS}} ( \llbracket \moment_{nt}({\bf w}) \rrbracket \llbracket \moment_{nt}({\bf v}) \rrbracket ) \Big|_C \right.\\
    &\phantom{=} \left. + \sum_{E_2 \in \mathcal{E}_{D_2}} \int_{E_2} \frac{h_{E_2}}{\cTrace{,3}^{\KLS}} \moment_{nn}({\bf w}) \moment_{nn}({\bf v}) \ d \physBoundary + \sum_{E_1 \in \mathcal{E}_{D_1}} \int_{E_1} \frac{h_{E_1}}{\cTrace{,4}^{\KLS}} \surfVec{\ersatz}^{(\bendStress)}({\bf w}) \cdot \surfVec{\ersatz}^{(\bendStress)}({\bf v}) \ d \physBoundary \right)\\
    &\phantom{=} + \sum_{E_1 \in \mathcal{E}_{D_1}} \int_{E_1} \frac{h_{E_1}}{\thickness | \mathbb{\constitutive} | \cTrace{,5}^{\KLS}} \surfVec{\ersatz}^{(\memStress)}({\bf w}) \cdot \surfVec{\ersatz}^{(\memStress)}({\bf v}) \ d \physBoundary
  \end{aligned}
\end{equation*}
for all ${\bf w}, {\bf v} \in \tilde{\mathcal{V}}^{\KLS} + \mathcal{V}^{\KLS}_h$. After these choices of linear maps have been made, the generalized trace and Cauchy-Schwarz inequalities appearing in Assumption~\ref{assumption2} are satisfied.

\begin{lemma}[Generalized Trace Inequality for the Kirchhoff-Love Shell]
  It holds that
  \begin{equation*}
    \left\langle \mathcal{B}^{\KLS} {\bf v}_h, \eta^{\KLS} \mathcal{B}^{\KLS} {\bf v}_h \right\rangle \leq a^{\KLS}({\bf v}_h, {\bf v}_h)
  \end{equation*}
  for all ${\bf v}_h \in \mathcal{V}^{\KLS}_h$.
  \begin{proof}
    The proof follows immediately from Lemma~\ref{lemma:TI_KLS} and the definition of $\eta^{\KLS}$.
  \end{proof}
  \label{lemma:TI_KLS_gen}
\end{lemma}

\begin{lemma}[Generalized Cauchy-Schwarz Inequality for the Kirchhoff-Love Shell]
  Let $\cPen{,1}^{\KLS} = \CSconstant_1^2 \cTrace{,1}^{\KLS}$, $\cPen{,2}^{\KLS} = \CSconstant_2^2 \cTrace{,2}^{\KLS}$, $\cPen{,3}^{\KLS} = \CSconstant_3^2 \cTrace{,3}^{\KLS}$, and $\cPen{,4}^{\KLS} = \CSconstant_4^2 \max(\cTrace{,4}^{\KLS},\cTrace{,5}^{\KLS})$, where $\CSconstant_1, \CSconstant_2, \CSconstant_3, \CSconstant_4 \in (1,\infty)$. Then
  \begin{equation*}
    \left| \left\langle \mathcal{B}^{\KLS} {\bf v}, \mathcal{T}^{\KLS} {\bf w} \right\rangle \right| \le \frac{1}{\CSconstant} \left\langle \mathcal{B}^{\KLS} {\bf v}, \eta^{\KLS} \mathcal{B}^{\KLS} {\bf v} \right\rangle^{1/2} \left\langle \left( \epsilon^{\KLS} \right)^{-1} \mathcal{T}^{\KLS} {\bf w}, \mathcal{T}^{\KLS} {\bf w} \right\rangle^{1/2}
  \end{equation*}
  for all ${\bf v}, {\bf w} \in \tilde{\mathcal{V}}^{\KLS} + \mathcal{V}^{\KLS}_h$, where $\CSconstant = \min(\CSconstant_1,\CSconstant_2,\CSconstant_3,\CSconstant_4)$.
  \begin{proof}
    Recall \eqref{eqn:bdy_integrals_KLS} and the ersatz force decomposition presented in \eqref{eqn:KL_Shell_ersatz}. We then write
    \begin{equation}
      \begin{aligned}
        \left\langle \mathcal{B}^{\KLS} {\bf w}, \mathcal{T}^{\KLS} {\bf v} \right\rangle &= \sum_{E_1 \in \mathcal{E}_{D_1}} \int_{E_1} \ersatz_3({\bf w}) v_3 \ d \physBoundary + \sum_{C \in \cornerSet{D}} \left( \llbracket \moment_{nt}({\bf w}) \rrbracket \outOfPlane{v} \right) \Big|_{C} + \sum_{E_2 \in \mathcal{E}_{D_2}} \int_{E_2} \moment_{nn}({\bf w}) \midsurfRot_n(\fullTens{\midsurfDispTest}) \ d \physBoundary\\
        &\phantom{=} + \sum_{E_1 \in \mathcal{E}_{D_1}} \int_{E_1} \surfVec{\ersatz}^{(\bendStress)}({\bf w}) \cdot \surfVec{v} \ d \physBoundary + \sum_{E_1 \in \mathcal{E}_{D_1}} \int_{E_1} \surfVec{\ersatz}^{(\memStress)}({\bf w}) \cdot \surfVec{v} \ d \physBoundary.
      \end{aligned}
      \label{eqn:KLS_BvTw_CS_proof}
    \end{equation}
    We individually bound these five terms in \eqref{eqn:KLS_BvTw_CS_proof} by utilizing standard continuous ($(f, g)_{L^2(D)} \leq \| f \|_{L^2(D)} \| g \|_{L^2(D)}$ for $f, g \in L^2(D)$) and discrete ($|(x,y)| \leq \| x \|_2 \| y \|_2$ for $x, y \in \mathbb{R}^n$) Cauchy-Schwarz inequalities. The first term is bounded according to
    \begin{equation*}
      \sum_{E_1 \in \mathcal{E}_{D_1}} \int_{E_1} \ersatz_3({\bf w}) v_3 \ d \physBoundary \le \frac{1}{\CSconstant_1} \left( \sum_{E_1 \in \mathcal{E}_{D_1}} \int_{E_1} \frac{h_{E_1}^3}{\cTrace{,1}^{\KLS} \thickness^3 | \mathbb{\constitutive} |} \outOfPlane{\ersatz}({\bf w}) \ d \physBoundary \right)^{1/2} \left( \sum_{E_1 \in \mathcal{E}_{D_1}} \int_{E_1} \frac{\cPen{,1}^{\KLS} \thickness^3 | \mathbb{\constitutive} |}{h_{E_1}^3} v_{3} \ d \physBoundary \right)^{1/2}.
    \end{equation*}
    The second term is bounded according to the following relationship:
    \begin{equation*}
      \sum_{C \in \cornerSet{D}} \left( \llbracket \moment_{nt}({\bf w}) \rrbracket \outOfPlane{v} \right) \Big|_{C} \le \frac{1}{\CSconstant_2} \left( \sum_{C \in \cornerSet{D}} \frac{h_C^2}{\cTrace{,2}^{\KLS} \thickness^3 | \mathbb{\constitutive} |} \llbracket \moment_{nt}({\bf w}) \rrbracket^2 \Big|_{C} \right)^{1/2} \left( \sum_{C \in \cornerSet{D}} \frac{\cPen{,2}^{\KLS} \thickness^3 | \mathbb{\constitutive} |}{h_C^2} ( \outOfPlane{v} )^2 \Big|_{C} \right)^{1/2}.
    \end{equation*}
    The third term is bounded according to
    \begin{equation*}
      \begin{aligned}
        \sum_{E_2 \in \mathcal{E}_{D_2}} \int_{E_2} \moment_{nn}({\bf w}) \midsurfRot_n({\bf v}) \ d \physBoundary &\le \frac{1}{\CSconstant_3} \left( \sum_{E_2 \in \mathcal{E}_{D_2}} \int_{E_2} \frac{h_{E_2}}{\cTrace{,3}^{\KLS} \thickness^3 | \mathbb{\constitutive} |} ( \moment_{nn}({\bf w}) )^2 \ d \physBoundary \right)^{1/2} \left( \sum_{E_2 \in \mathcal{E}_{D_2}} \int_{E_2} \frac{\cPen{,3}^{\KLS} \thickness^3 | \mathbb{\constitutive} |}{h_{E_2}} ( \midsurfRot_n({\bf v}) )^2 \ d \physBoundary \right)^{1/2}.
      \end{aligned}
    \end{equation*}
    The fourth term is bounded according to
    \begin{equation*}
      \begin{aligned}
        \sum_{E_1 \in \mathcal{E}_{D_1}} \int_{E_1} \surfVec{\ersatz}^{(B)}({\bf w}) \cdot \surfVec{v} \ d \physBoundary &\le \frac{1}{\CSconstant_4} \left( \sum_{E_1 \in \mathcal{E}_{D_1}} \int_{E_1} \frac{h_{E_2}}{\cTrace{,4}^{\KLS} \thickness^3 | \mathbb{\constitutive} |} \left| \surfVec{\ersatz}^{(B)}({\bf w}) \right|^2 \ d \physBoundary \right)^{1/2} \left( \sum_{E_1 \in \mathcal{E}_{D_1}} \int_{E_1} \frac{\cPen{,4}^{\KLS} \thickness^3 | \mathbb{\constitutive} |}{h_{E_2}} \left| \surfVec{v} \right|^2 \ d \physBoundary \right)^{1/2}.
      \end{aligned}
    \end{equation*}
    and finally, the fifth term is bounded according to
    \begin{equation*}
      \begin{aligned}
        \sum_{E_1 \in \mathcal{E}_{D_1}} \int_{E_1} \surfVec{\ersatz}^{(A)}({\bf w}) \cdot \surfVec{v} \ d \physBoundary &\le \frac{1}{\CSconstant_4} \left( \sum_{E_1 \in \mathcal{E}_{D_1}} \int_{E_1} \frac{h_{E_2}}{\cTrace{,5}^{\KLS} \thickness | \mathbb{\constitutive} |} \left| \surfVec{\ersatz}^{(A)}({\bf w}) \right|^2 \ d \physBoundary \right)^{1/2} \left( \sum_{E_1 \in \mathcal{E}_{D_1}} \int_{E_1} \frac{\cPen{,4}^{\KLS} \thickness | \mathbb{\constitutive} |}{h_{E_2}} \left| \surfVec{v} \right|^2 \ d \physBoundary \right)^{1/2}.
      \end{aligned}
    \end{equation*}
    Combining these bounds with the bounds $1/\CSconstant_1, 1/\CSconstant_2, 1/\CSconstant_3, 1/\CSconstant_4 < 1/\CSconstant$, where $\CSconstant = \min(\CSconstant_1,\CSconstant_2,\CSconstant_3,\CSconstant_4)$, yields the following:
    \begin{equation*}
      \begin{aligned}
        \left\langle \mathcal{B}^{\KLS} {\bf v}, \mathcal{T}^{\KLS} {\bf w} \right\rangle &\le \frac{1}{\CSconstant_1} \left( \sum_{E_1 \in \mathcal{E}_{D_1}} \int_{E_1} \frac{h_{E_1}^3}{\cTrace{,1}^{\KLS} \thickness^3 | \mathbb{\constitutive} |} \outOfPlane{\ersatz}({\bf w}) \ d \physBoundary \right)^{1/2} \left( \sum_{E_1 \in \mathcal{E}_{D_1}} \int_{E_1} \frac{\cPen{,1}^{\KLS} \thickness^3 | \mathbb{\constitutive} |}{h_{E_1}^3} v_{3} \ d \physBoundary \right)^{1/2}\\
        &\phantom{\le} + \frac{1}{\CSconstant_2} \left( \sum_{C \in \cornerSet{D}} \frac{h_C^2}{\cTrace{,2}^{\KLS} \thickness^3 | \mathbb{\constitutive} |} \llbracket \moment_{nt}({\bf w}) \rrbracket^2 \Big|_{C} \right)^{1/2} \left( \sum_{C \in \cornerSet{D}} \frac{\cPen{,2}^{\KLS} \thickness^3 | \mathbb{\constitutive} |}{h_C^2} ( \outOfPlane{v} )^2 \Big|_{C} \right)^{1/2}\\
        &\phantom{\le} + \frac{1}{\CSconstant_3} \left( \sum_{E_2 \in \mathcal{E}_{D_2}} \int_{E_2} \frac{h_{E_2}}{\cTrace{,3}^{\KLS} \thickness^3 | \mathbb{\constitutive} |} ( \moment_{nn}({\bf w}) )^2 \ d \physBoundary \right)^{1/2} \left( \sum_{E_2 \in \mathcal{E}_{D_2}} \int_{E_2} \frac{\cPen{,3}^{\KLS} \thickness^3 | \mathbb{\constitutive} |}{h_{E_2}} ( \midsurfRot_n({\bf v}) )^2 \ d \physBoundary \right)^{1/2}\\
        &\phantom{\le} + \frac{1}{\CSconstant_4} \left( \sum_{E_1 \in \mathcal{E}_{D_1}} \int_{E_1} \frac{h_{E_2}}{\cTrace{,4}^{\KLS} \thickness^3 | \mathbb{\constitutive} |} \left| \surfVec{\ersatz}^{(B)}({\bf w}) \right|^2 \ d \physBoundary \right)^{1/2} \left( \sum_{E_1 \in \mathcal{E}_{D_1}} \int_{E_1} \frac{\cPen{,4}^{\KLS} \thickness^3 | \mathbb{\constitutive} |}{h_{E_2}} \left| \surfVec{v} \right|^2 \ d \physBoundary \right)^{1/2} \\
        &\phantom{\le} + \frac{1}{\CSconstant_4} \left( \sum_{E_1 \in \mathcal{E}_{D_1}} \int_{E_1} \frac{h_{E_2}}{\cTrace{,5}^{\KLS} \thickness | \mathbb{\constitutive} |} \left| \surfVec{\ersatz}^{(A)}({\bf w}) \right|^2 \ d \physBoundary \right)^{1/2} \left( \sum_{E_1 \in \mathcal{E}_{D_1}} \int_{E_1} \frac{\cPen{,4}^{\KLS} \thickness | \mathbb{\constitutive} |}{h_{E_2}} \left| \surfVec{v} \right|^2 \ d \physBoundary \right)^{1/2}.
      \end{aligned}
    \end{equation*}
    Taking absolute values of both sides followed by an application of the discrete Cauchy-Schwarz inequality yields the desired result.
  \end{proof}
  \label{lemma:CS_KLS}
\end{lemma}


\subsection{Nitsche's Method}

Following the abstract variational framework of Section~\ref{sec:Nitsche} and with the appropriate definitions of $\epsilon^{\KLS}$, $\eta^{\KLS}$, and $\mathcal{B}^{\KLS}$ in place, our Nitsche-based formulation for the Kirchhoff-Love shell is posed as:

\begin{mybox}[\emph{Nitsche's Method for the Kirchhoff-Love Shell}]
  \vspace{-7pt}
$$
(N^{\KLS}_h) \left\{ \hspace{5pt}
\parbox{6.00in}{
Given $\linFunctional^{\KLS} \in \left( \mathcal{V}^{\KLS} \right)^*$ and $\left( \hat{\bf u}, \hat{\theta}_n \right) \in \mathcal{Q}^{\KLS}$, find ${\bf u}_h \in \mathcal{V}^{\KLS}_h$ such that

\begin{equation*}
  \begin{aligned}
    a^{\KLS}_h({\bf u}_h, \delta {\bf u}_h) &= \underbrace{ \int_{\physDomain} \applied{\textbf{\forcing}} \cdot \delta {\bf u}_h \ d \physDomain + \int_{\physBoundary_{N_1}} \applied{\bf T} \cdot \delta {\bf u}_h \ d \physBoundary + \sum_{C \in \cornerSet{N}} \left. \left( \applied{\force} \delta u_{3,h} \right) \right|_{C} + \int_{\physBoundary_{N_2}} \applied{\moment}_{nn} \midsurfRot_n(\delta {\bf u}_h) \ d \physBoundary}_{\left\langle \linFunctional^{\KLS}, \delta {\bf u}_h \right\rangle}\\
    &\phantom{=} {\color{ForestGreen} \underbrace{ - \sum_{E_1 \in \mathcal{E}_{D_1}} \int_{E_1} {\bf T}(\delta {\bf u}_h) \cdot \applied{\bf u} \ d \physBoundary - \sum_{C \in \cornerSet{D}} \left( \llbracket \moment_{nt}(\delta {\bf u}_h) \rrbracket \applied{u}_3 \right) \Big|_{C} - \sum_{E_2 \in \mathcal{E}_{D_2}} \int_{E_2} \moment_{nn}(\delta {\bf u}_h) \applied{\midsurfRot}_n \ d \physBoundary}_{\text{Symmetry Terms}} }\\
    &\phantom{=} \clipbox{-2 0 400 0}{${\color{Orchid} \underbrace{ + \thickness^3 | \mathbb{\constitutive} | \left( \sum_{E_1 \in \mathcal{E}_{D_1}} \int_{E_1} \frac{\cPen{,1}^{\KLS}}{h^3_{E_1}} \delta u_{3,h} \applied{u}_{3} \ d \physBoundary + \sum_{C \in \cornerSet{D}} \frac{\cPen{,2}^{\KLS}}{h^2_{C}} ( \delta u_{3,h} \applied{u}_{3} ) \Big|_{C} \right. \hspace{40em} }}$} \\
    &\phantom{=} \hspace{2pt} \clipbox{10 0 -2 0}{$ {\color{Orchid} \underbrace{ \hspace{1em} \left. + \sum_{E_2 \in \mathcal{E}_{D_2}} \int_{E_2} \frac{\cPen{,3}^{\KLS}}{h_{E_2}} \theta_n(\delta {\bf u}_h) \applied{\theta}_n \ d \physBoundary \right) + \sum_{E_1 \in \mathcal{E}_{D_1}} \int_{E_1} \frac{\cPen{,4}^{\KLS} \thickness | \mathbb{\constitutive} |}{h_{E_1}} \delta \surfVec{u}_h \cdot \surfVec{\applied{u}} \ d \physBoundary}_{\text{Penalty Terms}} }$}
      \end{aligned}
\label{eqn:KL_Shell_Weak_Nitsche}
\end{equation*}
for every $\delta {\bf u}_h \in \mathcal{V}^{\KLS}_h$, where $a^{\KLS}_h \colon \left( \tilde{\mathcal{V}}^{\KLS} + \mathcal{V}^{\KLS}_h \right) \times \left( \tilde{\mathcal{V}}^{\KLS} + \mathcal{V}^{\KLS}_h \right) \rightarrow \mathbb{R}$ is the bilinear form defined by
\begin{equation*}
  \begin{aligned}
    a^{\KLS}_h({\bf u}_h, \delta {\bf u}_h) &= \underbrace{ \int_{\physDomain} \surfTens{\memStress}({\bf u}_h) \colon \surfTens{\memStrain}(\delta {\bf u}_h) \ d \physDomain + \int_{\physDomain} \surfTens{\bendStress}({\bf u}_h) \colon \surfTens{\bendStrain}(\delta {\bf u}_h) \ d \physDomain }_{a^{\KLS}({\bf u}_h, \delta {\bf u}_h)}\\
    &\phantom{=} {\color{Cerulean} \underbrace{ - \sum_{E_1 \in \mathcal{E}_{D_1}} \int_{E_1} {\bf T}({\bf u}_h) \cdot \delta {\bf u}_h \ d \physBoundary - \sum_{C \in \cornerSet{D}} \left( \llbracket \moment_{nt}({\bf u}_h) \rrbracket \delta u_{3,h} \right) \Big|_{C} - \sum_{E_2 \in \mathcal{E}_{D_2}} \int_{E_2} \moment_{nn}({\bf u}_h) \midsurfRot_n(\delta {\bf u}_h) \ d \physBoundary }_{\text{Consistency Terms}} }\\
    &\phantom{=} {\color{ForestGreen} \underbrace{ - \sum_{E_1 \in \mathcal{E}_{D_1}} \int_{E_1} {\bf T}(\delta {\bf u}_h) \cdot {\bf u}_h \ d \physBoundary - \sum_{C \in \cornerSet{D}} \left( \llbracket \moment_{nt}(\delta {\bf u}_h) \rrbracket u_{3,h} \right) \Big|_{C} - \sum_{E_2 \in \mathcal{E}_{D_2}} \int_{E_2} \moment_{nn}(\delta {\bf u}_h) \midsurfRot_n({\bf u}_h) \ d \physBoundary }_{\text{Symmetry Terms}} }\\
    &\phantom{=} \clipbox{-2 0 395 0}{${\color{Orchid} \underbrace{ + \thickness^3 | \mathbb{\constitutive} | \left( \sum_{E_1 \in \mathcal{E}_{D_1}} \int_{E_1} \frac{\cPen{,1}^{\KLS}}{h^3_{E_1}} \delta u_{3,h} u_{3,h} \ d \physBoundary + \sum_{C \in \cornerSet{D}} \frac{\cPen{,2}^{\KLS}}{h^2_{C}} ( \delta u_{3,h} u_{3,h} ) \Big|_{C} \right. \hspace{40em} }}$} \\
    &\phantom{=} \hspace{2pt} \clipbox{10 0 -2 0}{$ {\color{Orchid} \underbrace{ \hspace{1em} \left. + \sum_{E_2 \in \mathcal{E}_{D_2}} \int_{E_2} \frac{\cPen{,3}^{\KLS}}{h_{E_2}} \theta_n(\delta {\bf u}_h) \theta_n({\bf u}_h) \ d \physBoundary \right) + \sum_{E_1 \in \mathcal{E}_{D_1}} \int_{E_1} \frac{\cPen{,4}^{\KLS} \thickness | \mathbb{\constitutive} |}{h_{E_1}} \delta \surfVec{u}_h \cdot \surfVec{u}_h \ d \physBoundary}_{\text{Penalty Terms}} }$}.
  \end{aligned}
\end{equation*}
}
\right.
$$
\end{mybox}

Now that we have constructed a Nitsche-based formulation for the Kirchhoff-Love shell that satisfies Assumptions~\ref{assumption1} and~\ref{assumption2} according to Lemmas~\ref{lemma:Greens_ID_KLS},~\ref{lemma:TI_KLS_gen}, and~\ref{lemma:CS_KLS}, we have the following theorem stating well-posedness and an error estimate for our formulation.

\begin{theorem}[Well-Posedness and Error Estimate for the Kirchhoff-Love Shell]
  Let $\cPen{,1}^{\KLS} = \CSconstant_1^2 \cTrace{,1}^{\KLS}$, $\cPen{,2}^{\KLS} = \CSconstant_2^2 \cTrace{,2}^{\KLS}$, $\cPen{,3}^{\KLS} = \CSconstant_3^2 \cTrace{,3}^{\KLS}$, and $\cPen{,4}^{\KLS} = \CSconstant_4^2 \max(\cTrace{,4}^{\KLS}, \cTrace{,5}^{\KLS})$, where $\CSconstant_1, \CSconstant_2, \CSconstant_3, \CSconstant_4 \in (1,\infty)$. Then there exists a unique discrete solution ${\bf u} \in \mathcal{V}^{\KLS}_h$ to the Nitsche-based formulation of the Kirchhoff-Love shell Problem $(N_h^{\KLS})$. Moreover, if the continuous solution ${\bf u} \in \mathcal{V}^{\KLS}$ to Problem $(V^{\KLS})$ satisfies ${\bf u} \in \tilde{\mathcal{V}}^{\KLS}$, then the discrete solution ${\bf u}_h$ satisfies the error estimate
  \begin{equation*}
\vvvertiii{{\bf u} - {\bf u}_h}_{\KLS} \leq \left( 1+ \frac{2}{1-\frac{1}{\CSconstant}} \right) \min_{{\bf v}_h \in \mathcal{V}^{\KLS}_h} \vvvertiii{{\bf u} - {\bf v}_h}_{\KLS},
\label{eq:KLS_Error}
\end{equation*}
  where $\CSconstant = \min(\CSconstant_1,\CSconstant_2,\CSconstant_3,\CSconstant_4)$ and $\vvvertiii{\cdot}_{\KLS}: \tilde{\mathcal{V}}^{\KLS} + \mathcal{V}^{\KLS}_h \rightarrow \mathbb{R}$ is the energy norm defined by
\begin{equation*}
\vvvertiii{\bf v}_{\KLS}^2 := a^{\KLS}({\bf v},{\bf v}) + \left\langle \mathcal{B}^{\KLS} {\bf v}, \eta^{\KLS} \mathcal{B}^{\KLS} {\bf v} \right\rangle + 2 \left\langle \left( \epsilon^{\KLS}\right)^{-1} \mathcal{T}^{\KLS} {\bf v}, \mathcal{T}^{\KLS} {\bf v} \right\rangle.
\end{equation*}
  \begin{proof}
  Note that the presented Nitsche-based formulation for the Kirchhoff-Love shell precisely fits into the abstract variational framework presented in Section~\ref{sec:Nitsche} with $\mathcal{V} = \mathcal{V}^{\KLS}$, $\mathcal{Q} = \mathcal{Q}^{\KLS}$, $a(\cdot,\cdot) = a^{\KLS}(\cdot,\cdot)$, $\linFunctional = \linFunctional^{\KLS}$, $\mathcal{T} = \mathcal{T}^{\KLS}$, $\tilde{\mathcal{V}} = \tilde{\mathcal{V}}^{\KLS}$, $\mathcal{L} = \mathcal{L}^{\KLS}$, $\mathcal{B} = \mathcal{B}^{\KLS}$, $\mathcal{V}_h = \mathcal{V}^{\KLS}_h$, $\epsilon = \epsilon^{\KLS}$, and $\eta = \eta^{\KLS}$.  Moreover, Assumption~\ref{assumption1} of the abstract variational framework is satisfied due to Lemma~\ref{lemma:Greens_ID_KLS}, and Assumption~\ref{assumption2} is satisfied due to Lemmas~\ref{lemma:TI_KLS_gen} and~\ref{lemma:CS_KLS}.  Then well-posedness is a direct result of the Lax-Milgram theorem and coercivity and continuity as established by Lemmas~\ref{lemma:abstract_coercivity} and~\ref{lemma:abstract_continuity}, and the error estimate follows directly from Theorem~\ref{theorem:error_estimate}.
  \end{proof}
\label{thm:KLS_Error}
\end{theorem}

The above result indicates that our Nitsche-based formulation is quasi-optimal in the energy norm (in the sense that the error in the discrete solution is proportional to the best approximation error) when the continuous solution ${\bf u} \in \mathcal{V}^{\KLS}$ to Problem $(V^{\KLS})$ satisfies ${\bf u} \in \tilde{\mathcal{V}}^{\KLS}$.  However, the above result does not reveal the rates of convergence of the energy norm error, nor does it reveal the rates of convergence for other norms one may care about (for instance, the $L^2$-norm).

\begin{remark}
  The presented Nitsche-based formulation for the Kirchhoff-Love shell as well as the presented well-posedness and error estimate results are new to the best of our knowledge, though the presented formulation is quite similar to the formulations presented in \cite{guo2015weak,guo2015nitsche}.  However, the formulations diverge in two important ways.  First, the formulation presented in this paper includes corner forces, while the formulations presented in \cite{guo2015weak,guo2015nitsche} do not.  Second, following \cite[p.155]{Ciarlet2005} and \cite[p.156]{Koiter1973foundations}, the formulations presented in \cite{guo2015weak,guo2015nitsche} employ incorrect plane bending contributions to the ersatz force (see Remark \ref{rem:IncorrectErsatz}).  Consequently, the formulations presented in \cite{guo2015weak,guo2015nitsche} are actually variationally inconsistent and yield sub-optimal convergence rates when used with common boundary condition specifications.  We demonstrate this through numerical example later in Section \ref{sec:num_results}.
\end{remark}

\begin{remark}
  Note that, according to our analysis, a practitioner may select any $\CSconstant_1, \CSconstant_2, \CSconstant_3, \CSconstant_4 \in (1,\infty)$.  Generally speaking, Dirichlet boundary conditions are enforced more strongly for larger $\CSconstant_1, \CSconstant_2, \CSconstant_3, \CSconstant_4$ as opposed to smaller $\CSconstant_1, \CSconstant_2, \CSconstant_3, \CSconstant_4$.  However, the condition number of the linear system associated with Nitsche's method scales linearly with $\max\left( \CSconstant_1, \CSconstant_2, \CSconstant_3, \CSconstant_4 \right)$ \cite{juntunen2009nitsche} and, in certain circumstances, the discrete solution becomes over-constrained and boundary locking occurs as $\max\left( \CSconstant_1, \CSconstant_2, \CSconstant_3, \CSconstant_4 \right) \rightarrow \infty$, resulting in a loss of solution accuracy \cite{lew2008discontinuous}.  On the other hand, as $\min\left( \CSconstant_1, \CSconstant_2, \CSconstant_3, \CSconstant_4 \right) \rightarrow 1$, Lemma~\ref{lemma:abstract_coercivity} suggests that the linear system associated with Nitsche's method may lose definiteness, and Theorem~\ref{thm:KLS_Error} suggests that the energy norm error may blow up in the limit $\min\left( \CSconstant_1, \CSconstant_2, \CSconstant_3, \CSconstant_4 \right) \rightarrow 1$.  It is advisable then to choose moderate values for $\CSconstant_1, \CSconstant_2, \CSconstant_3, \CSconstant_4$.  Based on our collective experience, we recommend setting $\CSconstant_1 = \CSconstant_2 = \CSconstant_3 = \CSconstant_4 = 2$.
\end{remark}

Now that we have derived, presented, and proved well-posedness and an error estimate for our Nitsche-based formulation for the Kirchhoff-Love shell, we proceed with a discussion of the spline discretization to be employed for our numerics followed by a discussion of the associated \textit{a priori} error estimates for the Kirchhoff-Love shell.


\section{NURBS-Based Isogeometric Kirchhoff-Love Shell Discretizations}
\label{sec:apriori}

In this section, we provide a brief discussion of the discretization we employ for our numerical results, namely, \textbf{\emph{Non-Uniform Rational B-Splines}}, or NURBS. After presenting a brief introduction to NURBS, we provide \textit{a priori} error estimates for the Kirchhoff-Love shell under a NURBS discretization. It is worth noting that although we choose to employ NURBS for our numerical results, our theoretical exposition is not limited to such discretizations. In fact, since we have not discussed discretization until this point, the abstract Nitsche's framework discussed in Section~\ref{sec:Nitsche} is amenable to any discretization, as long as it provides sufficient smoothness.

\subsection{B-splines and NURBS}

The $i^{th}$ univariate B-spline basis function of degree $p$, herein denoted by $\hat{N}_{i,p}(\xi)$, is generated recursively over a \textbf{\emph{parametric domain}}, denoted herein by $\hat{\Omega}$. This parametric domain is defined by a \textbf{\emph{knot vector}}, that is, a non-decreasing set of real numbers called knots $\Xi = \{ \xi_1,\xi_2,\ldots,\xi_{n+p+1} \}$, where $n$ is the number of basis functions. The knot vector describes the support and continuity of the resulting basis functions. In NURBS-based isogeometric analysis, we typically employ an \textbf{\emph{open knot vector}}, where the first and last knots are repeated $p+1$ times, thus ensuring that the basis interpolates the geometry and solution field at the boundaries in one dimension and at the corners in higher dimensions.  We consider the maximally smooth case when each interior knot is unique, which yields a basis that is $C^{p-1}$-continuous. We also assume that the first and last knots in the knot vector are $0$ and $1$, respectively, without loss of generality.  The parametric domain is then $\hat{\Omega} = (0,1)$ in the one-dimensional setting.

The multivariate, tensor-product B-spline basis is obtained through a product of one-dimensional basis functions. In particular,
\begin{equation*}
\hat{N}_{{\bf i},{\bf p}}(\bm{\xi}) = \prod_{j=1}^{d_p} \hat{N}_{i_j,p_j}(\xi^j)
\end{equation*}
for multi-indices ${\bf i} = (i_1,i_2,...,i_{d_p})$ and ${\bf p} = (p_1,p_2,...,p_{d_p})$ representing basis function number and polynomial degree, respectively. Here, $d_p$ refers to the \textbf{\emph{parametric dimension}} while $d_s$ later refers to the \textbf{\emph{spatial dimension}}. Note that $d_s \ge d_p$. For the Kirchhoff-Love shell, $d_p = 2$ and $d_s = 3$.

A NURBS function is a projective transformation of a B-spline function in one higher spatial dimension. Given a set of B-spline basis functions and \textbf{\emph{NURBS weights}}, $w_{\bf i} \in \R^+$, we define the corresponding set of NURBS basis functions via
\begin{equation*}
\hat{R}_{{\bf i},{\bf p}}(\bm{\xi}) = \frac{w_{\bf i} \hat{N}_{{\bf i},{\bf p}}(\bm{\xi})}{w(\bm{\xi})}, \hspace{20pt} \text{where} \hspace{20pt} w(\bm{\xi}) = \sum_{\bf i} w_{\bf i} \hat{N}_{{\bf i},{\bf p}}(\bm{\xi}).
\end{equation*}
Here we have adopted the multi-index notation used for the multivariate B-splines in this definition.

We construct the \textbf{\emph{control mesh}} in $d_s$-dimensions that, together with the complete set of NURBS basis functions, define a $d_s$-dimensional geometry $\physDomain \subset \R^{d_s}$. This serves as our \textbf{\emph{physical domain}}. More specifically, given a set of NURBS control points ${\bf P}_{\bf i}$ and weights $w_{\bf i}$, the parameterization of the physical domain $\midsurfMap \colon \parDomain \rightarrow \physDomain$ is given by

\begin{equation*}
\midsurfMap(\bm{\xi}) = \sum_{\bf i} \fullTens{P}_{\bf i} \hat{R}_{{\bf i}}(\bm{\xi})
\end{equation*}

\noindent for all $\bm{\xi} \in \parDomain$, where $\parDomain = (0,1)^{d_p}$. Note that we dropped the subscript ${\bf p}$ for notational ease, as we do henceforth.

Since the vector-valued PDE considered herein is cast over a spatial variable, we require an appropriate space of basis functions defined in the physical space. To this end, we leverage the isoparametric concept through our geometric parameterization. Namely, we use the \textbf{\emph{push-forward}} operator describing how the physical variable $\midsurfMap$ is related to the parametric variable $\bm{\xi}$ in order to define NURBS basis functions in physical space as

\begin{equation*}
  R_{{\bf i}}({\bf{x}}(\bm{\xi})) = \hat{R}_{{\bf i}}(\bm{\xi}).
\end{equation*}

\noindent We then describe test and trial functions in terms of NURBS basis functions in physical space.  For the Kirchhoff-Love shell, we set

\begin{equation*}
\mathcal{V}^{\KLS}_h := \left\{ {\bf{v}} : \Omega \rightarrow \mathbb{R}^3 : {\bf{v}}({\bf{x}}) = \sum_{\bf i} {\bf{v}}_{\bf i} R_{{\bf i}}({\bf{x}}) \right\},
\end{equation*}

\noindent where the coefficients ${\bf{v}}_{\bf i} \in \mathbb{R}^3$ are commonly referred to as \textbf{\emph{control variables}}. For a comprehensive discussion of NURBS, their properties, and their implementation see \cite{Piegl2012}, and for a deeper discussion of NURBS-based isogeometric analysis and various applications, see \cite{Hughes2005,Cottrell2009}.  It should be noted that complex geometries of arbitrary topology may be represented using so-called multi-patch NURBS mappings \cite[Chapter~2]{Cottrell2009} or alternative parameterization techniques such as subdivision surfaces \cite{Cirak2000} and T-splines \cite{Bazilevs2010}.

Now that the discretization we employ has been presented, we take note of a subtle but important detail.  As discussed in Remarks \ref{remark:eig_KLS}, suitable trace constants for an isogeometric shell discretization may be attained by solving generalized eigenproblems.  In the asymptotic range, it is well known that these trace constants are independent of the mesh size $h$ for quasi-uniform isogeometric discretizations \cite{evans2013explicit}.  In practice, it is usually sufficient to compute trace constants for a coarse isogeometric discretization and then employ them for finer isogeometric discretizations. However, this property does not hold in general for all discretizations since it relies on the existence of discrete trace inequalities with mesh-independent constants.

\subsection{Sobolev Spaces on Manifolds}
To establish \textit{a priori} error estimates for NURBS-based Kirchhoff-Love shell discretizations, we first need to extend the concept of a Sobolev space from the Euclidean setting to the more general manifold setting.  To this end, let $\physDomain \subset \mathbb{R}^3$ be a smooth two-dimensional immersed manifold with Lipschitz-continuous boundary $\Gamma = \partial \physDomain$, and assume that $\physDomain$ is represented in terms of a smooth bijective mapping $\undef{\midsurfMap}: \parDomain \rightarrow \physDomain$, where $\parDomain \subset \mathbb{R}^2$ is an open domain with Lipschitz-continuous boundary $\hat{\Gamma} = \partial \parDomain$.  In this subsection, we use the word smooth to mean infinitely differentiable.  We can define a number of differential geometric objects on the manifold as discussed in \ref{sec:Appendix_Diff_Geo}, \ref{sec:Appendix_Cont_Mech}, and \ref{sec:Appendix_Components} and, in particular, we can define the surface gradient of a smooth scalar-valued function $v: \physDomain \rightarrow \mathbb{R}$ as
\begin{equation*}
\surfVec{\nabla} v = \frac{\partial v}{\partial \xi^{\alpha}} {\bf a}^{\alpha},
\end{equation*}
where $\xi^{\alpha}$ is the $\alpha^{\text{th}}$ in-plane convective coordinate and ${\bf a}^{\alpha}$ is the $\alpha^{\text{th}}$ contravariant tangent vector.  Similarly, we can define the surface gradient of a smooth order-$r$ tensor-valued function ${\bf A} = A_{m_1 \ldots m_r} {\bf a}^{m_1} \otimes \ldots \otimes {\bf a}^{m_r}$ as
\begin{equation*}
\surfVec{\nabla} {\bf A} =  \frac{\partial {\bf A}}{\partial \xi^{\alpha}} \otimes {\bf a}^{\alpha}.
\end{equation*}
Thus, the surface gradient of a smooth order-$r$ tensor-valued function is a smooth order-$(r+1)$ tensor-valued function.  Higher-order surface derivatives are defined recursively (e.g., $\surfVec{\nabla}^2 {\bf A} = \surfVec{\nabla} \left( \surfVec{\nabla} {\bf A} \right)$), and it is easily seen that the $k^{\text{th}}$ surface gradient of a smooth order-$r$ tensor-valued function is a smooth order-$(r+k)$ tensor-valued function.  Consequently, we may write the $k^{\text{th}}$ surface gradient of a smooth order-$r$ tensor-valued function as
\begin{equation*}
\surfVec{\nabla}^k {\bf A} = \left( \surfVec{\nabla}^k {\bf A} \right)_{m_1 \ldots m_{r+k}} {\bf a}^{m_1} \otimes \ldots \otimes {\bf a}^{m_{r+k}},
\end{equation*}
and we define the magnitude of the $k^{\text{th}}$ surface gradient as
\begin{equation*}
| \surfVec{\nabla}^k {\bf A} |^2 = a^{m_1 n_1} \ldots a^{m_k n_{r+k}} \left( \surfVec{\nabla}^k {\bf A} \right)_{m_1 \ldots m_{r+k}} \left( \surfVec{\nabla}^k {\bf A} \right)_{n_1 \ldots n_{r+k}},
\end{equation*}
where $a^{ij} = {\bf a}^i \cdot {\bf a}^j$ are the contravariant metric coefficients. By convention, we define $\surfVec{\nabla}^0 {\bf A} = {\bf A}$. It should be noted that we can define the surface divergence of a smooth order-$r$ tensor-valued function similarly, namely,
\begin{equation*}
\surfVec{\nabla} \cdot {\bf A} = \frac{\partial {\bf A}}{\partial \xi^{\alpha}} \cdot {\bf a}^{\alpha},
\end{equation*}
and the surface divergence of a smooth order-$r$ tensor-valued function is a smooth order-$(r-1)$ tensor-valued function.  The above definitions of surface gradient and surface divergence generalize the definitions used in Section \ref{sec:KL_Shell}.

Now, let $C^{\infty}(\physDomain)$ denote the space of smooth scalar-valued functions over the manifold. Also, for $s$ a non-negative integer, let
\begin{equation*}
C^{\infty}_s(\physDomain) := \left\{ v \in C^{\infty}(\physDomain) : \left\| v \right\|^2_{H^s(\physDomain)} < \infty \right\},
\end{equation*}
where
\begin{equation*}
\left\| v \right\|^2_{H^s(\physDomain)} := \sum_{k = 0}^{s} \ell^{2k-2} \int_{\physDomain} | \surfVec{\nabla}^k v |^2 \ d \physDomain
\end{equation*}
and $\ell = \text{diam}(\physDomain)$. We then define the Sobolev space $H^s(\physDomain)$ of scalar-valued functions as the completion of $C^{\infty}_s(\physDomain)$ with respect to $\left\| \cdot \right\|_{H^s(\physDomain)}$, and the Sobolev spaces for tensor-valued functions analogously.  The Sobolev space $H^0(\physDomain)$ coincides with $L^2(\physDomain)$, the space of square-integrable scalar-valued functions equipped with the norm
\begin{equation*}
\left\| v \right\|^2_{L^2(\physDomain)} := \ell^{-2} \int_{\physDomain} v^2 d\physDomain.
\end{equation*}
Note that all of the above Sobolev norms have the same units, simplifying the following analysis.  The Sobolev spaces presented above also coincide with those employed in Section \ref{sec:KL_Shell}.

Let $L^2(\Gamma)$ denote the space of square-integrable functions over the boundary of the manifold, equipped with the norm
\begin{equation*}
    \left\| v \right\|^2_{L^2(\Gamma)} := \ell^{-1} \int_{\Gamma} v^2 d\Gamma.
\end{equation*}
As in the Euclidean setting (see, e.g., \cite{Adams2003}), we define a linear and bounded trace operator $\text{Tr} : H^1(\physDomain) \rightarrow L^2(\Gamma)$ such that $\text{Tr}(v) = v|_{\Gamma}$ for smooth scalar-valued functions $v \in H^1(\physDomain)$.  For non-negative integers $s$, define $H^{s+1/2}(\Gamma) = \text{Tr}(H^{s+1}(\physDomain))$ and
\begin{equation*}
\left\| w \right\|^2_{H^{s+1/2}(\Gamma)} := \inf_{\substack{v \in H^{s+1}(\physDomain) \\ \text{Tr}(v) = w}} \left\| v \right\|^2_{H^{s+1}(\physDomain)}.
\end{equation*}
These fractional Sobolev spaces on the manifold boundary coincide with those employed in Section \ref{sec:KL_Shell}.

As a final remark, note that we can also define Sobolev spaces on non-smooth manifolds.  In particular, if the geometric mapping $\undef{\midsurfMap}: \parDomain \rightarrow \physDomain$ is a $C^{s-1}$-continuous NURBS mapping, then we can define the space $H^s(\Omega)$ on the manifold similarly to that presented here.  We can also define Sobolev spaces on manifolds that cannot be described in terms of a single parametric mapping.  This requires the use of charts, atlases, and transition maps.  For more information, see \cite{Schick2001}.

\subsection{Interpolation Estimates for NURBS-Based Kirchhoff-Love Shell Discretizations}

We are now in a position to state interpolation estimates for NURBS-based Kirchhoff-Love shell discretizations.  Following the work of \cite{Bazilevs2006}, we can construct a quasi-interpolation operator $\mathcal{I}^{\KLS}_h: \left(L^2(\Omega)\right)^3 \rightarrow \mathcal{V}^{\KLS}_h$ such that, for each set of integers $0 \leq k < l \leq p + 1$ and for all ${\bf v} \in \left(H^l(\Omega)\right)^3$,
\begin{equation*}
\left\| {\bf v} - \mathcal{I}^{\KLS}_h {\bf v} \right\|_{\left(H^k(\Omega)\right)^3} \leq C_{\text{interp}} \left( \frac{h}{\ell} \right)^{l-k} \left\| {\bf v} \right\|_{\left(H^l(\Omega)\right)^3},
\end{equation*}
where $h = \max_{K \in \mathcal{K}} h_K$ is the mesh size and $C_{\text{interp}}$ is a dimensionless constant independent of the mesh-to-domain-size ratio $h/\ell$ but dependent on the integers $k$ and $l$, the polynomial degree $p$, the normalized geometric mapping $\left({\bf x}(\bm{\xi}) - {\bf x}(\bm{0})\right)/\ell$, and the parametric mesh regularity.  The quasi-interpolation operator is defined by first constructing a locally $L^2$-stable quasi-interpolation operator $\hat{\mathcal{I}}^{\KLS}_h$ over the parametric domain using locally supported dual basis functions \cite[Chapter~12]{Schumaker2007} and then setting $\mathcal{I}^{\KLS}_h = \hat{\mathcal{I}}^{\KLS}_h \circ {\bf x}^{-1}$.  Similar interpolation estimates hold over individual elements of the computational mesh, as in \cite[Theorem~3.1]{Bazilevs2006}.

\subsection{\textit{A Priori} Error Estimate in the Energy Norm for NURBS-Based Kirchhoff-Love Shell Discretizations}

Armed with interpolation estimates, we are able to prove the following result for NURBS-based Kirchhoff-Love shell discretizations.

\begin{theorem}[\textit{A Priori} Error Estimate in the Energy Norm for the Kirchhoff-Love Shell]
If $p \geq 2$, then for any ${\bf u} \in \left(H^{p+1}(\Omega)\right)^3$, we have the estimate
$$\vvvertiii{{\bf u} - {\bf u}_h}^2_{\KLS} \leq C_{\text{bound}} | \mathbb{C} | \ell \left( \left( \frac{\zeta}{\ell} \right) \left( \frac{h}{\ell} \right)^{2p} + \left( \frac{\zeta}{\ell} \right)^3 \left( \frac{h}{\ell} \right)^{2p-2} \right) \| {\bf u} \|^2_{\left(H^{p+1}(\Omega)\right)^3},$$
where $C_{\text{bound}}$ is a dimensionless constant independent of the mesh-to-domain-size ratio $h/\ell$, the thickness-to-domain-size ratio $\zeta/\ell$, and the normalized elasticity tensor $\mathbb{C}/|\mathbb{C}|$, but dependent on polynomial degree $p$, the normalized geometric mapping $\left({\bf x}(\bm{\xi}) - {\bf x}(\bm{0})\right)/\ell$, the trace constants $C^S_{\textup{tr},1}$, $C^S_{\textup{tr},2}$, $C^S_{\textup{tr},3}$, $C^S_{\textup{tr},4}$, and $C^S_{\textup{tr},5}$, the penalty constants $C^S_{\textup{pen},1}$, $C^S_{\textup{pen},2}$, $C^S_{\textup{pen},3}$, and $C^S_{\textup{pen},4}$, and the parametric mesh regularity.
\end{theorem}

\begin{proof}
From Theorem~\ref{thm:KLS_Error}, we know that
$$\vvvertiii{{\bf u} - {\bf u}_h}_{\KLS} \leq \left( 1+ \frac{2}{1-\frac{1}{\CSconstant}} \right) \min_{{\bf v}_h \in \mathcal{V}^{\KLS}_h} \vvvertiii{{\bf u} - {\bf v}_h}_{\KLS},$$
so it holds that
\begin{equation}
\vvvertiii{{\bf u} - {\bf u}_h}^2_{\KLS} \leq \left( 1+ \frac{2}{1-\frac{1}{\CSconstant}} \right)^2 \vvvertiii{{\bf u} - \mathcal{I}^{\KLS}_h {\bf u}}^2_{\KLS}. \label{eq:Theorem_7_First}
\end{equation}
We now expand as follows:
\begin{align*}
\vvvertiii{{\bf u} - \mathcal{I}^{\KLS}_h {\bf u}}^2_{\KLS} &= \int_{\physDomain} \surfTens{\memStress}({\bf u} - \mathcal{I}^{\KLS}_h {\bf u}) \colon \surfTens{\memStrain}({\bf u} - \mathcal{I}^{\KLS}_h {\bf u}) \ d \physDomain + \int_{\physDomain} \surfTens{\bendStress}({\bf u} - \mathcal{I}^{\KLS}_h {\bf u}) \colon \surfTens{\bendStrain}({\bf u} - \mathcal{I}^{\KLS}_h {\bf u}) \ d \physDomain\\
    &\phantom{=} + \sum_{E_1 \in \mathcal{E}_{D_1}} \int_{E_1} \frac{h_{E_1}^3}{\cTrace{,1}^{\KLS} | \mathbb{C} | \thickness^3} \left|T_3({\bf u} - \mathcal{I}^{\KLS}_h {\bf u})\right|^2 + \sum_{C \in \cornerSet{D}} \frac{h_C^2}{\cTrace{,2}^{\KLS} | \mathbb{C} | \thickness^3} \llbracket \moment_{nt}({\bf u} - \mathcal{I}^{\KLS}_h {\bf u}) \rrbracket^2 \Big|_C\\
    &\phantom{=} + \sum_{E_2 \in \mathcal{E}_{D_2}} \int_{E_2} \frac{h_{E_2}}{\cTrace{,3}^{\KLS} | \mathbb{C} | \thickness^3} \left|\moment_{nn}({\bf u} - \mathcal{I}^{\KLS}_h {\bf u})\right|^2 \ d \physBoundary + \sum_{E_1 \in \mathcal{E}_{D_1}} \int_{E_1} \frac{h_{E_1}}{\cTrace{,4}^{\KLS} | \mathbb{C} | \thickness} \left| \surfVec{T}^{(\memStress)}({\bf u} - \mathcal{I}^{\KLS}_h {\bf u}) \right|^2 \ d \physBoundary\\
    &\phantom{=} + \sum_{E_1 \in \mathcal{E}_{D_1}} \int_{E_1} \frac{h_{E_1}}{\cTrace{,5}^{\KLS} | \mathbb{C} | \thickness^3} \left| \surfVec{T}^{(\bendStress)}({\bf u} - \mathcal{I}^{\KLS}_h {\bf u}) \right|^2 \ d \physBoundary + \sum_{E_1 \in \mathcal{E}_{D_1}} \int_{E_1} \frac{\cPen{,1}^{\KLS} |\mathbb{C}| \thickness^3}{h^3_{E_1}} \left| u_3 - \mathcal{I}^{\KLS}_h u_3 \right|^2 \ d \physBoundary\\
    &\phantom{=} + \sum_{C \in \cornerSet{D}} \frac{\cPen{,2}^{\KLS} |\mathbb{C}| \thickness^3}{h^2_{C}} \left| u_3 - \mathcal{I}^{\KLS}_h u_3 \right|^2 \Big|_{C} + \sum_{E_2 \in \mathcal{E}_{D_2}} \int_{E_2} \frac{\cPen{,3}^{\KLS} |\mathbb{C}| \thickness^3}{h_{E_2}} \left| \theta_n({\bf u} - \mathcal{I}^{\KLS}_h {\bf u}) \right|^2 \ d \physBoundary\\
    &\phantom{=} + \sum_{E_1 \in \mathcal{E}_{D_1}} \int_{E_1} \frac{\cPen{,4}^{\KLS} | \mathbb{C} | \thickness}{h_{E_1}} \left| \surfVec{u} - \mathcal{I}^{\KLS}_h \surfVec{u} \right|^2 \ d \physBoundary,
\end{align*}
where we used the abuse of notation $\mathcal{I}^{\KLS}_h u_3 = \mathcal{I}^{\KLS}_h {\bf u} \cdot {\bf a}_3$ and $\mathcal{I}^{\KLS}_h \surfVec{u} = \left( \mathcal{I}^{\KLS}_h {\bf u} \cdot {\bf a}_{\alpha} \right) {\bf a}^{\alpha}$.  A quick calculation reveals that
\begin{equation*}
\int_{\physDomain} \surfTens{\memStress}({\bf u} - \mathcal{I}^{\KLS}_h {\bf u}) \colon \surfTens{\memStrain}({\bf u} - \mathcal{I}^{\KLS}_h {\bf u}) \ d \physDomain \leq | \mathbb{C} | \ell \left( \frac{\thickness}{\ell} \right) \| {\bf u} - \mathcal{I}^{\KLS}_h {\bf u} \|^2_{\left(H^1(\Omega)\right)^3},
\end{equation*}
so, by our interpolation estimates,
\begin{equation}
\int_{\physDomain} \surfTens{\memStress}({\bf u} - \mathcal{I}^{\KLS}_h {\bf u}) \colon \surfTens{\memStrain}({\bf u} - \mathcal{I}^{\KLS}_h {\bf u}) \ d \physDomain \leq C_{1} | \mathbb{C} | \ell \left( \frac{\thickness}{\ell} \right) \left( \frac{h}{\ell} \right)^{2p} \| {\bf u} \|^2_{\left(H^{p+1}(\Omega)\right)^3},
\end{equation}
where $C_{1}$ is a dimensionless constant only dependent on polynomial degree, the normalized geometric mapping, and the parametric mesh regularity.  A similar calculation reveals that
\begin{equation*}
\int_{\physDomain} \surfTens{\bendStress}({\bf u} - \mathcal{I}^{\KLS}_h {\bf u}) \colon \surfTens{\bendStrain}({\bf u} - \mathcal{I}^{\KLS}_h {\bf u}) \ d \physDomain \leq | \mathbb{C} | \ell \left( \frac{\thickness}{\ell} \right)^3 \| {\bf u} - \mathcal{I}^{\KLS}_h {\bf u} \|^2_{\left(H^2(\Omega)\right)^3},
\end{equation*}
so, by our interpolation estimates,
\begin{equation}
\int_{\physDomain} \surfTens{\bendStress}({\bf u} - \mathcal{I}^{\KLS}_h {\bf u}) \colon \surfTens{\bendStrain}({\bf u} - \mathcal{I}^{\KLS}_h {\bf u}) \ d \physDomain \leq C_{2} | \mathbb{C} | \ell \left( \frac{\thickness}{\ell} \right)^3 \left( \frac{h}{\ell} \right)^{2p-2} \| {\bf u} \|^2_{\left(H^{p+1}(\Omega)\right)^3},
\end{equation}
where $C_{2}$ also is a dimensionless constant only dependent on polynomial degree, the normalized geometric mapping, and the parametric mesh regularity.  The other terms require more finesse and patience to bound.  However, by appealing to the continuous trace equality and local versions of our interpolation estimates (see, e.g., the proof of \cite[Theorem~6.2]{Evans2013DivFree}), it can be shown that
\begin{equation}
\sum_{E_1 \in \mathcal{E}_{D_1}} \int_{E_1} \frac{h_{E_1}^3}{\cTrace{,1}^{\KLS} | \mathbb{C} | \thickness^3} \left|T_3({\bf u} - \mathcal{I}^{\KLS}_h {\bf u})\right|^2 \leq C_{{\bf u}} C_{3} \left( \frac{\thickness}{\ell} \right)^3 \left( \frac{h}{\ell} \right)^{2p-2}
\end{equation}
\begin{equation}
\sum_{C \in \cornerSet{D}} \frac{h_C^2}{\cTrace{,2}^{\KLS} | \mathbb{C} | \thickness^3} \llbracket \moment_{nt}({\bf u} - \mathcal{I}^{\KLS}_h {\bf u}) \rrbracket^2 \Big|_C \leq C_{{\bf u}} C_{4} \left( \frac{\thickness}{\ell} \right)^3 \left( \frac{h}{\ell} \right)^{2p-2}
\end{equation}
\begin{equation}
\sum_{E_2 \in \mathcal{E}_{D_2}} \int_{E_2} \frac{h_{E_2}}{\cTrace{,3}^{\KLS} | \mathbb{C} | \thickness^3} \left|\moment_{nn}({\bf u} - \mathcal{I}^{\KLS}_h {\bf u})\right|^2 \ d \physBoundary \leq C_{{\bf u}} C_{5} \left( \frac{\thickness}{\ell} \right)^3 \left( \frac{h}{\ell} \right)^{2p-2}
\end{equation}
\begin{equation}
\sum_{E_1 \in \mathcal{E}_{D_1}} \int_{E_1} \frac{h_{E_1}}{\cTrace{,4}^{\KLS} | \mathbb{C} | \thickness} \left| \surfVec{T}^{(\memStress)}({\bf u} - \mathcal{I}^{\KLS}_h {\bf u}) \right|^2 \ d \physBoundary \leq C_{{\bf u}} C_{6} \left( \frac{\thickness}{\ell} \right) \left( \frac{h}{\ell} \right)^{2p}
\end{equation}
\begin{equation}
\sum_{E_1 \in \mathcal{E}_{D_1}} \int_{E_1} \frac{h_{E_1}}{\cTrace{,5}^{\KLS} | \mathbb{C} | \thickness^3} \left| \surfVec{T}^{(\bendStress)}({\bf u} - \mathcal{I}^{\KLS}_h {\bf u}) \right|^2 \ d \physBoundary \leq C_{{\bf u}} C_{7} \left( \frac{\thickness}{\ell} \right)^3 \left( \frac{h}{\ell} \right)^{2p-2}
\end{equation}
\begin{equation}
\sum_{E_1 \in \mathcal{E}_{D_1}} \int_{E_1} \frac{\cPen{,1}^{\KLS} |\mathbb{C}| \thickness^3}{h^3_{E_1}} \left| u_3 - \mathcal{I}^{\KLS}_h u_3 \right|^2 \ d \physBoundary  \leq C_{{\bf u}} C_{8} \left( \frac{\thickness}{\ell} \right)^3 \left( \frac{h}{\ell} \right)^{2p-2}
\end{equation}
\begin{equation}
\sum_{C \in \cornerSet{D}} \frac{\cPen{,2}^{\KLS} |\mathbb{C}| \thickness^3}{h^2_{C}} \left| u_3 - \mathcal{I}^{\KLS}_h u_3 \right|^2 \Big|_{C} \leq C_{{\bf u}} C_{9} \left( \frac{\thickness}{\ell} \right)^3 \left( \frac{h}{\ell} \right)^{2p-2}
\end{equation}
\begin{equation}
\sum_{E_2 \in \mathcal{E}_{D_2}} \int_{E_2} \frac{\cPen{,3}^{\KLS} |\mathbb{C}| \thickness^3}{h_{E_2}} \left| \theta_n({\bf u} - \mathcal{I}^{\KLS}_h {\bf u}) \right|^2 \ d \physBoundary  \leq C_{{\bf u}} C_{10} \left( \frac{\thickness}{\ell} \right)^3 \left( \frac{h}{\ell} \right)^{2p-2}
\end{equation}
\begin{equation}
\sum_{E_1 \in \mathcal{E}_{D_1}} \int_{E_1} \frac{\cPen{,4}^{\KLS} | \mathbb{C} | \thickness}{h_{E_1}} \left| \surfVec{u} - \mathcal{I}^{\KLS}_h \surfVec{u} \right|^2 \ d \physBoundary \leq C_{{\bf u}} C_{11} \left( \frac{\thickness}{\ell} \right) \left( \frac{h}{\ell} \right)^{2p},
\label{eq:Theorem_7_Last}
\end{equation}
where $C_{\bf u} = | \mathbb{C} | \ell \| {\bf u} \|^2_{\left(H^{p+1}(\Omega)\right)^3}$ and $C_{3}$ through $C_{11}$ are dimensionless constants only dependent on polynomial degree, the normalized geometric mapping, the parametric mesh regularity, the trace constants $C^S_{\text{tr},1}$, $C^S_{\text{tr},2}$, $C^S_{\text{tr},3}$, $C^S_{\text{tr},4}$, and $C^S_{\text{tr},5}$, and the penalty constants $C^S_{\text{pen},1}$, $C^S_{\text{pen},2}$, $C^S_{\text{pen},3}$, and $C^S_{\text{pen},4}$.  Collecting \eqref{eq:Theorem_7_First}-\eqref{eq:Theorem_7_Last}, we obtain
\begin{equation*}
  \begin{aligned}
    \vvvertiii{{\bf u} - {\bf u}_h}^2_{\KLS} &\leq C_{\bf u} \left( 1 + \frac{2}{1-\frac{1}{\CSconstant}} \right)^2 \Bigg( C_{1} + C_{6} + C_{11} \Bigg) \left( \frac{\zeta}{\ell} \right) \left( \frac{h}{\ell} \right)^{2p} \\
    &\phantom{=} + C_{\bf u} \left( 1+ \frac{2}{1-\frac{1}{\CSconstant}} \right)^2 \Bigg( C_{2} + C_{3} + C_{4} + C_{5} + C_{7} + C_{8} + C_{9} + C_{10} \Bigg) \left( \frac{\zeta}{\ell} \right)^3 \left( \frac{h}{\ell} \right)^{2p-2}.
  \end{aligned}
\end{equation*}
Since the coercivity constant $\gamma$ depends only on the trace constants $C^S_{\text{tr},1}$, $C^S_{\text{tr},2}$, $C^S_{\text{tr},3}$, $C^S_{\text{tr},4}$, and $C^S_{\text{tr},5}$ and the penalty constants $C^S_{\text{pen},1}$, $C^S_{\text{pen},2}$, $C^S_{\text{pen},3}$, and $C^S_{\text{pen},4}$, the desired result follows with
\begin{equation*}
C_{\text{bound}} = \left( 1+ \frac{2}{1-\frac{1}{\CSconstant}} \right)^2 \max\left\{ C_{1} + C_{6} + C_{11}, C_{2} + C_{3} + C_{4} + C_{5} + C_{7} + C_{8} + C_{9} + C_{10} \right\}.
\end{equation*}
This completes the proof.
\end{proof}

\noindent An immediate consequence of the above theorem is that the membrane strain satisfies the error bound
$$\int_{\Omega} \left( \left( \surfTens{\memStrain}({\bf u}) - \surfTens{\memStrain}({\bf u}_h) \right) : \frac{\mathbb{C}}{|\mathbb{C}|} : \left( \surfTens{\memStrain}({\bf u}) - \surfTens{\memStrain}({\bf u}_h) \right) \right) d\Omega \leq C_{\text{bound}} \left( \left( \frac{\zeta}{\ell} \right)^2 \left( \frac{h}{\ell} \right)^{2p-2} + \left( \frac{h}{\ell} \right)^{2p} \right) \| {\bf u} \|^2_{\left(H^{p+1}(\Omega)\right)^3}$$
and the bending strain satisfies the error bound
$$\frac{\ell^2}{12} \int_{\Omega} \left( \left( \surfTens{\bendStrain}({\bf u}) - \surfTens{\bendStrain}({\bf u}_h) \right) : \frac{\mathbb{C}}{|\mathbb{C}|} : \left( \surfTens{\bendStrain}({\bf u}) - \surfTens{\bendStrain}({\bf u}_h) \right) \right) d\Omega \leq C_{\text{bound}} \left( \left( \frac{h}{\ell} \right)^{2p-2} + \left( \frac{\zeta}{\ell} \right)^{-2} \left( \frac{h}{\ell} \right)^{2p} \right) \| {\bf u} \|^2_{\left(H^{p+1}(\Omega)\right)^3}.$$
Thus, when the thickness-to-domain-size ratio $\zeta/\ell$ is fixed, both the membrane strain and the bending strain converge as the mesh-to-domain-size ratio $h/\ell$ tends to zero.  Alternatively, when $h/\ell$ is fixed, the error bound for the membrane strain remains finite but the error bound for the bending strain tends to infinity as $\zeta/\ell$ tends to zero.  This is a consequence of \textbf{\emph{membrane locking}}, which affects virtually all Kirchhoff-Love shell discretizations relying on a primal (i.e., displacement only) formulation.  There are many different approaches to alleviate membrane locking, including the use of mixed methods wherein membrane strain is introduced as an additional variable \cite{Bathe1986}, but these approaches are not discussed further here since the focus is on weak enforcement of boundary conditions.

\subsection{\textit{A Priori} Error Estimates in Lower-Order Norms for NURBS-Based Kirchhoff-Love Shell Discretizations}

Using the well-known Aubin-Nitsche trick \cite[Chapter~4]{Strang1973}, we can also prove \textit{a priori} error estimates in low-order norms for NURBS-based Kirchhoff-Love shell discretizations.  The proof of this result is omitted for brevity.

\begin{theorem}[\textit{A Priori} Error Estimate in Lower-Order Norms for the Kirchhoff-Love Shell]
If $p \geq 2$, then for any ${\bf u} \in \left(H^{p+1}(\Omega)\right)^3$, we have the estimates
$$\| {\bf u} - {\bf u}_h \|^2_{\left(H^1(\Omega)\right)^3} \leq C_{\text{bound},1} \left( \frac{h}{\ell} \right)^{2p} \| {\bf u} \|^2_{\left(H^{p+1}(\Omega)\right)^3}$$
and
$$\| {\bf u} - {\bf u}_h \|^2_{\left(L^2(\Omega)\right)^3} \leq C_{\text{bound},2} \left( \frac{h}{\ell} \right)^{\min\left\{2p+2,4p-4\right\}} \| {\bf u} \|^2_{\left(H^{p+1}(\Omega)\right)^3},$$
where $C_{\text{bound,1}}$ and $C_{\text{bound,2}}$ are dimensionless constants independent of the mesh-to-domain-size ratio $h/\ell$, but dependent on the thickness-to-domain-size ratio $\zeta/\ell$, the normalized elasticity tensor $\mathbb{C}/|\mathbb{C}|$, the polynomial degree $p$, the normalized geometric mapping $\left({\bf x}(\bm{\xi}) - {\bf x}(\bm{0})\right)/\ell$, the trace constants $C^S_{\textup{tr},1}$, $C^S_{\textup{tr},2}$, $C^S_{\textup{tr},3}$, $C^S_{\textup{tr},4}$, and $C^S_{\textup{tr},5}$, the penalty constants $C^S_{\textup{pen},1}$, $C^S_{\textup{pen},2}$, $C^S_{\textup{pen},3}$, and $C^S_{\textup{pen},4}$, and the parametric mesh regularity.
\end{theorem}

\noindent Note that it is impossible to make the bounding constants $C_{\text{bound,1}}$ and $C_{\text{bound,2}}$ appearing in the above theorem independent of the thickness-to-domain-size ratio $\zeta/\ell$ since membrane locking occurs in the zero thickness limit. Also, although we are able to employ a discretization comprised of quadratic splines, we only expect to see optimal convergence rates in the energy and $H^1$-equivalent norms for such a discretization. The $L^2$ norm cannot exceed convergence rates faster than second-order for quadratic discretizations of fourth-order partial differential equations \cite[Chapter~2]{Strang1973}.


\section{Numerical Results}
\label{sec:num_results}

In this section, we demonstrate the robustness and effectiveness of our proposed methodology through numerical experiments. The robustness is shown by the ability of our framework to accommodate a wide variety of geometric configurations with complex boundary conditions, while the effectiveness is demonstrated by the discretization obtaining optimal convergence rates in both the associated energy- and $L^2$-norms.

\begin{figure}[ht!]
  \centering
  \begin{subfigure}[b]{0.24\textwidth}
    \centering
    \includegraphics[width=\textwidth]{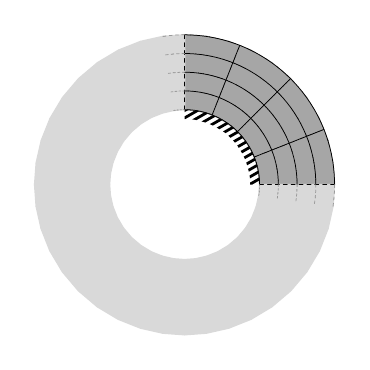}
    \caption*{Problem 1}
  \end{subfigure}
  \begin{subfigure}[b]{0.24\textwidth}
    \centering
    \includegraphics[width=\textwidth]{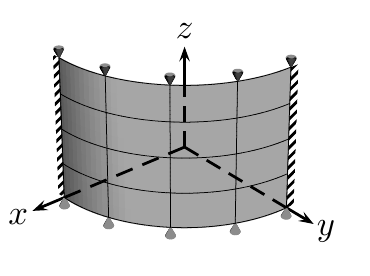}
    \caption*{Problem 3}
  \end{subfigure}
  \begin{subfigure}[b]{0.24\textwidth}
    \centering
    \includegraphics[width=\textwidth]{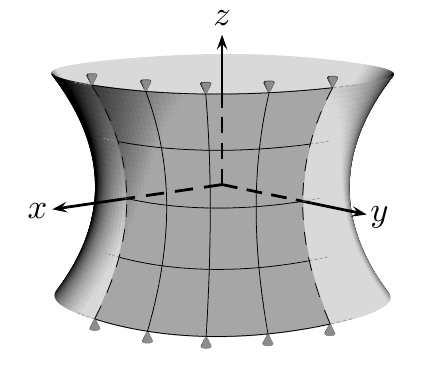}
    \caption*{Problem 5}
  \end{subfigure}
  \begin{subfigure}[b]{0.24\textwidth}
    \centering
    \includegraphics[width=\textwidth]{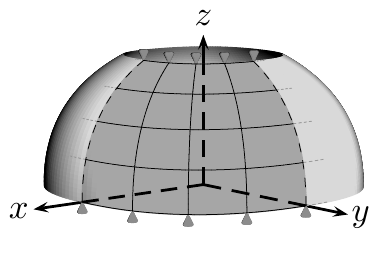}
    \caption*{Problem 7}
  \end{subfigure}
  \\
  \begin{subfigure}[b]{0.24\textwidth}
    \centering
    \includegraphics[width=\textwidth]{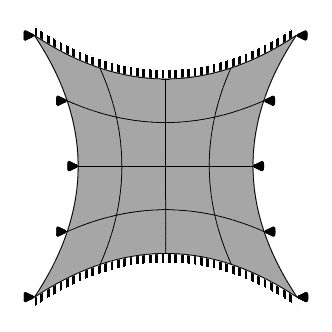}
    \caption*{Problem 2}
  \end{subfigure}
  \begin{subfigure}[b]{0.24\textwidth}
    \centering
    \includegraphics[width=\textwidth]{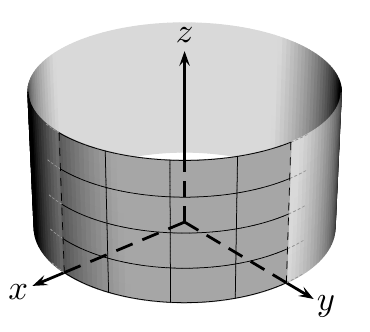}
    \caption*{Problem 4}
  \end{subfigure}
  \begin{subfigure}[b]{0.24\textwidth}
    \centering
    \includegraphics[width=\textwidth]{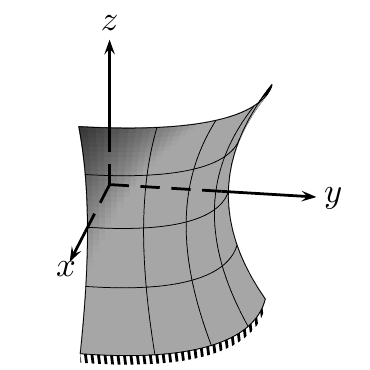}
    \caption*{Problem 6}
  \end{subfigure}
  \begin{subfigure}[b]{0.24\textwidth}
    \centering
    \includegraphics[width=\textwidth]{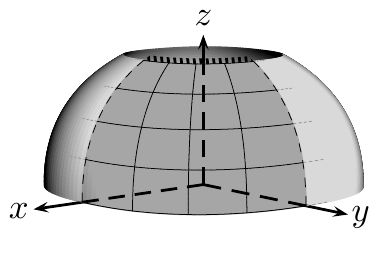}
    \caption*{Problem 8}
  \end{subfigure}
\caption{
The eight problems that comprise the linear shell obstacle course. From left to right, the first column contains flat geometries, the second column contains parabolic geometries, the third column contains hyperbolic geometries, and the fourth column contains elliptic geometries. Boundaries with prescribed displacement and bending moment (e.g., simply supported boundaries) are denoted by \protect\TWODSS/\protect\THREEDSS, boundaries with prescribed displacement and normal rotation (e.g., clamped boundaries) are denoted by \protect\CLAMPED, boundaries with prescribed ersatz traction and normal rotation (e.g., symmetric boundaries) are denoted by \protect\SYM, and boundaries with prescribed ersatz traction and bending moment (e.g., free boundaries) are denoted by \protect\FREE.
}
\label{fig:shell_obstacle_course}
\end{figure}

Performing numerical validation for shells is particularly difficult for several reasons. There are few, if any, analytic solutions available for assessing discretization accuracy in Sobolev-equivalent norms. For this reason, many have resorted to measuring performance by pointwise measures, such as the displacement at the location of point-load application or where the point of maximal displacement is likely to occur. The so-called ``shell obstacle course'' found in much of the related literature is perhaps the most common suite of problems where converged pointwise values are used to indicate validity \cite{Scordelis1964,Belytschko1985}. Unfortunately, these values only agree up to a few digits of precision and are therefore not reliable for a rigorous assessment of convergence. Furthermore, as we have no theoretical error estimates in terms of pointwise quantities, this approach does not suffice for our numerical validation.

Separately, we also observe that for fine meshes discretized with high-order elements, roundoff errors due to ill-conditioning of the resulting linear system often dominate the solution, as is evident in the forthcoming numerical results. This presents difficulty in using high-resolution solutions as a benchmark for convergence. To combat roundoff errors in the asymptotic regime before this phenomenon dominates the solution, we employ three iterations of residual-based iterative refinement to all of our forthcoming problems \cite{Wilkinson1948,Wilkinson1963}. To demonstrate that we truly obtain optimal convergence rates, we instead resort to ``manufactured'' forcing functions that, when applied to the geometries we consider, yield known displacement fields. For simple and flat domains, this task is relatively straightforward and amounts to applying the differential operator $\mathcal{L}$ to a desired solution field ${\bf u}$ to obtain the corresponding $\applied{\textbf{\forcing}}$. Manufacturing forcing functions for shell problems posed over curved manifolds is conceptually no different, but in practice it is a much more involved task and care must be taken at every step. To facilitate this process, we carefully implemented all the steps in Mathematica, which allows many of the operations to be done symbolically. We have found one other instance where such a process has been performed \cite{gfrerer2018code}, however our linear shell obstacle is comprehensive in that it encompasses all possible boundary condition configurations. Moreover, we have provided the forcing functions and their corresponding displacement, strain, and stress fields from our linear shell obstacle course for the research community\footnote{https://github.com/wdas/shell-obstacle-course}.

In order to make our testing as exhaustive as possible, we devised a new linear shell obstacle course, which covers flat, parabolic, hyperbolic, and elliptic geometries subject to simply supported, clamped, free, and symmetric boundary conditions (see Figure~\ref{fig:shell_obstacle_course}). We fix the shell thickness to be $\thickness = 0.1 \ m$ and set the material parameters $E = 10 \ MPa$ and $\nu = 0.3$ in our linear constitutive model \eqref{eqn:constitutive} for all of the problems we consider.  Note that in the forthcoming subsections, all displacement fields presented are in meters.  Our methodology is free of membrane locking for all of the problems we consider because the thickness-to-domain-size ratio, $\thickness/\ell$, is always $0.1$. This enables us to numerically examine asymptotic rates of convergence. For all the experiments, we employ uniform, tensor-product meshes in the parametric domain. Since the geometric mappings are non-degenerate, the corresponding physical mesh is comprised of curvilinear quadrilateral elements.

Recall Remark~\ref{remark:shellBCs} wherein it is mentioned that the Kirchhoff-Love shell accommodates four common types of boundary conditions: clamped, simply supported, symmetric, and free. In terms of Dirichlet boundary conditions, a simply supported shell is one such that the boundary displacement is zero while the normal rotation is unconstrained. A shell with symmetric boundary conditions is one such that the boundary displacement is unconstrained while normal rotation is zero. Finally, a shell with free boundary conditions is one such that both the boundary displacement and the normal rotation are unconstrained. From energetic principles, this implies that the quantities that are energetically conjugate to those that are unconstrained must vanish. For simply supported structures, this is the bending moment; for those with symmetric boundaries, this is the ersatz traction; and for those with free boundaries, this is both. Since it is nearly impossible to manufacture a solution exhibiting exactly these properties,
our linear shell obstacle course simply emulates this behavior by instead prescribing nonhomogeneous
boundary conditions in lieu of those that should be unconstrained. However, for readability, we still refer to these boundary condition types by their classical name throughout this section even though it is understood that they are in fact emulated.

The presence of four covariant differentiation operators in the underlying strong formulation of the Kirchhoff-Love shell yields complex forcing functions that are not only very nonlinear, they are also often numerically unstable to evaluate in double precision. For this reason, it is recommended that these entities be evaluated in extended precision and truncated to the operating precision only after all operations comprising the function have been completed. For our particular results, we compute forcing function data as well as strain and stress tensor data with 100 digits of accuracy, truncate the data to double precision, and save the data in external files that are later read by our isogeometric analysis routines. This ensures that the function values evaluated at the quadrature points are accurate enough for our formation and assembly routines, as well as for post-processing. Furthermore, to handle the nonlinearities of these functions, we employ a $25 \times 25$-point quadrature rule to avoid under-integration. Alternatively, one could use an adaptive quadrature scheme to overcome this issue without suffering from the curse of dimensionality for tensor product meshes, since these nonlinearities are most prevalent near the boundaries and at ``corner'' points.

For the benefit of the community, we provide a Mathematica notebook containing the problem data for our shell obstacle course that allows the results presented in this section to be reconstructed. In particular, the notebook includes geometric parameterizations, displacement fields, the strain and stress tensors for bending and membrane action, and the forcing function. For convenience, the control meshes and NURBS weights associated with the geometries we consider in our numerical results are also tabulated in \ref{sec:Geo_Param}. In the following, each manufactured displacement field is denoted by a superscript number in parenthesis that indicates the problem number. These numbers correspond to the tabulated geometry data in \ref{sec:Geo_Param} as well as in the supplemental notebook.

\subsection{Flat Geometry}

We begin our presentation of numerical results with the examples having flat geometric configurations. In this case, the in-plane and out-of-plane phenomena are completely decoupled due to the lack of curvature. This property is useful for determining that Nitsche's method is implemented correctly for in-plane and out-of-plane behaviors separately. Regardless of this decoupling, we still consider flat plate-membrane systems subject to both in-plane and out-of-plane displacement fields.

The first problem set we consider is comprised of (i) a NURBS-mapped, annular domain and (ii) an astroid domain as shown in Figure~\ref{fig:shell_obstacle_course}. Note that the astroid domain is not truly an astroid by its mathematical definition, but rather closely resembles one. The annular domain is modeled through a quarter-annulus with symmetric boundary conditions employed on the straight edges. This domain is subject to a linear, radial displacement in-plane and an exponential transverse displacement field. Moreover it accommodates a clamped boundary on the inner radius and a free boundary on the outer radius. More specifically, the displacement field for Problem 1 over the annular domain is characterized by

\begin{equation*}
  {\bf u}^{(1)}\left( \xi^1, \xi^2 \right) :=  \left( \frac{\xi^1}{\|\aVec_1\|} \right) \aVec_1 + \left( e^{\xi^1} - 1 \right)\xi^1 \aVec_3.
\end{equation*}

By comparison, the astroid domain is loaded such that the resulting in-plane displacement field is a vortex with no displacement on the domain boundary and a sinusoidal transverse displacement field. This choice of displacement field effectively emulates a plate with two, simply supported edges on opposite ends and clamped edges on the remaining boundaries. The displacement field for Problem 2 over the astroid domain is given by the following set of Cartesian displacement modes:

\begin{equation*}
  {\bf u}^{(2)}\left( \xi^1, \xi^2 \right) := \left( \begin{array}{c}
  u_x\\
  u_y\\
  u_z \end{array} \right) = \left( \begin{array}{c}
    \left( \xi^1 - 1 \right)^2 \left( \xi^1 \right)^2 \left( \frac{1}{2} - \xi^2 \right) (1 - \xi^2) \xi^2 \vspace{1pt}\\
    \left( \xi^2 - 1 \right)^2 \left( \xi^2 \right)^2 \left( \frac{1}{2} - \xi^1 \right) (1 - \xi^1) \xi^1 \vspace{1pt}\\
    \left( 1 - \xi^1 \right) \xi^1 \sin(\pi \xi^1) \sin(\pi \xi^2) \end{array} \right).
\end{equation*}

\begin{figure}[t!]
  \includegraphics{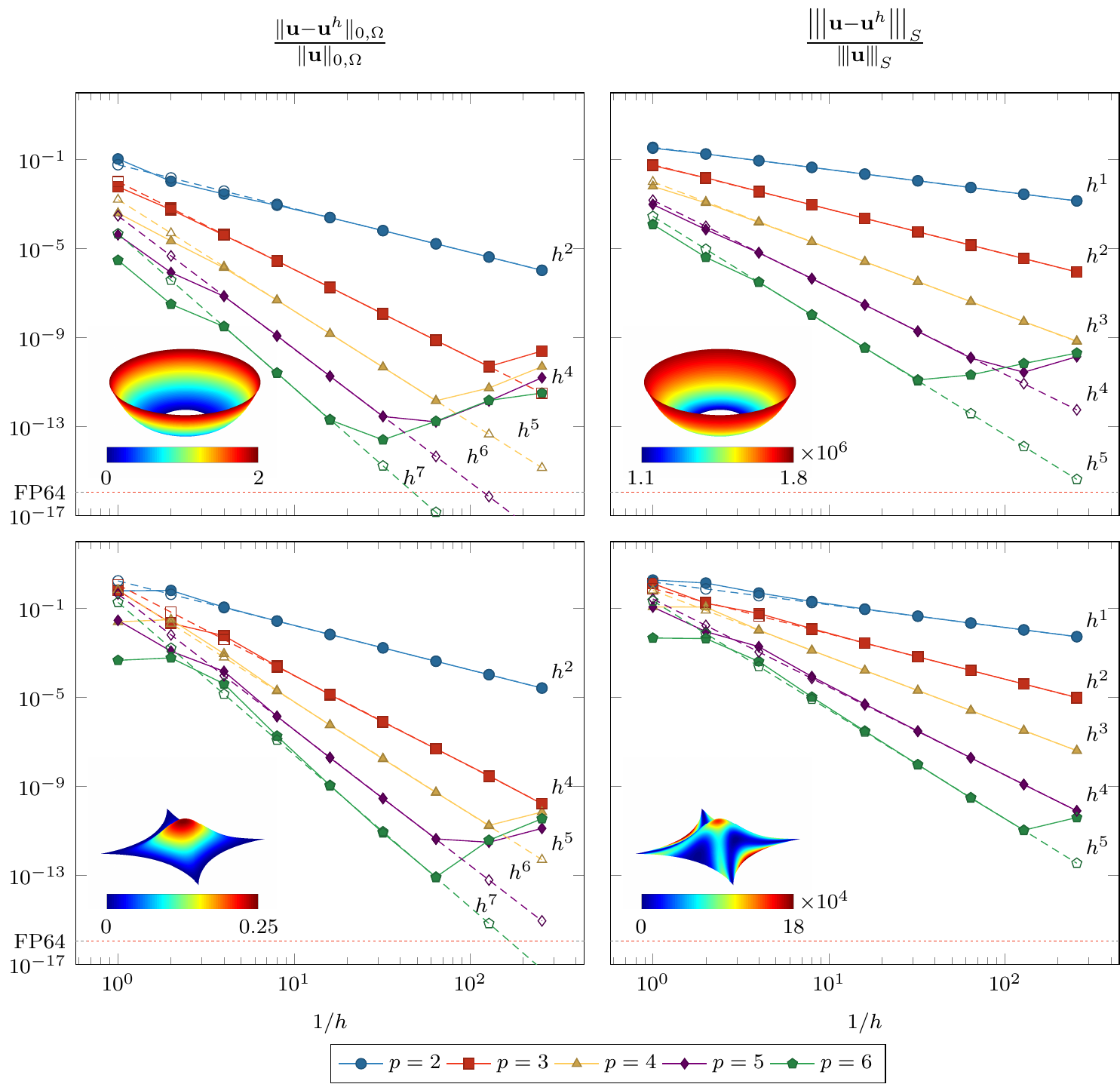}
  \caption{The convergence behavior of the annular problem in the $L^2$-norm (top left) and energy norm (top right) are shown. The convergence results for the astroid problem are shown in the $L^2$-norm (bottom left) and the energy norm (bottom right). Optimal convergence rates are observed with their theoretical counterparts shown as dashed lines with hollow, identical markers. The magnitude of the displacement field is plotted over the geometry for plots pertaining to the $L^2$-norm, while the total internal energy density is plotted over the geometry for plots pertaining to the energy norm.}
  \label{fig:flat_results}
\end{figure}

As is clearly demonstrated in Figure~\ref{fig:flat_results}, optimal convergence rates are obtained in both the $L^2$-norm for $p>2$ and in the energy norm for all polynomial degrees of discretization considered. Note that the convergence rate in the $L^2$-norm is sub-optimal for $p=2$, while the convergence rate in the energy norm is optimal. This phenomenon is expected and will be observed for all problems considered in the linear shell obstacle course. For an elaboration, refer to \cite[Chapter~2]{Strang1973}. For large $p$ and small $h$, we observe the aforementioned roundoff divergence due to matrix ill-conditioning.

\subsection{Parabolic Geometry}

The next problem class that we consider are shells over a parabolic geometry, namely, a NURBS-parameterized cylinder. In this instance, we encounter a coupling between in-plane and out-of-plane behaviors due to the curvature of the shell body. First, we consider a NURBS-mapped quarter-cylinder domain and, next, we model a full cylindrical shell by employing symmetric boundary conditions across the edges of the quarter-cylinder. In the first configuration, we apply a forcing function such that the resulting displacement field is a quartic-by-quadratic polynomial. This choice of displacement field emulates clamped and simply supported boundary conditions. Moreover, the displacement field for Problem 3 over the quarter-cylinder is given by the following:

\begin{equation*}
  {\bf u}^{(3)}\left( \xi^1, \xi^2 \right) := - \left( \xi^1 - 1\right)^2 \left( \xi^1 \right)^2 \xi^2 \left(\xi^2 - 1 \right) \aVec_3.
\end{equation*}

The second problem configuration applies a forcing function such that the resulting displacement field is a sinusoid in the radial direction. This displacement field has free boundary conditions along the axis of the cylinder. The displacement field for Problem 4 over the cylindrical domain is given by the following:

\begin{equation*}
  {\bf u}^{(4)}\left( \xi^1, \xi^2 \right) := \frac{1}{2} \cos \left( \pi \xi^1 \right) \aVec_3.
\end{equation*}

\begin{figure}[t!]
  \includegraphics{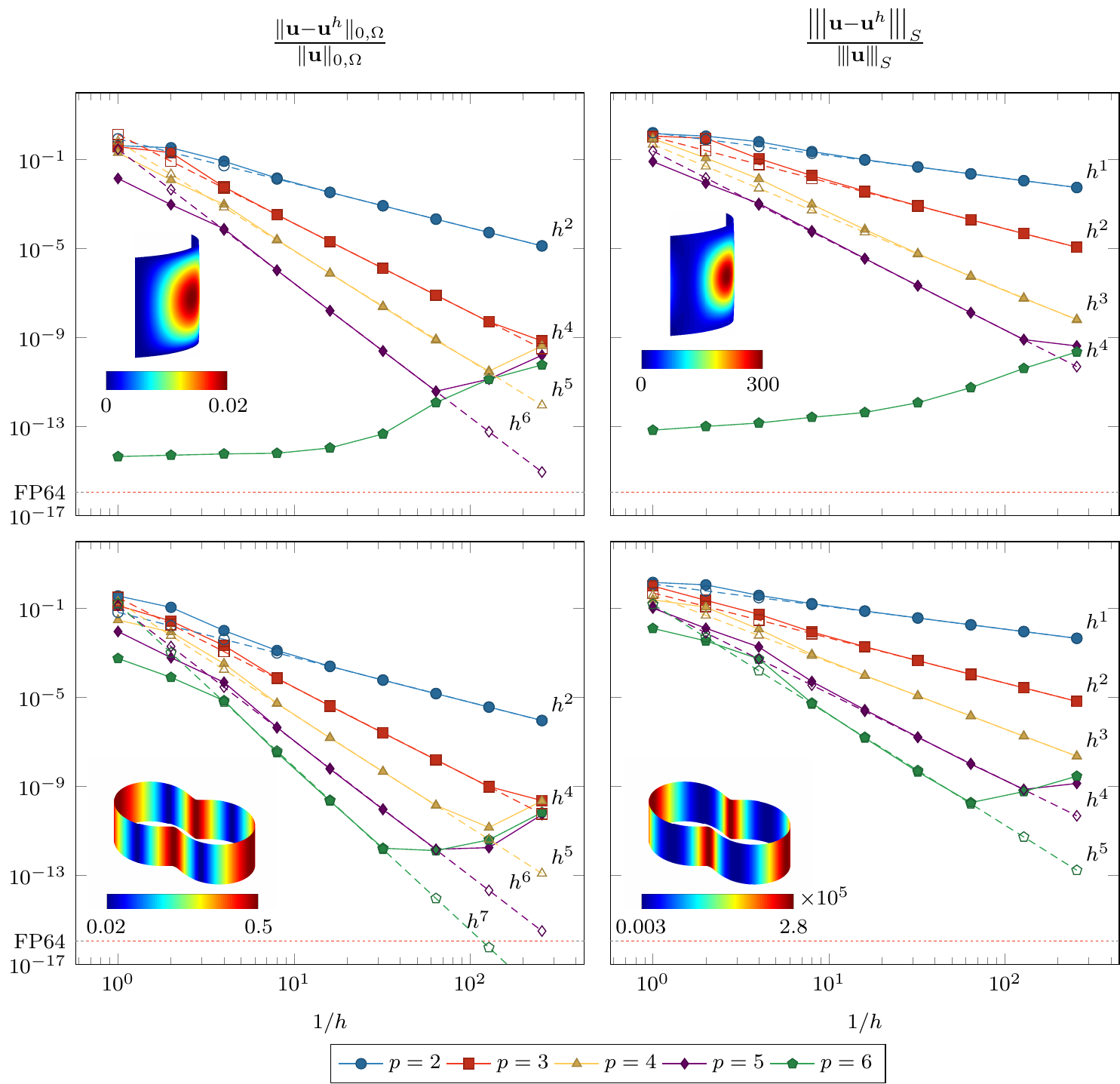}
  \caption{The convergence behavior of the quarter-cylinder from configuration 1 in the $L^2$-norm (top left) and energy norm (top right) are shown. The convergence results for the cylinder in configuration 2 are shown in the $L^2$-norm (bottom left) and in the energy norm (bottom right). Optimal convergence rates are observed with their theoretical counterparts shown as dashed lines with hollow, identical markers. The magnitude of the displacement field is plotted over the geometry for plots pertaining to the $L^2$-norm, while the total internal energy density is plotted over the geometry for plots pertaining to the energy norm.}
  \label{fig:par_results}
\end{figure}

Clearly, the results depicted in Figure~\ref{fig:par_results} demonstrate that the optimal convergence rates are obtained in the $L^2$-norm for $p>2$ and the energy norm for all polynomial degrees of discretization. Once again, the convergence rate in the $L^2$-norm is sub-optimal for $p=2$, while the convergence rate in the energy norm is optimal. Note that the displacement field for Problem 3 is in the span of biquartic polynomial basis functions. Moreover, the geometric mapping is a rational quadratic, as the NURBS weighting function is a quadratic polynomial. Therefore, we obtain machine precision for any mesh size with a sixth-order discretization in both $L^2$- and energy norms until the ill-conditioning roundoff divergence begins. In the case of the energy norm, we never truly obtain floating-point machine precision. This is also due to the ill-conditioning of the linear system and the computation of the strain energy.

\subsection{Hyperbolic Geometry}

This next class of geometries considered pertains to hyperbolic configurations. Once again, geometric curvatures couple the in-plane and out-of-plane effects; however, contrary to the previous two scenarios, this geometry is doubly curved and, hence, has a nonzero Gaussian curvature. Note that in this instance, the hyperbolic paraboloid is only an approximation in the sense that it is not a true NURBS domain but rather a B-spline approximation, i.e., $w(\bm{\xi}) \equiv 1$. This choice is made for sake of simplicity in solution field manufacturing and convergence analysis; it does not alter the hyperbolic classification of the geometry. The resulting forcing function, as well as the stress and strain tensors, for the NURBS-mapped hyperbolic paraboloid are drastically more complex than the polynomial counterpart.

The first problem configuration we consider is a full hyperbolic paraboloid that is modeled through the use of symmetric boundary conditions. The top and bottom of the hyperbolic paraboloid have simply supported boundary conditions and the shell is subject to a loading such that the resulting geometry is a B-spline approximation of a cylinder. The displacement field for Problem 5 over the hyperbolic paraboloid domain is given by

\begin{equation*}
  {\bf u}^{(5)}\left( \xi^1, \xi^2 \right) := \left( \begin{array}{c}
  u_x\\
  u_y\\
  u_z \end{array} \right) = \left( \begin{array}{c}
    \sqrt{2} \left( \left( \xi^1 \right)^2 - 1 \right) \left( \xi^2 -1 \right) \xi^2 \vspace{1pt}\\
    \sqrt{2} \left( \xi^1 - 2 \right) \xi^1 \left( \xi^2 -1 \right) \xi^2 \vspace{1pt}\\
    0 \end{array} \right).
\end{equation*}

In the second configuration, we instead consider a quarter of this hyperbolic paraboloid. In this scenario, one edge is clamped while the other edges are free and the entire system is subject to a forcing such that the resulting displacement field is a sinusoid. In particular, the displacement field for Problem 5 over the quarter-hyperbolic paraboloid domain is given by

\begin{equation*}
  {\bf u}^{(6)}\left( \xi^1, \xi^2 \right) := \left( \begin{array}{c}
  u_x\\
  u_y\\
  u_z \end{array} \right) = \left( \begin{array}{c}
    \xi^2 \sin \left( \frac{\pi}{2} \xi^2 \right) \vspace{1pt}\\
    \xi^2 \sin \left( \frac{\pi}{2} \xi^2 \right) \vspace{1pt}\\
    0 \end{array} \right).
\end{equation*}

\begin{figure}[t!]
  \includegraphics{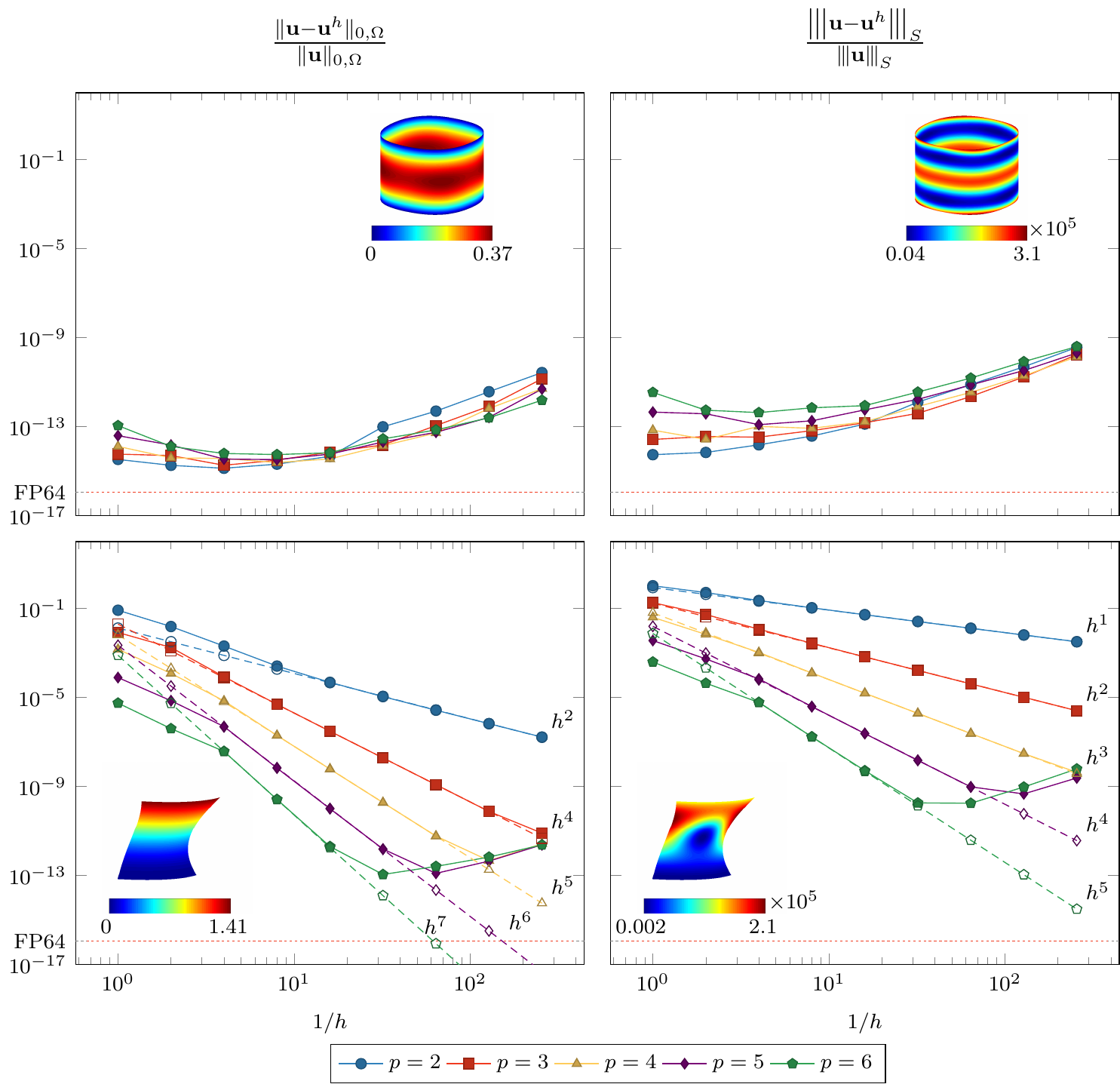}
  \caption{The convergence behavior of the hyperbolic paraboloid problem in configuration 1 are shown in the $L^2$-norm (top left) and energy norm (top right) are shown. The convergence results for the quarter-hyperbolic paraboloid in configuration 2 are shown in the $L^2$-norm (bottom left) and the energy norm (bottom right). Optimal convergence rates are observed with their theoretical counterparts shown as dashed lines with hollow, identical markers. The magnitude of the displacement field is plotted over the geometry for plots pertaining to the $L^2$-norm, while the total internal energy density is plotted over the geometry for plots pertaining to the energy norm.}
  \label{fig:hyp_results}
\end{figure}

\begin{figure}[t!]
  \includegraphics{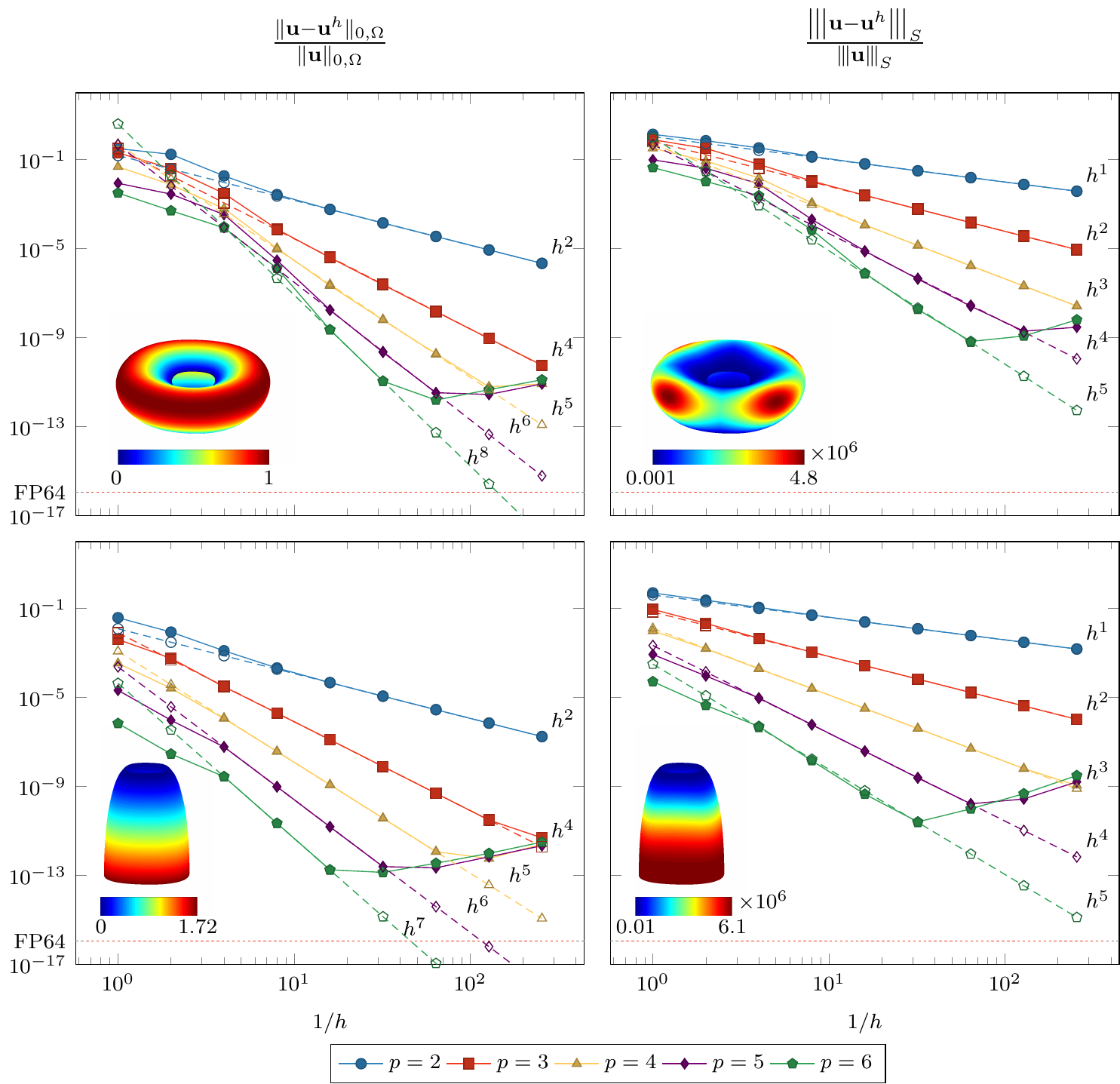}
  \caption{The convergence behavior of the hemispherical shell in configuration 1 is shown in the $L^2$-norm (top left) and in the energy norm (top right). The convergence results for the hemispherical shell in configuration 2 are shown in the $L^2$-norm (bottom left) and the energy norm (bottom right). Optimal convergence rates are observed with their theoretical counterparts shown as dashed lines with hollow, identical markers. The magnitude of the displacement field is plotted over the geometry for plots pertaining to the $L^2$-norm, while the total internal energy density is plotted over the geometry for plots pertaining to the energy norm.}
  \label{fig:ell_results}
\end{figure}

Figure~\ref{fig:hyp_results} demonstrate that our discretization yields optimal convergence behavior in this case in both the $L^2$-norm for $p>2$ and in the energy norm for all polynomial degrees considered. We once again observe sub-optimal convergence rates in the $L^2$ norm, while the energy norm is unaffected. Note that the displacement field for the first configuration is in the span of all polynomial degrees considered. This is because the B-spline approximations for both the hyperbolic paraboloid and the cylinder are quadratic, and consequently so is their difference. Therefore, we obtain machine precision for all degrees of discretization in this instance. Once again, matrix ill-conditioning presents itself in the form of roundoff divergence. The effects of this ill-conditioning are also present in the preasymptotic region where true floating-point machine precision is not obtained, as was the case for the $p=6$ discretization of Problem 3. In the second configuration, we obtain the expected convergence behavior.

\subsection{Elliptic Geometry}

The final class of geometries considered here are elliptic configurations in the form of a hemispherical shell. Much like the hyperbolic case, these geometries also have nonzero Gaussian curvature so, as before, we only approximate the hemisphere by letting $w(\bm{\xi}) \equiv 1$.

The first problem configuration considered is a hemispherical shell subject to an internal pressure resulting in a radial sinusoidal displacement field. This problem is subject to symmetric boundary conditions along the edges of the hemispherical section as well as simply supported boundary conditions on the top and bottom of the shell. In particular, the displacement field for Problem 7 over the hemispherical shell domain is given by

\begin{equation*}
  {\bf u}^{(7)}\left( \xi^1, \xi^2 \right) := - \sin \left( \pi \xi^1 \right) \aVec_3.
\end{equation*}

The second configuration is the same hemispherical shell employing symmetric boundary conditions along the edges to emulate a full hemisphere, but with the top edge clamped and the bottom edge free in this scenario. The shell is subject to a loading such that the resulting displacement field is exponential and oriented downward. The displacement field for Problem 8 over the hemispherical shell domain is given by

\begin{equation*}
  {\bf u}^{(8)}\left( \xi^1, \xi^2 \right) := \left( \begin{array}{c}
  u_x\\
  u_y\\
  u_z \end{array} \right) = \left( \begin{array}{c}
    0 \vspace{1pt}\\
    0 \vspace{1pt}\\
    \left( \xi^1 - 1 \right)\left( e - e^{\xi^1} \right) \end{array} \right).
\end{equation*}

The convergence analysis shown in Figure~\ref{fig:ell_results} demonstrates that optimal convergence rates are once again obtained in the $L^2$-norm for $p>2$ and in the energy norm for all polynomial degrees discretized over the elliptic geometries. The convergence rate in the $L^2$ norm for $p=2$ is inhibited as discussed previously.

\begin{figure}[t!]
  \centering
  \begin{subfigure}[t]{0.99\textwidth}
    \centering
    \includegraphics{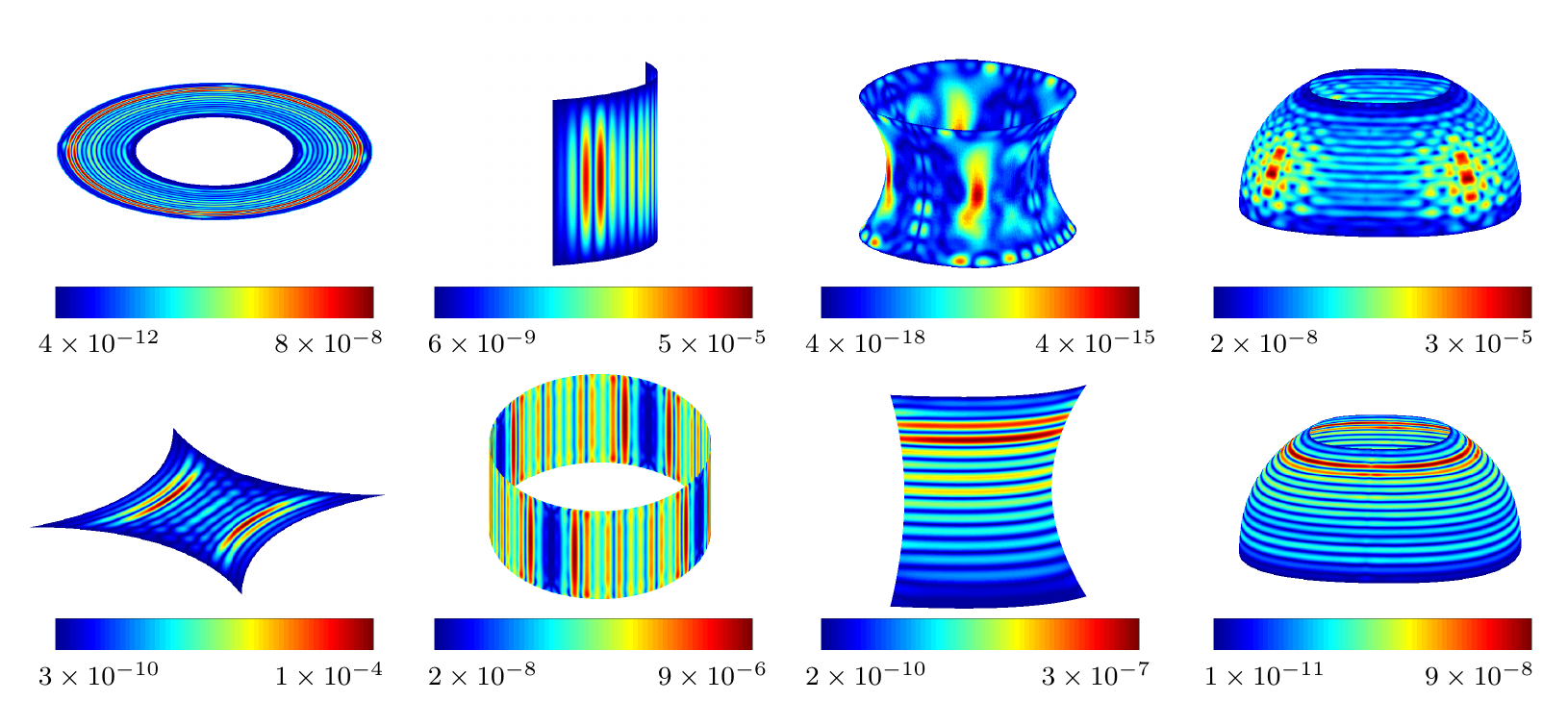}
    \caption{Relative Displacement Errors}
  \end{subfigure}
  \\
  \begin{subfigure}[b]{0.99\textwidth}
    \centering
    \includegraphics{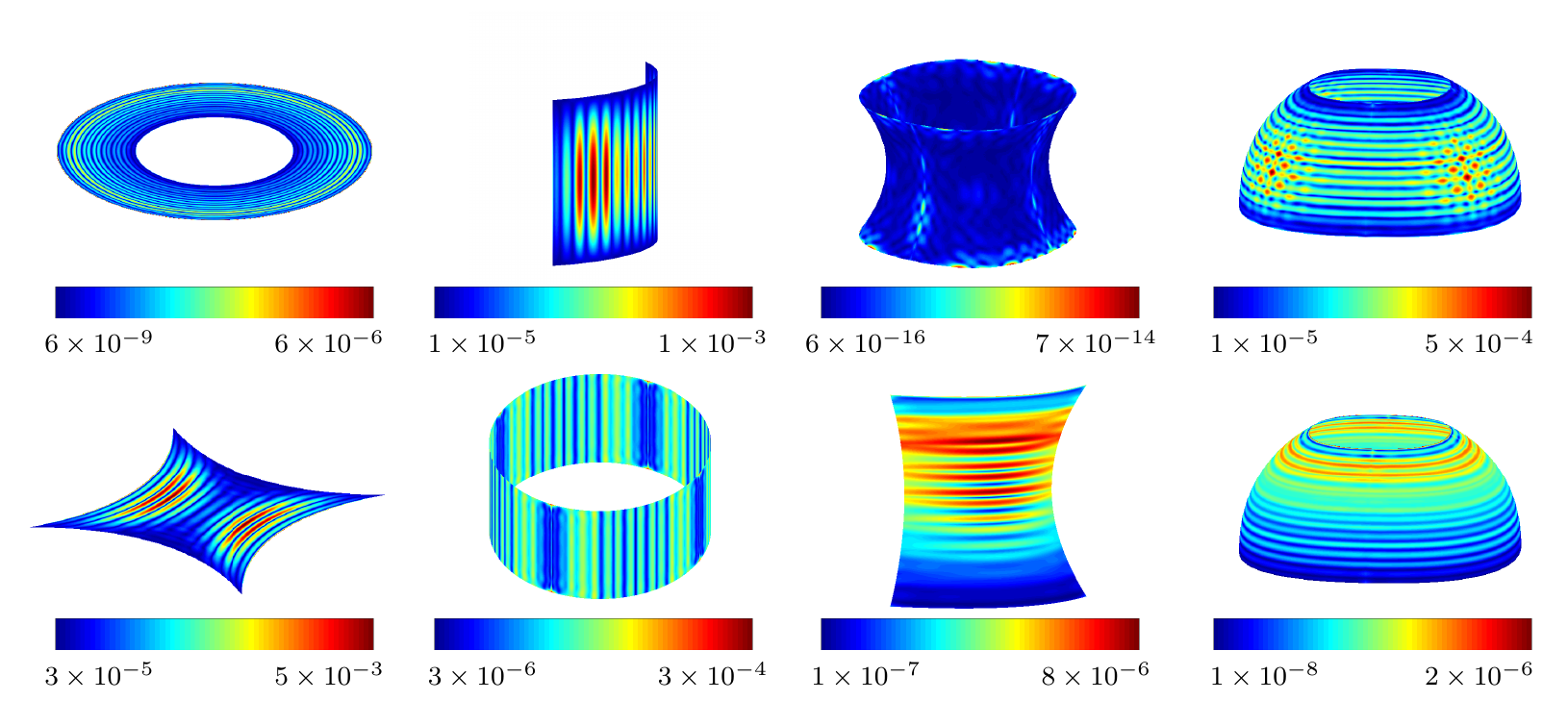}
    \caption{Relative Shell Energy Errors}
  \end{subfigure}
\caption{(a) The relative displacement errors and (b) the relative shell energy errors visualized over the undeformed geometries for our new shell obstacle course for a $p=4$, $16 \times 16$-element mesh. Each column represents a geometric class of problems where, from left to right, we have flat, parabolic, hyperbolic, and elliptic geometries. From top to bottom, in the first column are the annular domain and the astroid, in the second column are the quarter cylinder and the full cylinder, in the third column are the inflated hyperbolic paraboloid and the hyperbolic paraboloid diving board, and in the fourth column are the inflated hemispherical shell and the stretched hemisphere.}
\label{fig:Disp_and_energy_Errs}
\end{figure}

\subsection{Displacement and Energy Errors}

Another benefit of our new shell obstacle manufactured solutions suite is the capability to visualize the pointwise displacement and energy density errors throughout the geometric domain. This ability is exceptionally useful for understanding various discretization methods and how errors accrue accordingly. To illustrate this, the displacement errors and energy errors are plotted in Figure~\ref{fig:Disp_and_energy_Errs}. As is shown in these figures, the error is quite oscillatory throughout the domain, but the amplitude of the oscillations is bounded on the order of the discretization error.

\subsection{Comparison to Variationally Inconsistent Nitsche-Based Formulation}

To convey the importance of variational consistency, we have included an additional set of numerical experiments in this subsection. In these experiments, we compare our formulation and discretization to one using the incorrect ersatz forces that were described in Remark~\ref{rem:IncorrectErsatz} and, more specifically, from those presented in \cite[p.155]{Ciarlet2005} and in \cite[p.156]{Koiter1973foundations}. The resulting Nitsche-based formulation is identical in form to the formulations proposed in \cite{guo2015weak,guo2015nitsche} up to corner forces. The results of this experiment are shown in Figure~\ref{fig:bad_results}. To reiterate the differences between these results and those in previous subsections, we employ the incorrect bending component of the ersatz force $\surfVec{\ersatz}^{(\bendStress)}({\bf w}) = - 2 \surfTens{\SFF} \cdot \surfTens{\bendStress}({\bf w}) \cdot \surfVec{\bdyNormal}$ instead of our derived forces $\surfVec{\ersatz}^{(\bendStress)}({\bf w}) = - \surfTens{\SFF} \cdot \left( \surfTens{\bendStress}({\bf w}) \cdot \surfVec{\bdyNormal} + \surfVec{\bdyTangent} \moment_{nt}({\bf w}) \right)$. Asymptotic convergence rates of roughly $0.5$ and $1.5$ are observed in the energy norm and the standard $L^2$-norm, respectively, regardless of polynomial degree. The deteriorated convergence rates are due to the inconsistency of the underlying formulation that renders the effectiveness of the formulation no better than a classical penalty method. In fact, the convergence rates presented in Figure~\ref{fig:bad_results} agree with theoretically-expected convergence rates from such a penalty formulation \cite[Thm. 5.2]{graser2019discretization}. Note the reference lines and slope parameters are computed using only the tail of the data where the arrested rates begin to highlight this observation.

\begin{figure}[t!]
  \includegraphics{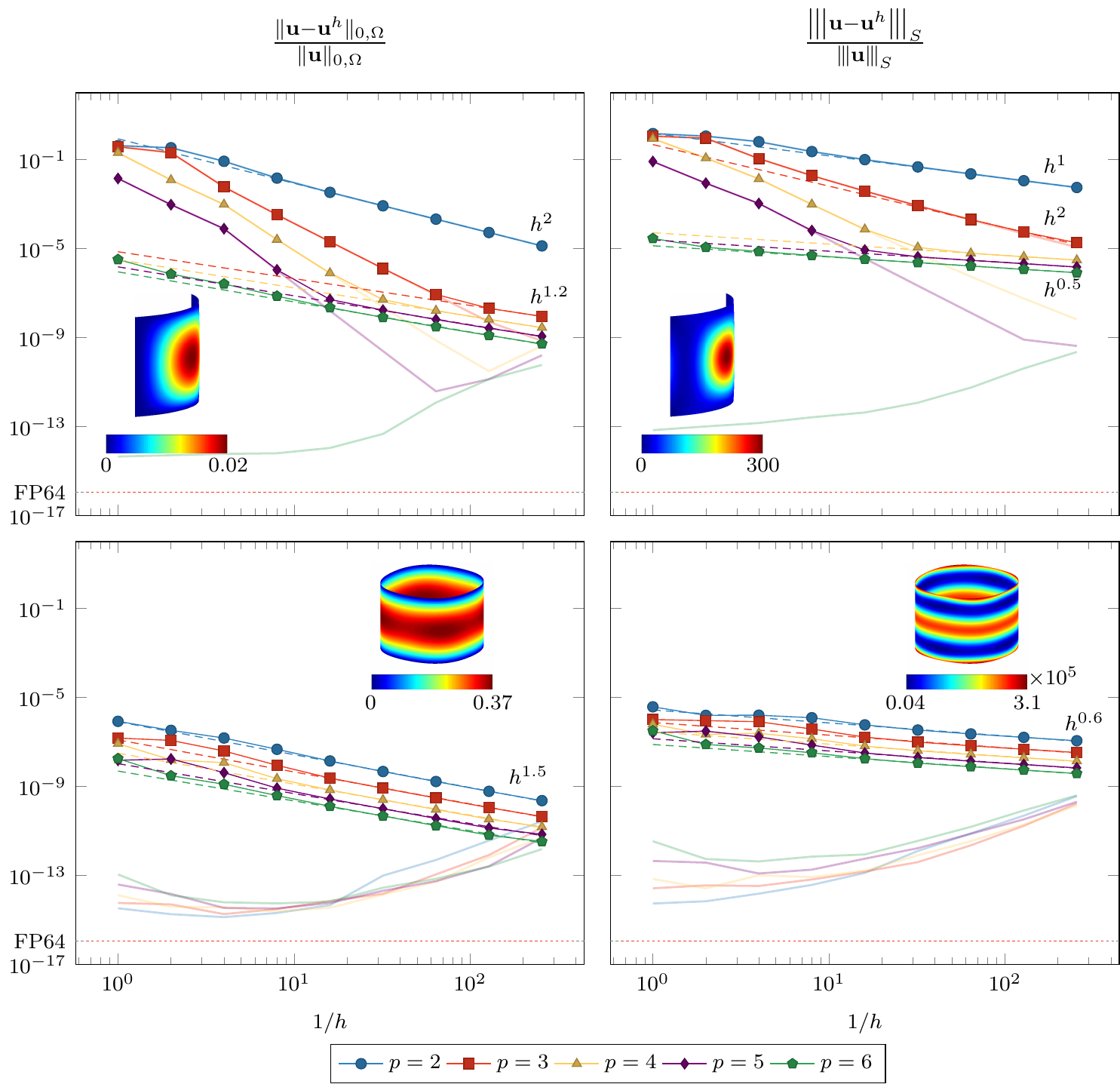}
  \caption{The convergence behavior of two selected problems from our linear shell obstacle course using the variationally inconsistent Nitsche-based formulation. The results for Problem 3 are shown in the $L^2$-norm (top left) and energy norm (top right) and those for Problem 5 are shown in the $L^2$-norm (bottom left) and energy norm (bottom right). The results from our discretization are shown transparently in the background for reference. The convergence behavior of the incorrect discretization are shown by the lines with filled markers. The dashed lines of matching color are a linear fit of the tail-end of the data.}
  \label{fig:bad_results}
\end{figure}

The variationally inconsistent formulation results in arrested convergence rates after a certain critical mesh size where the boundary error dominates the total error. This is especially clear for the Problem 3 results in Figure~\ref{fig:bad_results} where optimal convergence rates are obtained initially until they are eventually inhibited. This observation is also readily seen in \cite[Fig.7]{guo2015weak} where initial, optimal convergence rates begin to taper off at the last available data point. The Problem 5 results in Figure~\ref{fig:bad_results} clearly illustrate the underlying formulation is inconsistent because the manufactured solution lies in the span of each discretization and yet the error is not at machine precision. This is also the case for the $p=6$ discretization for Problem 3.


\section{Conclusion}
\label{sec:conclusion}
In this paper, we have presented a new Nitsche-based formulation for the linear Kirchhoff-Love shell that is provably stable and optimally convergent for general sets of admissible boundary conditions.  To arrive at our formulation, we first presented a systematic framework for constructing Nitsche-based formulations for variational constrained minimization problems.  We proved that this framework yields a well-posed and convergent Nitsche-based formulation provided that a generalized Green's identity and generalized trace and Cauchy-Schwarz inequalities are available.  We then applied this framework to the linear Kirchhoff-Love shell and, for the particular case of NURBS-based isogeometric analysis, we proved that the resulting formulating Nitsche-based formulation yields optimal convergence rates in both the shell energy norm and the standard $L^2$-norm.  To arrive at this formulation, we derived the Euler-Lagrange equations for general sets of admissible boundary conditions, and we discovered that the equations typically presented in the literature are incorrect.  To verify our Nitsche-based formulation, we constructed a linear shell obstacle course encompassing flat, parabolic, hyperbolic, and elliptic geometric configurations subject to clamped, simply supported, symmetric, and free boundary conditions.  For all examples, we used NURBS to discretize the governing equations, and we demonstrated that optimal convergence rates are obtained in both the shell energy norm and the standard $L^2$-norm for polynomial degrees $p = 2$ through $p = 6$.  We also demonstrated that a variationally inconsistent Nitsche-based formulation based on the incorrect Euler-Lagrange equations typically presented in the literature yields sub-optimal convergence rates of 0.5 and 1.5 in the shell energy norm and standard $L^2$-norm, respectively.

As discussed in Section~\ref{sec:num_results}, it is necessary to manufacture forcing functions that yield known displacement fields in order to confirm optimal convergence rates. This process is extremely non-trivial due to the inherent complexity of the PDE governing the Kirchhoff-Love shell. Historically, shell discretizations have been verified through ad hoc and unreliable means such as pointwise measures of convergence to values that are not backed by theory. Although these methods may show that a given discretization ultimately approaches an agreed-upon value, it has no notion of the rates at which it converges, nor the rigor associated with error estimates in standard Sobolev norms. To enable future researchers to rigorously validate their results, we have released the eight problems in our extended shell obstacle course in a supplemental notebook file. To the best of our knowledge, there does not otherwise exist a comprehensive suite of validation problems for shells that (i) encompasses all possible geometric classifications, (ii) considers all admissible boundary condition configurations, and (iii) serves as a tool for confirming optimal convergence behaviors. We therefore believe our new linear shell obstacle course stands as a valuable contribution to the research community on its own.

Since our framework for constructing Nitsche-based formulations is general and applicable to more complex problems, we plan to extend our methodology to Kirchhoff-Love shells with both geometric and material nonlinearities. To this end, we plan to extend our new shell obstacle course to this setting as well to serve as another validation tool for the shell community. We also plan to explore alternative discretization strategies and, in particular, Catmull-Clark subdivision spline discretizations with extraordinary vertices.  Finally, we plan to extend our methodology to the weak enforcement of displacement and normal rotation along patch interfaces for non-conforming multi-patch NURBS geometries and along trimming curves for trimmed NURBS geometries.

\appendix


\section{Differential Geometry}
\label{sec:Appendix_Diff_Geo}

For clarity, these appendices outline the various components needed for implementation of the methods developed in this paper. We begin with an outline of the differential geometry concepts required for such an exposition. In the following, we employ Einstein notation, i.e., repeated high-low indices have an induced summation. Latin indices (e.g., $i,j,k$) take values 1, 2, and 3, while Greek indices (e.g., $\alpha,\beta,\lambda$) take values 1 and 2. A comma preceding one or more indices denotes derivatives with respect to those indices.


Let $\physDomain$ denote the midsurface of our physical domain, an arbitrary differentiable manifold immersed in $\R^3$ with Lipschitz-continuous boundary $\partial \physDomain$. Accordingly, we define a sufficiently smooth geometric mapping $\undef{\midsurfMap}: \parDomain \rightarrow \physDomain$. Through derivatives of this geometric mapping, we define in-plane covariant tangent vectors along the convective parametric coordinates
\begin{equation*}
\aVec_\alpha(\pVariable^1,\pVariable^2) = \undef{\midsurfMap}_{,\alpha}(\pVariable^1,\pVariable^2).
\end{equation*}
We henceforth drop the explicit dependence on $\pVariable^1$ and $\pVariable^2$ for notational ease. The covariant metric coefficients are then given by
\begin{equation}
\aMet_{\alpha\beta} = \aVec_\alpha \cdot \aVec_\beta.
\label{eqn:covMetCoeff}
\end{equation}
Note that the metric tensor is symmetric, since the Euclidean inner product is symmetric (i.e., $\aMet_{\alpha\beta} = \aMet_{\beta\alpha}$).

We construct the algebraically dual, contravariant vector basis through the satisfaction of the Kronecker delta relationship
\begin{equation*}
\aVec^\alpha \cdot \aVec_\beta = \delta^\alpha_\beta,
\end{equation*}
elucidating the fact that the covariant and contravariant metric coefficients are related through their collective inverses, i.e.,
\begin{equation}
\left[ \aMet^{\alpha\beta} \right] = \left[ \aMet_{\alpha\beta} \right]^{-1},
\label{eqn:conMetCoeff}
\end{equation}
where $\left[ \cdot \right]$ denotes the matrix of components. Accordingly, the contravariant vector basis is defined via
\begin{equation*}
\aVec^\alpha = \aMet^{\alpha\mu} \aVec_\mu.
\end{equation*}
In general, the metric coefficients permit the ``raising'' and ``lowering'' of tensor component indices.

The in-plane covariant vectors define a basis for the tangent bundle of the manifold, and through the vector cross product, we can define a surface normal director:
\begin{equation*}
\aVec_3 = \frac{\aVec_1 \times \aVec_2}{\fullNorm{\aVec_1 \times \aVec_2 }_2},
\end{equation*}
where in this instance, $\fullNorm{\cdot}_2$ denotes the standard Euclidean norm. Note, by this definition, that (i) the normal director is always orthogonal to the in-plane vectors (i.e., $\aMet_{3\alpha} = \aMet_{\alpha3} = 0$) and (ii) the normal director always has unit length (i.e., $\aMet_{33} = 1$). Consequently, the covariant and contravariant midsurface normal directors are identical (i.e., $\aVec_3 = \aVec^3$). The derivatives of the in-plane vectors are given by higher-ordered derivatives of the geometric mapping:
\begin{equation*}
\aVec_{\alpha,\,\beta} = \undef{\midsurfMap}_{,\alpha\beta}  \hspace{20pt} \text{and} \hspace{20pt} \aVec_{\alpha,\,\beta\lambda}= \undef{\midsurfMap}_{,\alpha\beta\lambda}.
\end{equation*}
For our purposes, we assume that these derivatives are symmetric due to an assumed sufficient differentiability of $\undef{\midsurfMap}$. Additionally, $\aVec_{3,3} = 0$ by the inextensibility of the normal director.

The covariant components of the curvature tensor are given by
\begin{equation*}
\bCurv_{\alpha\beta} = \aVec_3 \cdot \aVec_{\alpha,\,\beta}
\end{equation*}
and the covariant components of the third fundamental form are given by the composition
\begin{equation*}
\cCurv_{\alpha\beta} = \bCurv^\lambda_{\phantom{\lambda} \alpha} \bCurv_{\lambda\beta},
\end{equation*}
where the mixed components of the curvature tensor arise through the relationship
\begin{equation}
  \bCurv^\alpha_{\phantom{\alpha} \beta} = \aMet^{\alpha\lambda}\bCurv_{\lambda\beta}.
  \label{eqn:secondFF}
\end{equation}
Both the curvature tensor and the third fundamental form are symmetric.  The \textbf{\emph{Christoffel symbols of the second kind}} are defined by
\begin{equation}
\christoffel^\lambda_{\alpha\beta} = \aVec^\lambda \cdot \aVec_{\alpha,\,\beta}.
\label{eqn:Christoffel}
\end{equation}
Note that these entities exhibit a symmetry between their lower indices, i.e., $\christoffel^\lambda_{\alpha\beta} = \christoffel^\lambda_{\beta\alpha}$. The components of the ersatz forces require derivatives of the Christoffel symbols, which are given via
\begin{equation}
  \Gamma^\lambda_{\alpha\beta,\mu} = -\christoffel^\lambda_{\nu\mu}\christoffel^\nu_{\alpha\beta} + b_{\alpha\beta} b^\lambda_\mu + \aVec^\lambda \cdot \aVec_{\alpha,\,\beta\mu}.
  \label{eqn:dChristoffel}
\end{equation}

Let $\surfVec{\bdyNormal}$ be the outward-facing unit normal and $\surfVec{\bdyTangent}$ be the positively oriented, counterclockwise unit tangent vector to $\physBoundary$. Note that $\surfVec{\bdyNormal}$ is the normal to $\physBoundary$ and should not be confused with the midsurface normal director, denoted $\aVec_3$, which coincidentally also dictates the positive orientation of $\surfVec{\bdyTangent}$. These boundary quantities are defined via the non-normalized normal and tangent vectors denoted $\surfVec{\nu}$ and $\surfVec{s}$, respectively.

Beginning with the non-normalized tangent vector $\surfVec{s}$, we first select $s^\alpha$ in any manner such that $\surfVec{s} = s^\alpha \aVec_\alpha$ is aligned with the domain boundary. For convenient parameterizations, such as those isomorphic to rectangular parametric domains, these components can be selected to align with the parametric edges, e.g., $\left[ s^\alpha \right] = (1,0)$ for the south boundary, etc.; however, our exposition here is not limited to such parameterizations. The corresponding covariant components of this vector are given by $s_\alpha = \aMet_{\alpha\beta}s^\beta$ and the boundary Jacobian used for integration is given by $\| \surfVec{s} \| = \sqrt{s^\lambda s_\lambda}$. The covariant and contravariant components of the \textbf{\emph{normalized}} boundary tangent vector are then given by
\begin{equation}
  t_\alpha = \frac{s_\alpha}{\sqrt{s^\lambda s_\lambda}} \hspace{10pt} \text{and} \hspace{10pt} t^\alpha = \frac{s^\alpha}{\sqrt{s^\lambda s_\lambda}},
  \label{eqn:tan_comps}
\end{equation}
respectively. Since we have assumed that the contravariant components $s^\alpha$ are known \textit{a priori}, we accordingly expect that the coordinate derivatives of this field $s^\alpha_{,\,\beta}$ and, hence, the covariant derivative $s^\alpha_{|\beta} = s^\alpha_{,\,\beta} + \Gamma^\alpha_{\lambda\beta} s^\lambda$ are also known. The covariant derivative of the covariant components of the non-normalized tangent vector are then given via an ``index-lowering'' operation of the contravariant counterpart:
\begin{equation*}
  s_{\alpha|\beta} = \aMet_{\alpha\lambda} s^\lambda_{|\beta}.
\end{equation*}
Lastly, the covariant derivative of the covariant components of the unit tangent vector to the boundary are given by
\begin{equation*}
  t_{\alpha|\beta} = \frac{s_{\alpha|\beta}}{\sqrt{s^\lambda s_\lambda}} - \frac{t_\alpha \left( t^\lambda s_{\lambda|\beta} \right)}{\sqrt{s^\lambda s_\lambda}}.
\end{equation*}

Next, we discuss the components of and various derivatives of components of the normal vector. The normal vector is defined to be (i) outward-facing and (ii) normal to the boundary. The normal vector defined through the vector cross product $\surfVec{n} = \surfVec{t} \times \aVec_3$ can be shown to satisfy these properties. The covariant components of the normal vector are given via
\begin{equation}
  \left( \begin{array}{c}
  n_1\\
  n_2
\end{array} \right) = |\aMet| \left( \begin{array}{c}
  t^2\\
  -t^1
  \end{array} \right),
  \label{eqn:d_nor_comp}
\end{equation}
where $|\aMet| = \det(a_{\alpha\beta}) = \sqrt{\aMet_{11} \aMet_{22} - 2\aMet_{12}}$. The contravariant components of the normal vector are then given by the ``index-raising'' operation $n^\alpha = \aMet^{\alpha\beta} n_\beta$. The derivative of the covariant components of the normal are given by
\begin{equation*}
  \left( \begin{array}{c}
  n_{1,\,\beta}\\
  n_{2,\,\beta}
\end{array} \right) = |\aMet| \left( \begin{array}{c}
    t^2_{,\,\beta}\\
    -t^1_{,\,\beta}
  \end{array} \right) + |\aMet|_{,\,\beta} \left( \begin{array}{c}
      t^2\\
      -t^1
      \end{array} \right),
\end{equation*}
where
\begin{equation*}
  |\aMet|_{,\,\beta} = \christoffel^\lambda_{\lambda\beta}|\aMet|,
\end{equation*}
and the coordinate derivatives of the tangent vector used in \eqref{eqn:d_nor_comp} can be obtained by utilizing the relationship $t^\alpha_{,\,\beta} = t_{\alpha|\beta} + \christoffel^\lambda_{\alpha\beta} t_\lambda$. Finally, the covariant derivative of the covariant components of the normal is given by
\begin{equation*}
  n_{\alpha|\beta} = n_{\alpha,\beta} - \christoffel^\lambda_{\alpha\beta}n_\lambda.
\end{equation*}


\section{Continuum Mechanics}
\label{sec:Appendix_Cont_Mech}

In this appendix, we provide a brief discussion of continuum mechanics, particularly the stress and strain measures employed throughout this paper. We use \textbf{\emph{upper-case}} letters to denote entities defined over the shell body, e.g., ${\bf U} \colon \mathcal{B} \rightarrow \mathbb{R}^3$, and \textbf{\emph{lower-case}} letters to denote entities defined over the shell midsurface, e.g., ${\bf u} \colon \physDomain \rightarrow \mathbb{R}^3$. Furthermore, calligraphic fonts refer to entities pertaining to a \textbf{\emph{deformed configuration}} in contrast to all else that refers to the \textbf{\emph{undeformed configuration}}.

The differential-geometric discussion in~\ref{sec:Appendix_Diff_Geo} pertains to the two-dimensional ``midsurface'' of the shell body. However, when considering strain measurements of a three-dimensional body $\undef{\shellBody}$, we accordingly need a parameterization in which to perform analysis. Utilizing the $\aVec{}$-frame presented therein,
\begin{equation*}
\undef{\shellMap}{(\pVariable^1,\pVariable^2,\pVariable^3)} = \undef{\midsurfMap}(\pVariable^1,\pVariable^2) + \pVariable^3 \aVec_3(\pVariable^1,\pVariable^2)
\end{equation*}
suffices as a suitable parameterization for the shell body. Here, $\pVariable^1,\pVariable^2$ are the in-plane convective coordinates and $\pVariable^3 \in \left[ -\sfrac{\thickness}{2}, \sfrac{\thickness}{2} \right]$ denotes a ``thickness'' direction of the shell, oriented along $\aVec_3$. Analogously to the midsurface, we define a curvilinear frame for $\undef{\shellBody}$ that we utilize for the computation of strain tensors. These vectors are given by
\begin{equation*}
\undef{\gVec}_\alpha = \undef{\shellMap}_{,\alpha} = \left( \delta_\alpha^\beta - \pVariable^3 \bCurv^\beta_\alpha \right) \aVec_\beta \hspace{20pt} \text{and} \hspace{20pt} \undef{\gVec}_3 = \undef{\shellMap}_{,3} = \aVec_3.
\end{equation*}
Observe, in the absence of curvature (e.g., $\surfTens{\bCurv} = \surfTens{0}$), that the $\fullTens{\aMet}$ and $\undef{\gVec}$ coordinate frames coincide throughout the shell body.

Continuum mechanics is grounded in elasticity theory, where strain measures, constitutive models, and Newton's $2^{nd}$ law serve as a surrogate to the displacement field. In particular, external loadings determine internal stresses that, through a constitutive relationship, expose an induced straining in the elastic body that is defined through various derivatives of the displacement field. A wide range of strain measures exist; however, for the purposes of this paper, we employ the linearized Green-Lagrange strain measure because it is energetically conjugate to the $2^{nd}$ Kirchhoff-Piola stress tensor. This provides a natural foundation in variational formulations that seek the displacement configuration that minimizes the potential energy.

We assume there exists a smooth, bijective mapping $\mathcal{X}(\bm{\pVariable}) = \mathcal{X}(\undef{\shellMap}(\bm{\pVariable}))$ between $\mathcal{X}$, the deformed configuration, and $\undef{\shellMap}$, the undeformed configuration. Through the derivatives of this mapping, we obtain the entity known as the \textbf{\emph{deformation gradient}}, which is given by
\begin{equation}
{\bf F} \equiv \frac{\partial \mathcal{X}}{\partial \undef{\shellMap}} = \frac{\partial \mathcal{X}}{\partial \pVariable^i} \frac{\partial \pVariable^i}{\partial \undef{\shellMap}} = \bm{\mathcal{G}}_i \otimes \undef{\gVec}^i.
\label{eqn:DefGrad}
\end{equation}
The Green-Lagrange strain measure is given by
\begin{equation*}
{\bf E} = \frac{1}{2}\left( {\bf R} - {\bf I} \right) = \frac{1}{2}\left( {\bf F}^T{\bf F} - {\bf I} \right),
\label{eqn:GLStrain}
\end{equation*}
where ${\bf R} = {\bf F}^T{\bf F}$ is the \textbf{\emph{Right Cauchy strain tensor}}. Utilizing \eqref{eqn:DefGrad}, we obtain the following expression:
\begin{equation}
{\bf E} = \frac{1}{2}\left( (\undef{\gVec}^i \otimes \bm{\mathcal{G}}_i) \cdot (\bm{\mathcal{G}}_{j} \otimes \undef{\gVec}^{j}) - {\bf I} \right) = \frac{1}{2}\left( \mathcal{G}_{ij} - G_{ij} \right) \undef{\gVec}^{i} \otimes \undef{\gVec}^{j} \equiv E_{ij} \undef{\gVec}^{i} \otimes \undef{\gVec}^{j},
\label{eqn:GLStrain2}
\end{equation}
where $\mathcal{G}_{ij}$ and $G_{ij}$ are the covariant metric coefficients associated with the deformed and undeformed shell body, respectively. Observe that the coefficients $E_{ij}$ are in fact a function of the displacement field since
\begin{equation*}
\bm{\mathcal{G}}_{i} = \frac{\partial \mathcal{X}}{\partial \pVariable^i} = \frac{\partial \left( \undef{\shellMap} + \fullTens{\shellDisp} \right)}{\partial \pVariable^i} = \undef{\gVec}_{i} +  \fullTens{\shellDisp}_{,i}.
\end{equation*}
In light of this, the deformed metric tensor is given via
\begin{equation*}
\mathcal{G}_{ij} = \bm{\mathcal{G}}_{i} \cdot \bm{\mathcal{G}}_{j} = ( \gVec_{i} + \fullTens{\shellDisp}_{,i} ) \cdot ( \gVec_{j} + \fullTens{\shellDisp}_{,j} ) = G_{ij} + \fullTens{\shellDisp}_{,i} \cdot \undef{\gVec}_{j} + \fullTens{\shellDisp}_{,j} \cdot \undef{\gVec}_{i} + \fullTens{\shellDisp}_{,i} \cdot \fullTens{\shellDisp}_{,j}.
\end{equation*}
Combining this expression with \eqref{eqn:GLStrain2} yields the following components of the Green-Lagrange strain tensor:
\begin{equation*}
E_{ij}(\fullTens{\shellDisp}) = \frac{1}{2} \left( \fullTens{\shellDisp}_{,i} \cdot \undef{\gVec}_{j} + \fullTens{\shellDisp}_{,j} \cdot \undef{\gVec}_{i} + \fullTens{\shellDisp}_{,i} \cdot \fullTens{\shellDisp}_{,j} \right).
\end{equation*}
Neglecting the nonlinear dependencies, the linearized Green-Lagrange strain tensor coefficients reduce to
\begin{equation}
\GLStrain_{ij}(\fullTens{\shellDisp}) = \frac{1}{2} \left( \fullTens{\shellDisp}_{,i} \cdot \undef{\gVec}_{j} + \fullTens{\shellDisp}_{,j} \cdot \undef{\gVec}_{i} \right),
\label{eqn:GL_Coeff}
\end{equation}
arriving at the familiar symmetrized displacement gradient that governs linear elastic phenomena. Next, we define for vector fields ${\bf w}: \physDomain \rightarrow \mathbb{R}^3$ the \textbf{\emph{surface gradient operator}} \cite[\S 95]{Brand1947} to be
\begin{equation}
 \surfVec{\nabla} {\bf w} = \frac{\partial {\bf w}}{\partial \xi^\lambda} \otimes \aVec^\lambda.
 \label{eqn:surfGradDef}
\end{equation}
The surface gradient is ubiquitous throughout our exposition because it is the primary tool used for constructing strain fields over manifolds.  It should be noted that, for vector fields ${\bf W}: \mathcal{B} \rightarrow \mathbb{R}^3$, the surface gradient of the vector field along the midsurface is precisely
\begin{equation*}
 \surfVec{\nabla} {\bf W} = \nabla {\bf W} \cdot \surfTens{\Projector} = \nabla {\bf W} - \frac{\partial {\bf W}}{\partial \xi^3} \otimes \gVec^3,
\end{equation*}
where $\nabla {\bf W} = \frac{\partial {\bf W}}{\partial \xi^i} \otimes \gVec^i$ is the \textit{\textbf{full gradient}} of ${\bf W}$ and $\surfTens{\Projector} = {\bf I} - \aVec^3 \otimes \aVec_3$ is the \textit{\textbf{in-plane projector}}, with $\textbf{I}$ denoting the identity tensor. Note by the symmetry of $\surfTens{\Projector}$ and the construction of the midsurface normal director, i.e., $\aVec_3 = \aVec^3$, the in-plane projector also satisfies the definition $\surfTens{\Projector} = {\bf I} - \aVec_3 \otimes \aVec^3$. Akin to the surface gradient is the \textbf{\emph{surface divergence}} operator, which is defined as
\begin{equation*}
 \surfVec{\nabla} \cdot {\bf w} = \frac{\partial {\bf w}}{\partial \xi^\lambda} \cdot \aVec^\lambda.
\end{equation*}

The Kirchhoff-Love shell model considered in this paper employs the following displacement profile:
\begin{equation}
  \fullTens{\shellDisp}(\pVariable^1,\pVariable^2,\pVariable^3) = \fullTens{\midsurfDispTrial}(\pVariable^1,\pVariable^2) + \pVariable^3 \surfVec{\midsurfRot}(\pVariable^1,\pVariable^2)
\label{eqn:disp_profiles}
\end{equation}
Physically, $\fullTens{\midsurfDispTrial}(\pVariable^1,\pVariable^2)$ is a translational displacement, while $\surfVec{\midsurfRot}(\pVariable^1,\pVariable^2)$ is a rotational displacement of the midsurface normal director $\aVec^3$. It is worth mentioning that the rotational displacement $\surfVec{\midsurfRot}$ can be represented by a surface tensor due to the inextensibility of the midsurface normal director $\aVec^{3}$.

We now proceed with deriving the corresponding strain measures from these prescribed displacement profiles. In particular, we substitute the assumed displacement field \eqref{eqn:disp_profiles} into \eqref{eqn:GL_Coeff}, giving rise to Table \ref{table:VariousStrains} that shows the explicit decomposition of the strain fields into a membrane and bending strain mode as well as a transverse shear strain mode.

\begin{table}[h]
\centering
\begin{tabular}{rc}
membrane and bending straining & $\Sym\left(\surfVec{\nabla} \fullTens{\midsurfDispTrial} + \pVariable^3 \left( \surfVec{\nabla} \ \surfVec{\midsurfRot} - \surfTens{\bCurv} \cdot \surfVec{\nabla} \fullTens{\midsurfDispTrial} \right)\right)$ \\
transverse shearing & $\frac{1}{2} \left( \aVec_{3} \cdot \surfVec{\nabla} \fullTens{\midsurfDispTrial} + \surfVec{\midsurfRot} \right)$
\end{tabular}
\caption{The strain fields that arise from the linearized Green-Lagrange strain with the assumed Kirchhoff-Love shell displacement profile. Note we have dropped the $\mathcal{O} \left( \left( \pVariable^3 \right)^2 \right)$ term that contained $\Sym(\surfTens{\bCurv} \cdot \surfVec{\nabla} \surfVec{\midsurfRot})$ due to our ansatz of a \textbf{\emph{straight}} deformed midsurface normal director.}
\label{table:VariousStrains}
\end{table}

In the table, the operator $\Sym(\cdot)$ returns the symmetric part of a tensor.  That is, $\Sym(\textbf{A}) = \frac{1}{2} \left( A_{ij} + A_{ji} \right) \aVec^i \otimes \aVec^j$ for a tensor $\textbf{A} = A_{ij} \aVec^i \otimes \aVec^j$.

However, we are unfinished with the derivation of the strains. The Kirchhoff-Love kinematical assumption is that material normals remain straight and normal to the deformed material, i.e., the transverse shear strain vanishes. This introduces a constraint on the rotational degrees of freedom as follows:
\begin{equation}
  \surfVec{\midsurfRot}({\bf u}) = - \aVec_{3} \cdot \surfVec{\nabla} \fullTens{\midsurfDispTrial}.
\label{eqn:KL_constraint}
\end{equation}
Substituting this constraint back into the strain profile and invoking the additional assumption that the strain fields are purely in-plane yields the following strain fields for respective membrane and bending actions:
\begin{equation}
  \surfTens{\memStrain}(\fullTens{u}) = \surfTens{\Projector} \cdot \Sym\left(\surfVec{\nabla} \fullTens{\midsurfDispTrial}\right) \cdot \surfTens{\Projector}
\label{eqn:mem_strain}
\end{equation}
and
\begin{equation}
  \surfTens{\bendStrain}(\fullTens{u}) = - \surfTens{\Projector} \cdot \Sym\left(\aVec_{3} \cdot \surfVec{\nabla} \ \surfVec{\nabla} \fullTens{\midsurfDispTrial}\right) \cdot \surfTens{\Projector}.
\label{eqn:bend_strain}
\end{equation}
Lastly, it follows that we can decompose the linearized Green-Lagrange strain tensor into membrane- and bending-components:
\begin{equation}
\surfTens{\GLStrain}(\fullTens{\midsurfDispTrial}) = \surfTens{\memStrain}(\fullTens{\midsurfDispTrial}) + \pVariable^3 \surfTens{\bendStrain}(\fullTens{\midsurfDispTrial}).
\label{eqn:strain_decomposition}
\end{equation}

At this point, it is worthwhile to define the energetically conjugate stresses for both the membrane and bending strain tensors. Since both of these quantities require a material law in their definition, we present the following elasticity tensor arising from a linear constitutive model that we employ in our mathematics and numerics:
\begin{equation}
\mathbb{\constitutive} = \mathbb{\constitutive}^{\alpha\beta\lambda\mu} \aVec_{\alpha} \otimes \aVec_{\beta} \otimes \aVec_{\lambda} \otimes \aVec_{\mu} \ \ \text{where} \ \ \mathbb{\constitutive}^{\alpha\beta\lambda\mu} = \frac{E}{2(1+\nu)}\left( \aMet^{\alpha\lambda}\aMet^{\beta\mu} + \aMet^{\alpha\mu}\aMet^{\beta\lambda} + \frac{2\nu}{1-\nu}\aMet^{\alpha\beta}\aMet^{\lambda\mu} \right),
\label{eqn:constitutive}
\end{equation}
where $E$ is Young's modulus and $0 \le \nu \le \frac{1}{2}$ is Poisson's ratio.

This choice of constitutive law employs a ``plane stress'' assumption where transverse stresses are neglected, precisely the regime of interest and compatible with our model construction. Moreover, this choice of $\mathbb{\constitutive}$ is an approximation of the true  energetically conjugate elasticity tensor that is comprised of $G^{\alpha\beta}$, that is, the contravariant metric coefficients of the shell body, rather than $a^{\alpha\beta}$. This approximation is valid for sufficiently thin shells with gentle curvatures \cite{Ciarlet2000}. The primary goals of utilizing this assumption in our constitutive model is to induce a natural decoupling between the membrane and bending stresses, much like the the membrane and bending strains, and to allow for a simple integration through-thickness before discretization.

Through the tensor relationship $\surfTens{\sigma}(\fullTens{\midsurfDispTrial}) = \mathbb{\constitutive} \colon \surfTens{\GLStrain}(\fullTens{\midsurfDispTrial})$ between \eqref{eqn:constitutive} and \eqref{eqn:strain_decomposition}, we have the following decomposition of the Cauchy stress tensor:
\begin{equation*}
\surfTens{\sigma}(\fullTens{\midsurfDispTrial}) = \mathbb{\constitutive} \colon \surfTens{\memStrain}(\fullTens{\midsurfDispTrial}) + \pVariable^3 \mathbb{\constitutive} \colon \surfTens{\bendStrain}(\fullTens{\midsurfDispTrial}).
\label{eqn:stress_decomposition}
\end{equation*}
These entities give rise to the familiar relationship that is the strain energy that we have utilized throughout our example problems:
\begin{equation}
\Pi_\text{int}({\bf u}) = \frac{1}{2} \int_{\undef{\shellBody}} \surfTens{\sigma}\left(\fullTens{\midsurfDispTrial} \right) \colon \surfTens{\GLStrain}\left(\fullTens{\midsurfDispTrial}\right) \ d \undef{\shellBody} = \underbrace{ \frac{1}{2} \int_{\physDomain} \surfTens{\memStress}(\fullTens{\midsurfDispTrial}) \colon \surfTens{\memStrain}(\fullTens{\midsurfDispTrial}) \ d \physDomain }_{\text{membrane energy}} + \underbrace{ \frac{1}{2} \int_{\physDomain} \surfTens{\bendStress}(\fullTens{\midsurfDispTrial}) \colon \surfTens{\bendStrain}(\fullTens{\midsurfDispTrial}) \ d \physDomain }_{\text{bending energy}},
\label{eqn:strain_energy}
\end{equation}
where the membrane stress and bending stresses are given by
\begin{equation}
\surfTens{\memStress}(\fullTens{\midsurfDispTrial}) := \thickness \mathbb{\constitutive} \colon \surfTens{\memStrain}(\fullTens{\midsurfDispTrial})
\label{eqn:mem_stress}
\end{equation}
and
\begin{equation}
 \surfTens{\bendStress}(\fullTens{\midsurfDispTrial}) := \frac{\thickness^3}{12} \mathbb{\constitutive} \colon \surfTens{\bendStrain}(\fullTens{\midsurfDispTrial}),
\label{eqn:bend_stress}
\end{equation}
respectively\footnote{Classically, these entities are defined without the thickness variable $\thickness$. However, we find it convenient to define the stresses this way to condense the notation throughout. Furthermore, the linear and cubic dependencies on thickness are fundamental properties associated with \textbf{\emph{membrane}} and \textbf{\emph{bending}} phenomena, respectively.}. Observe that the cross-integrals containing the bending strain against the membrane stress and the membrane strain against the bending stress in \eqref{eqn:strain_energy} vanish due to the integration through-thickness of a linear $\pVariable^3$. Moreover, the $\thickness$-dependencies in \eqref{eqn:mem_stress} and \eqref{eqn:bend_stress} appears after the through-thickness integration of the elastic body performed before discretization, since we have assumed a constant thickness throughout.


\section{Component Representations for the Kirchhoff-Love Shell}
\label{sec:Appendix_Components}

The exposition throughout this paper pertains solely to tensorial quantities representing various strain and stress fields. Since we ultimately desire to represent these quantities numerically, we need the components that make up these tensors for discretization. Before proceeding, we define the basis transformation matrices that are used for discretization of our structural models since our displacement field is discretized in a global Cartesian frame. The following transformations map Cartesian components to covariant, curvilinear ones:
\begin{equation*}
\covTrans^i_\alpha = \cartBasis^i \cdot \aVec_\alpha \hspace{15pt} \text{and} \hspace{15pt} \covTrans^i_3 = \cartBasis^i \cdot \aVec_3.
\label{eqn:lambdaTrans}
\end{equation*}
We also require the derivatives of this transformation when discretizing the membrane and bending strains. These derivatives are given by
\begin{equation*}
\covTrans^i_{\alpha,\lambda} = \christoffel^\mu_{\alpha\lambda} \covTrans^i_\mu + \SFF_{\alpha\lambda}\covTrans^i_3 \hspace{15pt} \text{and} \hspace{15pt} \covTrans^i_{3,\lambda} = -\SFF^\mu_{\cdot\lambda} \covTrans^i_\mu.
\label{eqn:dlamdaTrans}
\end{equation*}
For example, if $\fullTens{\midsurfDispTrial} = \midsurfDispTrial_i \cartBasis^i = \tilde{\midsurfDispTrial}_i \aVec^i$, then $\tilde{\midsurfDispTrial}_i = \covTrans_i^j \midsurfDispTrial_j$. Note we have also implicitly introduced the $\cdot$ notation for tensor indices that denote the index ordering associated with tensorial components.

With these transformation matrices in place, we are ready to present the tensorial components of the various entities required for successful discretization. Note that for a general shell discretization, we use a Cartesian displacement field of the form ${\bf u} = u_i {\bf e}^i$. Beginning with membrane quantities, the covariant components of the membrane strain tensor and the corresponding contravariant components of the membrane stress tensor for a given basis function and degree of freedom are expressed as follows:

\begin{table}[h]
\centering
\begin{tabular}{cc}
Membrane Strain & Membrane Stress\\
\midrule
$\memStrain_{\alpha\beta}({\bf u}) = \frac{1}{2} \left[ \covTrans^i_\alpha u_{i,\beta} + \covTrans^i_\beta u_{i,\alpha} \right]$ & $\memStress^{\alpha\beta}({\bf u}) = \mathbb{\constitutive}^{\alpha\beta\lambda\mu} \memStrain_{\lambda\mu}({\bf u})$
\end{tabular}
\end{table}

\noindent Moreover, the boundary traction is given by
\begin{equation*}
\traction^\alpha({\bf u}) = \memStress^{\alpha\beta}({\bf u}) \bdyNormal_\beta.
\end{equation*}

Next, we present the various quantities pertaining to the bending action of the shell.  The rotational degrees of freedom are expressed through the negative surface gradient of this field, i.e., $\surfVec{\theta}({\bf u}) = - \covTrans_3^i u_{i,\lambda} \aVec^\lambda$. Therefore, the normal and twisting rotations are given by

\begin{table}[h]
\centering
\begin{tabular}{cc}
Normal Rotation & Twisting Rotation\\
\midrule
$\theta_n({\bf u}) = - \covTrans_3^i u_{i,\lambda} n^\lambda$ & $\theta_t({\bf u}) = - \covTrans_3^i u_{i,\lambda} t^\lambda$
\end{tabular}
\end{table}

Through this definition, we are able to represent the bending stresses and strains in terms of our Cartesian displacement field through the following relationships:

\begin{table}[h]
\centering
\begin{tabular}{cc}
Bending Strain & Bending Stress\\
\midrule
$\bendStrain_{\alpha\beta}({\bf u}) = - \frac{1}{2}\covTrans^i_3 \left[ u_{i|\alpha\beta} + u_{i|\beta\alpha} \right]$ & $\bendStress^{\alpha\beta}({\bf u}) = \mathbb{\constitutive}^{\alpha\beta\lambda\mu} \bendStrain_{\lambda\mu}({\bf u})$
\end{tabular}
\end{table}

\noindent where the covariant form of the second derivative of the Cartesian displacement field is defined via
\begin{equation*}
  u_{i|\alpha\beta} = u_{i,\alpha\beta} - \christoffel^\lambda_{\alpha\beta} u_{i,\lambda}.
\end{equation*}
The bending and twisting moments can be derived through this relationship since $\bendStress_{nn}({\bf u}) = n_\alpha n_\beta \bendStress^{\alpha\beta}({\bf u})$ and $\bendStress_{nt}({\bf u}) = n_\alpha t_\beta \bendStress^{\alpha\beta}({\bf u})$, respectively. Lastly, the ersatz forces defined in \eqref{eqn:KL_Shell_ersatz} have in-plane components given by

\begin{equation*}
  \ersatz^\alpha({\bf u}) = \traction^{\alpha}({\bf u}) - \SFF^\alpha_{\cdot\lambda} \bendStress^{\lambda \beta}({\bf u}) \bdyNormal_\beta - \moment_{nt}({\bf u}) \SFF^\alpha_{\cdot\lambda} \bdyTangent^\lambda.
\end{equation*}
The out-of-plane components of the ersatz forces are given by
\begin{equation*}
\ersatz^3({\bf u}) = \bdyNormal_\alpha \bendStress^{\alpha\beta}_{\cdot\cdot|\beta}({\bf u}) + \bdyNormal_{\alpha|\lambda} \bendStress^{\alpha\beta}({\bf u}) \bdyTangent_\beta \bdyTangent^\lambda + \bdyNormal_\alpha \bendStress^{\alpha\beta}_{\cdot\cdot|\lambda}({\bf u}) \bdyTangent_\beta \bdyTangent^\lambda + \bdyNormal_\alpha \bendStress^{\alpha\beta}({\bf u}) \bdyTangent_{\beta|\lambda} \bdyTangent^\lambda.
\end{equation*}
Note that $\mathbb{\constitutive}^{\alpha\beta\lambda\mu}_{\cdot\cdot\cdot\cdot|\nu} \equiv 0$ because the covariant derivative of the metric tensor vanishes since it is associated with the Levi-Civita connection. Therefore, the covariant derivative of the bending stress is given by
\begin{equation*}
  \bendStress^{\alpha\beta}_{\cdot\cdot|\nu}({\bf u}) = \mathbb{\constitutive}^{\alpha\beta\lambda\mu} \bendStrain_{\lambda\mu|\nu}({\bf u}),
\end{equation*}
where the covariant derivative of the bending strain is given by
\begin{equation*}
\bendStrain_{\alpha\beta|\nu}({\bf u}) = \frac{1}{2} \left[ \covTrans^i_\lambda \left( \SFF^\lambda_{\cdot\nu} u_{i|\alpha\beta} + \SFF^\lambda_{\cdot\nu} u_{i|\beta\alpha} \right) - \covTrans^i_3 \left( u_{i|\alpha\beta\mu} + u_{i|\beta\alpha\mu} \right) \right]
\end{equation*}
and the covariant form of the third derivative of the Cartesian displacement field is given by
\begin{equation*}
u_{i|\alpha\beta\mu}({\bf u}) = u_{i,\alpha\beta\mu} - u_{i,\lambda}\christoffel^\lambda_{\alpha\beta,\mu} - \christoffel^\lambda_{\alpha\beta}u_{i,\lambda\mu} - \christoffel^\lambda_{\alpha\mu} u_{i|\lambda\beta} - \christoffel^\lambda_{\beta\mu} u_{i|\alpha\lambda},
\end{equation*}
with the derivatives of the Christoffel symbols presented in \eqref{eqn:dChristoffel}.


\section{Geometric Parameterizations}
\label{sec:Geo_Param}

{\setlength\intextsep{0pt}
\begin{wrapfigure}{r}{0.2\textwidth}
  \centering
  \includegraphics{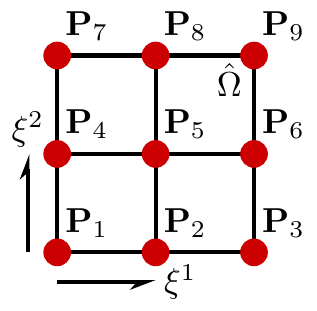}
  \caption{The parametric domain with control points denoted in red.}
  \label{fig:Geo_param_para}
\end{wrapfigure}

In this appendix, we provide the control points and NURBS weights associated with the geometries we considered in Section~\ref{sec:num_results}. Note that all geometries are parameterized using a one-element, biquadratic NURBS patch with a canonical ordering as shown in Figure~\ref{fig:Geo_param_para} over the parametric domain $\hat{\Omega}$. Experiments involving finer mesh resolutions and higher polynomial degrees are obtained through knot insertion and degree elevation, respectively, of this parameterization. The NURBS control points ${\bf P}_i = \left( P_{i,x}, P_{i,y}, P_{i,z} \right)$ and NURBS weights $w_i$ for each geometric configuration are given in Table~\ref{table:Geo_Params}, where the ordering of the points therein corresponds to the ordering of the control points shown in Figure~\ref{fig:Geo_param_para}. For quick reference, recall that Problem 1 is a quarter annulus, Problem 2 is an astroid, Problems 3 and 4 are quarter cylinders, Problems 5 and 6 are hyperbolic paraboloids, and Problems 7 and 8 are quarter-hemispherical shells with a hole.

\vspace{10pt}

\begin{table}[h!]
      \caption{NURBS control points and weights for geometries of Section~\ref{sec:num_results}.}
      \centering
      \begin{adjustbox}{width=\columnwidth,center}
        \begin{tabular}{c|cccc|cccc|cccc|cccc|cccc}
          \multicolumn{1}{c}{} & \multicolumn{4}{c}{Problem 1} & \multicolumn{4}{c}{Problem 2} & \multicolumn{4}{c}{Problem 3 \& 4} & \multicolumn{4}{c}{Problem 5 \& 6} & \multicolumn{4}{c}{Problem 7 \& 8} \vspace{3pt}\\
          $i$ & $P_{i,x}$ & $P_{i,y}$ & $P_{i,z}$ & $w_i$ & $P_{i,x}$ & $P_{i,y}$ & $P_{i,z}$ & $w_i$ & $P_{i,x}$ & $P_{i,y}$ & $P_{i,z}$ & $w_i$ & $P_{i,x}$ & $P_{i,y}$ & $P_{i,z}$ & $w_i$ & $P_{i,x}$ & $P_{i,y}$ & $P_{i,z}$ & $w_i$ \\
          \hline
          $1$ & $1$&$0$&$0$ & $1$ & $0$&$0$&$0$ & $1$ & $1$&$0$&$0$ & $1$ & $\sqrt{2}$&$0$&$-1$ & $1$ & $1$&$0$&$0$ & $1$\\
          $2$ & $\sfrac{1}{2}$&$0$&$0$ & $1$ & $\sfrac{1}{3}$&$\sfrac{1}{2}$&$0$ & $1$ & $1$&$1$&$0$ & $\sfrac{1}{\sqrt{2}}$ & $\sqrt{2}$&$\sqrt{2}$&$-1$ & $1$ & $1$&$0$&$\sfrac{1}{\sqrt{3}}$ & $1$\\
          $3$ & $2$&$0$&$0$ & $1$ & $0$&$1$&$0$ & $1$ & $0$&$1$&$0$ & $1$ & $0$&$\sqrt{2}$&$-1$ & $1$ & $\sfrac{1}{2}$&$0$&$\sfrac{\sqrt{3}}{2}$ & $1$\\
          $4$ & $1$&$1$&$0$ & $\sfrac{1}{\sqrt{2}}$ & $\sfrac{1}{2}$&$\sfrac{1}{3}$&$0$ & $1$ & $1$&$0$&$\sfrac{1}{2}$ & $1$ & $\sfrac{1}{\sqrt{2}}$&$0$&$0$ & $1$ & $1$&$1$&$0$ & $1$\\
          $5$ & $\sfrac{1}{2}$&$\sfrac{1}{2}$&$0$ & $\sfrac{1}{\sqrt{2}}$ & $\sfrac{1}{2}$&$\sfrac{1}{2}$&$0$ & $1$ & $1$&$1$&$\sfrac{1}{2}$ & $\sfrac{1}{\sqrt{2}}$ & $\sfrac{1}{\sqrt{2}}$&$\sfrac{1}{\sqrt{2}}$&$0$ & $1$ & $1$&$1$&$\sfrac{1}{\sqrt{3}}$ & $1$\\
          $6$ & $2$&$2$&$0$ & $\sfrac{1}{\sqrt{2}}$ & $\sfrac{1}{2}$&$\sfrac{1}{3}$&$0$ & $1$ & $0$&$1$&$\sfrac{1}{2}$ & $1$ & $0$&$\sfrac{1}{\sqrt{2}}$&$0$ & $1$ & $\sfrac{1}{2}$&$\sfrac{1}{2}$&$\sfrac{\sqrt{3}}{2}$ & $1$\\
          $7$ & $0$&$1$&$0$ & $1$ & $1$&$0$&$0$ & $1$ & $1$&$0$&$1$ & $1$ & $\sqrt{2}$&$0$&$1$ & $1$ & $0$&$1$&$0$ & $1$\\
          $8$ & $0$&$\sfrac{1}{2}$&$0$ & $1$ & $\sfrac{2}{3}$&$\sfrac{1}{2}$&$0$ & $1$ & $1$&$1$&$1$ & $\sfrac{1}{\sqrt{2}}$ & $\sqrt{2}$&$\sqrt{2}$&$1$ & $1$ & $0$&$1$&$\sfrac{1}{\sqrt{3}}$ & $1$\\
          $9$ & $0$&$2$&$0$ & $1$ & $1$&$1$&$0$ & $1$ & $0$&$1$&$1$ & $1$ & $\sqrt{2}$&$\sqrt{2}$&$1$ & $1$ & $0$&$\sfrac{1}{2}$&$\sfrac{\sqrt{3}}{2}$ & $1$\\
        \end{tabular}
      \end{adjustbox}
      \label{table:Geo_Params}
\end{table}
}

\bibliography{references}

\end{document}